\documentclass[11pt,reqno,a4paper]{amsart}

\usepackage{graphicx}
\usepackage{fullpage}
\usepackage[all]{xy}
\usepackage{tikz}
\usepackage{tikz-cd}
\usetikzlibrary{decorations.pathreplacing}
\usetikzlibrary{arrows}
\usetikzlibrary{automata,positioning}

\definecolor{grey}{rgb}{0.5,0.5,0.5}

\usepackage[math]{cellspace}
\usepackage{calrsfs}
\usepackage[T1]{fontenc} 
\usepackage{textcomp}
\usepackage{times}
\usepackage[scaled=0.85]{helvet} 

\usepackage{amsfonts}

\usepackage{amssymb}
\usepackage{hyperref}
\usepackage{mathtools}
\usepackage{stmaryrd}
\usepackage{booktabs}

\usepackage{etoolbox}
\patchcmd{\section}{\scshape}{\bfseries}{}{}
\makeatletter
\renewcommand{\@secnumfont}{\bfseries}
\makeatother

\DeclareRobustCommand{\SkipTocEntry}[5]{}

\makeatletter 
 \def\l@subsection{\@tocline{2}{0pt}{2pc}{0pc}{}} \makeatother
 
 \makeatletter 
 \def\l@subsubsection{\@tocline{2}{0pt}{3pc}{3pc}{}} \makeatother

\theoremstyle{definition}

\theoremstyle{plain}
\newtheorem{theorem}{Theorem}[subsection]
\newtheorem{proposition}[theorem]{Proposition}
\newtheorem{lemma}[theorem]{Lemma}
\newtheorem{corollary}[theorem]{Corollary}
\theoremstyle{definition}
\newtheorem{definition}[theorem]{Definition}
\newtheorem*{notation}{Notation}
\newtheorem{algorithm}[theorem]{$\ulcorner$ Algorithm}
\newtheorem{example}[theorem]{Example}
\newtheorem{examplecontinued}[theorem]{Example (continued)}
\newtheorem{remark}[theorem]{Remark}

\renewcommand{\bar}{\overline}
\renewcommand{\hat}{\widehat}
\renewcommand{\geq}{\geqslant}
\renewcommand{\leq}{\leqslant}
\renewcommand{\ge}{\geqslant}

\DeclareMathOperator{\Z}{\mathbf{Z}}

\DeclareMathOperator{\No}{\mathcal{N}}

\DeclareMathOperator{\R}{\mathbf{R}}
\DeclareMathOperator{\C}{\mathbf{C}}
\DeclareMathOperator{\F}{\mathbf{F}}

\DeclareMathOperator{\End}{\mathrm{End}}
\DeclareMathOperator{\Aut}{\mathrm{Aut}}

\DeclareMathOperator{\ord}{\mathrm{ord}}

\newcommand{\fl}{\llbracket}
\newcommand{\fr}{\rrbracket}

\newcommand{\Car}{\Lambda}
\newcommand{\Gal}{\mathrm{Gal}}
\newcommand{\lau}[1]{(\!(#1)\!)} 
\newcommand{\pau}[1]{\fl #1 \fr}
\newcommand{\us}{\sigma}

\newlength{\dhatheight}
\newcommand{\doublehat}[1]{%
    \settoheight{\dhatheight}{\ensuremath{\hat{#1}}}%
    \addtolength{\dhatheight}{-0.35ex}%
    \hat{\vphantom{\rule{1pt}{\dhatheight}}%
    \smash{\hat{#1}}}}
    
\DeclareMathOperator{\QS}{\doublehat{S\, }}

\usepackage{array}
\newcolumntype{H}{>{\setbox0=\hbox\bgroup}c<{\egroup}@{}}

\usepackage{enumitem}

\renewenvironment{quote}
               {\list{}{\rightmargin0.5\leftmargin}%
                \item\relax}
               {\endlist}

\usepackage{amsmath,amssymb,amsfonts,amsthm}

\begin{document}

\date{\today\ (version 1.0)} 
\title{Automata and finite order elements in the Nottingham group}
\author[J.~Byszewski]{Jakub Byszewski}
\address{\normalfont Wydzia\l{} Matematyki i Informatyki Uniwersytetu Jagiello\'nskiego, ul.\ S.\ \L ojasiewicza 6,
30-348 Krak\'ow, Polska}
\email{jakub.byszewski@gmail.com}
\author[G.~Cornelissen]{Gunther Cornelissen}
\address{\normalfont Mathematisch Instituut, Universiteit Utrecht, Postbus 80.010, 3508 TA Utrecht, Nederland} 
\email{g.cornelissen@uu.nl}
\author[D.~Tijsma]{Djurre Tijsma}
\address{\normalfont Mathematisches Institut der Heinrich-Heine-Universit\"at, Universit\"atsstra{\ss}e 1, 40225 D\"usseldorf, Deutschland}
\email{tijsma@uni-duesseldorf.de}
\thanks{JB was supported by National Science Center, Poland under grant no.\ 2016/\-23/\-D/ST1/\-01124. DT was supported in part by the research training group \textit{GRK 2240: Algebro-geometric Methods in Algebra, Arithmetic and Topology}, funded by the DFG. We thank Jeroen Sijsling for advice on various computations, Jonathan Lubin for sharing his unpublished work on conjugacy classes in the Nottingham group, Andrew Bridy and Eric Rowland for many interesting discussions about implementations, and Jason Bell for some insightful discussions. We also thank Ragnar Groot Koerkamp for setting up a computer search for small automata.}


\begin{abstract} \noindent 
The Nottingham group at $2$ is the group of (formal) power series $t+a_2 t^2+ a_3 t^3+ \cdots$ in the variable $t$ with coefficients $a_i$ from the field with two elements, where the group operation is given by composition of power series. The depth of such a series is the largest $d\geq 1$ for which $a_2=\dots=a_d=0$. 

Only a handful of power series of finite order (forcedly a power of $2$) are explicitly known through a formula for their coefficients. We argue in this paper that it is advantageous to describe such series in closed computational form through automata, based on effective versions of proofs of Christol's theorem identifying algebraic and automatic series. 

Up to conjugation, there are only finitely many series $\sigma$ of order $2^n$ with fixed break sequence (i.e. the sequence of depths of $\sigma^{\circ 2^i}$). 
Starting from Witt vector or Carlitz module constructions, we give an explicit automaton-theoretic description of: (a) representatives up to conjugation for all series of order $4$ with break sequence $(1,m)$ for $m<10$; (b) representatives up to conjugation for all series of order $8$ with minimal break sequence $(1,3,11)$; and (c) an embedding of the Klein four-group into the Nottingham group at $2$. 

We study the complexity of the new examples from the algebro-geometric properties of the equations they satisfy. For this, we generalise the theory of sparseness of power series to a four-step hierarchy of complexity, for which we give both Galois-theoretic and combinatorial descriptions.  We identify where our different series fit into this hierarchy. We construct sparse representatives for the conjugacy class of elements of order two and depth $2^\mu \pm 1$    $(\mu \geq 1)$. 
Series with small state complexity can end up high in the hierarchy. This is true, for example, for a new automaton we found, representing a series of order $4$ with $5$ states (the minimal possible number for such a series). 
 \end{abstract}

\maketitle

{\makeatletter
\def\@@underline#1{#1}
\tableofcontents
\makeatother}

\section{Introduction} 
Suppose $\sigma(t)=t+a_2 t^2 + a_3 t^3 + a_4 t^4 +  \cdots \neq t$ is a formal power series in the variable $t$ with coefficients from the field $\F_2 = {\Z}/{2}{\Z}$ with two elements. Since $\sigma(t)=t+O(t^2)$, substituting $\sigma(t)$ into itself produces a power series $\sigma^{\circ 2}(t) = t + a_2 (a_3 + 1)t^4 + \cdots$, and one may iterate this process to arrive at $\sigma^{\circ N}(t):=\sigma(\sigma(\cdots \sigma(t)))$. (We will systematically write $\sigma^{\circ N}(t)$ for the $N$-fold composition, and $\sigma(t)^N$ for the $N$-th power of the power series $\sigma(t)$; so here, for example, $\sigma(t)^2=t^2+a_2 t^4 + \cdots$.) Our concern is \emph{the explicit description of $\sigma$ and $N$ for which $\sigma^{\circ N}(t)=t$} (this is only possible if $N$ is a power of $2$). Our goal is not to compute finitely many coefficients $a_i$ of such $\sigma(t)$, but rather to give a \emph{finite description} of the complete series. To accomplish this, one might search for explicit formulas for the general coefficient $a_i$ or for the set $E(\sigma):=\{ i \in \Z_{\geq 0} : a_i \neq 0\}$ of occurring exponents, and this has been done  in a few cases. In this paper, we will argue that  one may push the boundaries of what is currently feasible by describing the coefficients of the power series by means of a finite automaton (that such a description is possible was already pointed out in \cite[Rem.\ 1.5]{BCPS}). We will construct the automaton using symbolic computation, based on Christol's characterisation of algebraic power series by automata \cite{Christol,CKMR}. We wish to stress that an automaton is a perfectly deterministic finite description of the corresponding power series $\sigma(t)$, but that a very small automaton (i.e.\ with very few states) may correspond to a power series for which an elementary description of the set $E(\sigma)$ is very complex. If one is interested in just the computation of the $k$-th coefficient of the power series $\sigma(t)$, the automaton can be used to do this in time logarithmic in $k$. 

We will first review the mathematical relevance of this  problem. Then we describe existing results and explain our method. Since the same question makes sense for the finite field $\F_p$ with $p$ elements (where $p$ is prime, and then forcedly $N$ is a power of $p$), we will consider this more general problem in the theoretical parts of the paper. 

\subsection{Connections} Fixing a prime number $p$, the \emph{Nottingham group} $\No(\F_p)$ is the pro-$p$-Sylow subgroup of the group of ring automorphisms $\Aut(\F_p\fl t \fr)$ of the formal power series ring $\F_p\fl t \fr$ over the finite field $\F_p$, with composition as multiplication.  A ring endomorphism $\sigma$ of  $\F_p\fl t \fr$ is determined uniquely by the image $\sigma(t) \in t\F_p\fl t \fr$ of $t$, and $\No(\F_p)$ is identified with the group of power series $\us(t) \in \F_p\fl t \fr$ with $\us(t) = t+O(t^2)$ under composition. We write $\sigma \circ \tau$ for the result of substituting the series $\tau\in \No(\F_p)$ for the variable $t$ in $\sigma \in \No(\F_p)$. The Nottingham group arises in many areas:
\begin{itemize}[leftmargin=*]
\item In \emph{group theory}, as Ershov remarked in \cite{Ershov}, $\No(\F_p)$ is `an excellent test example for many questions or conjectures in profinite group theory that have been settled for Chevalley groups'. In that reference, he proved that for $p \geq 5$, $\No(\F_p)$ admits no open embedding into a topologically simple group. On the other hand, every countably based  pro-$p$ group  embeds into $\No(\F_p)$ (Camina \cite{Camina}; Jennings \cite{Jennings}); in particular, every finite $p$-group embeds into $\No(\F_p)$ (an older unpublished result of Leedham-Green and Weiss; see \cite[Thm.\ 3]{Camina}). 
\item In \emph{number theory}, the Nottingham group occurs naturally in the theory of wild ramification (as the group of wild automorphisms of $\F_p(\!(t)\!)$; see Fesenko \cite{Fesenko}). 
\item The previous point relates to \emph{algebraic geometry}, namely: if a group $G$ acts on a smooth projective curve $X$ over $\F_p$, then the stabiliser $G_x$ of a point $x \in X$ acts on the completion of the local ring $\mathcal{O}_{X,x}$. This completion is isomorphic to $\F_p\fl t \fr$, leading to an embedding of the wild ramification group $G_x^1$ (the $p$-Sylow subgroup of $G_x$) into  $\No(\F_p)$; one can, for example, study deformations of group actions on curves through deformations of this group homomorphism, much like deformations of linear group representations, e.g.\ of Galois groups, cf.\ \cite{Mazur}. 
\end{itemize}
The need for explicit representations of finite order elements in $\No(\F_p)$ has been articulated several times, both in group theory (\cite[p.\ 216]{CaminaSurvey}, \cite[\S 5.4]{Lubin}), as well as in deformation theory, where conclusive results about formal deformation spaces and/or lifting are only known when standard forms for the series are available \cite{BM, BC, CK, CM, BCK, Green}. 

Our results are also relevant for the \emph{theory of automata} (that it relies upon), in particular, issues of implementation of certain algorithms for solving algebraic equations (Section \ref{section2a}, e.g.\ \cite{LDG}), the enumeration of automata with specific properties (cf.\ Section \ref{algoNvertices}), and an extension of Cobham's theory of complexity of automata/regular languages (cf.\ Section \ref{aridsection}).

\subsection{Review of previous work} Klopsch has proven that every element of order $p$ in $\No(\F_p)$ is conjugate to  \begin{equation} \label{eqKm} {t}{/}\!{\sqrt[m]{1-ma t^m}}=t+at^{m+1}+\cdots \end{equation}  for some positive integer $m$ coprime to $p$ and $a \in \F^*_p$, and that these series are mutually not conjugate \cite{Klopsch}. The expression  (\ref{eqKm}) may be readily converted into a formula for the coefficients of the corresponding power series by applying the binomial expansion (see also the discussion in Example \ref{exklopschauto}). 

Jean \cite{Jean} and Lubin \cite{Lubin} indicated how to use formal groups and explicit local class field theory to describe elements of any order $p^n$ in $\No(\F_p)$, and iterative procedures for the calculation of the coefficients of such elements were described (compare \cite{Jeanthesis},  \cite{Kiselev}, \cite[\S 6]{BBK}). However, the only known formulas for elements of order $p^n$ for $n>1$ are for  $p^n=4$ in $\No(\F_2)$, given by Jean in \cite[Ch.\ 7]{Jeanthesis}, Chinburg and Symonds \cite{ChinburgSymonds}, and Scherr and Zieve (cf.\ \cite[Rem.\ 1.4]{BCPS}). The Chinburg--Symonds example represents the action of an automorphism of order $4$ on the local completed ring at zero of the supersingular elliptic curve over $\F_2$; compare also \cite[Sect.\ 1]{BCPS}, where it is argued that this is essentially the only example that can be constructed by such a method; more precisely, up to conjugation, it is the only `almost rational' example. The final section of \cite{Jeanthesis} contains another (implicit) way of describing a solution to the problem, this time by using the method of Mellin \cite{Mellin} to solve algebraic equations---in this case, a trinomial---using hypergeometric series (the historically not entirely accurate reference in loc.\ cit.\ is to a monograph by Belardinelli). 

The \emph{break sequence} of $\sigma \in \No(\F_p)$ of order $p^n$ is a refined invariant with the property that there are only finitely many conjugacy classes of elements of fixed order $p^n$ with a given break sequence. The method of Lubin \cite{Lubin} can in principle be used to count that number using results from local class field theory. There is an exact characterisation of possible break sequences \cite[Obs.\ 5]{Lubin}. We briefly recall the definitions. 

\begin{definition} The \emph{depth} of $\sigma=\sigma(t) \in \No(\F_p)$ is $d(\sigma):= \mathrm{ord}_t(\us(t)-t)-1$ (and $d(t)=\infty$), so if $\us(t)=t+a_kt^k+O(t^{k+1})$ with $a_k\neq 0$, then $d(\sigma)=k-1$. The \emph{lower break sequence} of an element $\sigma \in \No(\F_p)$ of finite order $p^n$ is defined as $\mathfrak b_\sigma = (b_i)_{i=0}^{n-1} = (d(\sigma^{\circ p^i}))_{i=0}^{n-1}.$ 
\end{definition} 
The data $\mathfrak b_\sigma$ correspond bijectively to the so-called \emph{upper break sequence} $\mathfrak b^\sigma =\langle b^{(i)} \rangle_{i=0}^{n-1}$ that 
we will not define; for our purposes, it suffices to quote from  \cite[Def.\ 4]{Lubin} the formula that converts between lower and upper break sequences, which in our case of the  cyclic group generated by $\sigma$ becomes
\begin{equation} \label{convert} b^{(0)}=b_0 \qquad \mbox{and}\qquad \ b^{(i)} = b^{(i-1)} + p^{-i}(b_i-b_{i-1}) \quad \mbox{for}\ i>0. 
\end{equation} 
We will always indicate lower sequences by $( \mbox{ } )$-brackets, and the corresponding  upper sequences by $\langle \mbox{ } \rangle$-brackets, and we will write $(b_i)=\langle b^{(i)} \rangle$  for corresponding lower and upper break sequences.  

\subsection{The method of construction} We will use the term \emph{$p$-automaton} to describe a finite directed multigraph (allowing loops, as well as multiple edges between vertices) for which:
\begin{itemize}
\item vertices are labelled by elements of $\F_p$ [`output alphabet $\F_p$'];
\item one vertex (the so-called \emph{start vertex}) is additionally marked `Start'; 
\item each vertex has exactly $p$ outgoing edges,  each labelled by a different element of  $\{0,1,\dots,p-1\}$;  [`input alphabet $\{0,1,\ldots, p-1\}$']
\item  there is a path in the automaton from the start vertex to any vertex [`accessibility']; 
\item an edge with label $0$ always connects two vertices with the same label [`leading zeros invariance'].
\end{itemize} 
In the general theory of automata, this is called a `leading zeros invariant $p$-DFAO (deterministic finite $p$-automaton with output) with output alphabet $\F_p$ and all states accessible'. Vertices are also called `states'. We omit the qualifier $p$ when it is clear from the context.

Such an automaton produces the so-called \emph{$p$-automatic sequence} $(a_k)_{k\geq 0}$, where $a_k$ is the label carried by the final vertex of the walk that starts at the start vertex and follows the edges according to the successive digits of $k$ in base $p$ (starting from the least significant digit, also called the `reverse/backwards reading convention', compare \cite[12.2]{AS}).  The sequence $(a_k)_{k \geq 0}$ gives rise to the corresponding formal power series $\sum a_k t^k$ over $\F_p$ in the variable $t$. Note that the `leading zeros invariance' property means that we can allow the base-$p$ expansion of $k$ to have any number of leading zeros without affecting the resulting sequence. Should an automaton contain inaccessible vertices, they may be removed together with all their connecting edges without changing the corresponding series. 

\begin{example} \label{exklopschauto}
We consider Klopsch's series $$\sigma_{\mathrm{K},3}:=t/\sqrt[3]{1+t^3}=\sum_{k\geq 0} a_{3k+1} t^{3k+1}=t+t^4+t^{13}+\dots\in\No(\F_2)$$ 
of order $2$ with lower break sequence $(3)$. The coefficients of this series can be described  explicitly: $a_{3k+1}$ is equal to the binomial coefficient $\binom{-1/3}{k}$ modulo $2$.  Writing $-1/3$ as a $2$-adic integer $-1/3=\sum_{k\geq 0} 4^k$, we get an infinite product representation  \begin{equation*} \label{infprod} \sigma_{\mathrm{K},3}= t \prod_{k \geq 0} (1+t^{3 \cdot 4^k}),
\end{equation*} which shows that $a_k=1$ if and only if the base-$4$ expansion of $k-1$ contains only the digits $0$ or $3$. An automaton corresponding to this series is depicted in Figure \ref{ka1}; one way to construct it is to solve the algebraic equation $(t^3+1) \sigma^3=t^3$ with initial coefficients $\sigma=t+t^4+O(t^5)$ using one of the algorithms in Section \ref{section2} below. 

To illustrate our reading conventions, we compute the  coefficient $a_{13}$ of the corresponding power series: write $13=1\cdot 2^3+1 \cdot 2^2+0 \cdot 2^1 + 1 \cdot 2^0$ in base $2$ as $1101$; begin at the start vertex and follow the directed edges with respective labels $1,0,1,1$; we end up in a vertex with label $1$, so $a_{13}=1$. (If one adds leading zeros, e.g.\  by writing $13=0 \cdot 2^4 + 1\cdot 2^3+1 \cdot 2^2+0 \cdot 2^1 + 1 \cdot 2^0$, the result is the same even though the final vertex might be different.)  

\begin{figure} 
\begin{tikzpicture}[scale=0.15]
\tikzstyle{every node}+=[inner sep=0pt]
\draw [black] (19.4,-20.9) circle (3);
\draw (19.4,-20.9) node {$0$};
\draw [black] (19.4,-34.5) circle (3);
\draw (19.4,-34.5) node {$0$};
\draw [black] (32.9,-34.5) circle (3);
\draw (32.9,-34.5) node {$0$};
\draw [black] (32.9,-20.9) circle (3);
\draw (32.9,-20.9) node {$1$};
\draw [black] (47.8,-20.9) circle (3);
\draw (47.8,-20.9) node {$1$};
\draw [black] (47.8,-34.5) circle (3);
\draw (47.8,-34.5) node {$0$};
\draw [black] (11.3,-15.4) -- (16.92,-19.21);
\draw (10.7,-14.05) node [left] {$\text{{\small Start}}$};
\fill [black] (16.92,-19.21) -- (16.54,-18.35) -- (15.98,-19.18);
\draw [black] (22.4,-20.9) -- (29.9,-20.9);
\fill [black] (29.9,-20.9) -- (29.1,-20.4) -- (29.1,-21.4);
\draw (26.15,-21.4) node [below] {$1$};
\draw [black] (32.9,-23.9) -- (32.9,-31.5);
\fill [black] (32.9,-31.5) -- (33.4,-30.7) -- (32.4,-30.7);
\draw (32.4,-27.7) node [left] {$1$};
\draw [black] (22.4,-34.5) -- (29.9,-34.5);
\fill [black] (29.9,-34.5) -- (29.1,-34) -- (29.1,-35);
\draw (26.15,-35) node [below] {$1$};
\draw [black] (17.404,-32.28) arc (-147.99155:-212.00845:8.641);
\fill [black] (17.4,-32.28) -- (17.4,-31.34) -- (16.56,-31.87);
\draw (15.59,-27.7) node [left] {$0$};
\draw [black] (21.233,-23.258) arc (28.59943:-28.59943:9.279);
\fill [black] (21.23,-23.26) -- (21.18,-24.2) -- (22.06,-23.72);
\draw (22.87,-27.7) node [right] {$0$};
\draw [black] (35.122,-18.903) arc (122.98874:57.01126:9.601);
\fill [black] (45.58,-18.9) -- (45.18,-18.05) -- (44.63,-18.89);
\draw (40.35,-16.86) node [above] {$0$};
\draw [black] (45.467,-22.769) arc (-59.75662:-120.24338:10.16);
\fill [black] (35.23,-22.77) -- (35.67,-23.6) -- (36.18,-22.74);
\draw (40.35,-24.65) node [below] {$0$};
\draw [black] (45.852,-32.237) arc (-149.02444:-210.97556:8.816);
\fill [black] (45.85,-23.16) -- (45.01,-23.59) -- (45.87,-24.11);
\draw (44.09,-27.7) node [left] {$1$};
\draw [black] (49.699,-23.204) arc (29.95534:-29.95534:9.004);
\fill [black] (49.7,-32.2) -- (50.53,-31.75) -- (49.67,-31.25);
\draw (51.4,-27.7) node [right] {$1$};
\draw [black] (44.8,-34.5) -- (35.9,-34.5);
\fill [black] (35.9,-34.5) -- (36.7,-35) -- (36.7,-34);
\draw (40.35,-34) node [above] {$0$};
\draw [black] (34.223,-37.18) arc (54:-234:2.25);
\draw (37.8,-38.3) node [below] {$0,1$};
\fill [black] (31.58,-37.18) -- (30.7,-37.53) -- (31.51,-38.12);
\end{tikzpicture} 
\caption{A 2-automaton representing Klopsch's series $\sigma_{\mathrm{K},3}\in\No(\F_2)$ of order $2$ with lower break sequence $(3)$.} 
\label{ka1}
\end{figure}
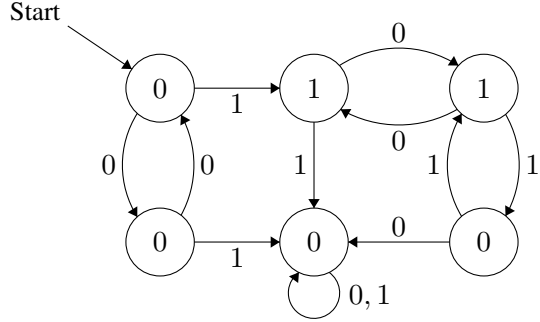
\end{example} 

Our construction of elements of order $p^n$ in $\No(\F_p)$ proceeds as follows: 
\begin{enumerate}
\item\label{enumwitt1}  Use Witt vectors to construct a cyclic Galois extension of order $p^n$ of the field of Laurent series $\F_p\lau{z}$ with certain ramification behaviour (this is similar to the method employed by Leedham-Green and Weiss, see \cite[Thm.\ 3]{Camina}; for a discussion using class field theoretic methods instead, see Remark \ref{remcft}). This field extension is described in terms of a finite set of generators $\alpha_i$ satisfying a set of explicit algebraic relations over $\F_p\lau{z}$ and with explicit formulas 
for the action of a generator $\sigma$ of the Galois group on the variables $\alpha_i$. Moreover, one can choose this field extension in such a way that $\alpha_i$ are algebraic over the field of rational functions $\F_p(z)$, so all computation involve algebraic functions only (cf.\ Examples \ref{re} \& \ref{re2}). 
\item Choose a rational function in the variables $\alpha_i$ that is a uniformiser for the field extension, say $t$. One can consider $\sigma$ as an automorphism of $\F_p\lau{t}$, and one has an explicit expression for $\us(t)$ as a rational function of the variables $\alpha_i$. This leads to a set of algebraic equations involving $\us(t), t$ and $\alpha_i$ (note that `algebraic' is w.r.t.\ the usual addition and multiplication of power series, not composition). By elimination of the variables $\alpha_i$ from those equations (in general with the help of a Groebner basis algorithm), one finds an explicit equation $F(t,X)=0$ for $\us=\us(t)$ over the field $\F_p(t)$.
\item\label{enumwitt3} Use an algorithmic version of a proof of Christol's theorem (based on using Ore polynomials, Furstenberg's diagonal method, or differential forms on algebraic curves) to find automata whose series correspond to the solutions of  the equation $F(t,X)=0$ in $\F_p\fl t \fr$. By Hensel's Lemma, sufficiently many initial coefficients of a solution will determine such a solution uniquely, so different solutions can be distinguished by solving iteratively for enough coefficients of a putative power series solution. 
\item\label{enumwitt4}  The equation found in \eqref{enumwitt3} might have several solutions, and at least one of these solutions is a power series of order $p^n$. Identify the solution(s) that correspond to elements of order $p^n$. 
\end{enumerate}

We describe the steps in some detail in the next section. In the first two steps, there are many possible choices of extensions and uniformisers, and hence there are many possible algebraic equations. 
 The size of the resulting automaton depends heavily on the choices made in the first two steps of the method, and the minimal size of an automaton representing a power series can vary greatly in a conjugacy class (theoretical bounds depending on the equations can be found in Bridy \cite{Bridybounds}).  

 Once the equation is fixed, the third and fourth step in the construction (which replace the naive method of trying to solve the equation recursively for the coefficients of a putative power series solution) have been automated by Rowland (see \cite{Rowland} for the source code and \cite{Rowlandproc} for the description) and partly in \cite{LDG}; we have used these implementations to produce the automata. 

\subsection{Results} We start by describing the case of elements of order $4$. 

\begin{theorem}[{Cor.\ \ref{31class} \& Props.\ \ref{sigmaminprop}, \ref{propsigmamin}, \ref{15class}, \ref{19class}}] 
The following is a complete list representing all possible elements of order $4$ in $\No(\F_2)$ with break sequence $(1,m) = \langle 1, (m+1)/2 \rangle $ for all admissible values $m < 10$, up to conjugation in $\No(\F_2)$\textup{:} 
\begin{itemize}[leftmargin=*]
\item \underline{with break sequence $(1,3)=\langle 1,2 \rangle$}\textup{:}  two (previously known) series $\sigma_{\mathrm{CS}}$ and $\sigma^{\circ 3}_{\mathrm{CS}}$ given in Equations \textup{(\ref{cseq})} \& \textup{(\ref{cs3eq})}, with the corresponding automata displayed in Table \textup{\ref{sigmatau}}. The series $\sigma_{\mathrm{CS}}$ is conjugate in $\No(\F_2)$ to a new series $\sigma_{\mathrm{min}}$ described by the automaton in Figure \textup{\ref{exrunpic}}, which is the unique series of order $4$ described by a $2$-automaton with at most $5$ states.  
\item \underline{with break sequence $(1,5)= \langle 1,3 \rangle$}\textup{:}  a series $\sigma_{(1,5)}$ corresponding to the $13$-state automaton displayed in Figure \textup{\ref{d13}}. 
\item \underline{with break sequence $(1,9) = \langle 1, 5 \rangle$}\textup{:}  a series $\sigma_{(1,9)}$ with $110$-state automaton described in Table \textup{\ref{19fig}}.  
\end{itemize}  
\end{theorem}

In Section \ref{algoNvertices} we present an algorithm for finding, for fixed integers $N$ and $n$,  all minimal $2$-automata representing an element of finite order $2^n$ in $\No(\F_2)$ with at most $\ N$ states. 

For some of the automata it is possible to extract a manageable \emph{closed formula} for the power series. We will present eight such formulas for power series of order $4$ with minimal break sequence, of which five are new: $\sigma_{\mathrm J}^{\circ 3}$ displayed in Equations 
 (\ref{J3eq}) \&  (\ref{J3eq2}) and  $ \sigma_{\mathrm{T},1},\sigma_{\mathrm{T},2}, \sigma_{\mathrm{T},3}$ and $\sigma_{\mathrm{T},4}$ in Table \ref{s5-s8}.
Note that although it is easy to determine which of these are mutually conjugate, the conjugating power series itself may be hard to describe: as far as we are concerned, it may be transcendental over $\F_2(t)$, and we are not aware of any criteria that guarantee the existence of an algebraic conjugating power series (but cf.\ Remark \ref{Rem:sparserep2}).

For order $8$, we have the following result (for the notion of `minimal' break sequence, see Example \ref{exjap}).

\begin{theorem}[{Props.\ \ref{8class}, \ref{conj8} \& \ref{sigma81}}] Up to conjugation in $\No(\F_2)$, there are precisely $4$ elements $\sigma_8, \sigma_8^{\circ 3}, \sigma_{8,2}, \sigma_{8,2}^{\circ 3}$ of order $8$ with `minimal' break sequence $(1,3,11) = \langle 1,2,4 \rangle$ in $\No(\F_2)$, where $\sigma_8$ corresponds to the $320$-state automaton given in Table \textup{\ref{8data}} and \cite{Database}, and $\sigma_{8,2}$ corresponds to the $926$-state automaton described in \textup{\ref{subsecCarlitz}} and \cite{Database}.
\end{theorem}

Since every finite $2$-group embeds in $\No(\F_2)$, Klopsch asked for a description of an embedding of the Klein four-group $V = {\Z}/{2}{\Z} \times {\Z}/{2}{\Z}$ in $\No(\F_2)$. We have the following result. 

\begin{theorem}[{Props.\ \ref{lem:Kleinfourbr} \& \ref{propkleinfour}}]
 For every embedding of the Klein four-group $V$ in the Nottingham group $\No(\F_2)$, some nontrivial element of $V$ has depth at least $5$. Furthermore, 
the series $ \sigma_{V,1}$ and $\sigma_{V,2}$ corresponding to the automata depicted in Table \textup{\ref{Klein}} have break sequences $(1)$ and $(5)$ and exhibit an explicit embedding of two generators of the Klein four-group into $\No(\F_2)$.
 \end{theorem}

One notices in the examples that 
for fixed order and break sequence, 
some series with an explicit `easy' formula
are produced by a rather large automaton, 
while at the same time there exist series requiring fewer states for which an `easy' formula does not seem to exist. 
 We study this phenomenon in Section \ref{aridsection}, generalising the concept of \emph{sparseness}. Recall that a series $\sigma=\sum a_i t^i$  is in the class $S$ of sparse series if the number of nonzero coefficients $a_i$ with $i\leq N$ grows like a power of a logarithm of $N$. Klopsch's series $\sigma_{\mathrm{K},m}$ are not sparse, but at least for some values of $m$ their conjugacy class contains a sparse series.

\begin{theorem}[{Prop.\ \ref{KSprop}}] Any power series of order $2$ and depth $m=2^\mu\pm 1$, $\mu \geq 1$, is conjugate to a sparse power series $\sigma_{\mathrm{S},m}$ given in Equations \textup{(\ref{order2m1})}, \textup{(\ref{order2m})} \& \textup{(\ref{order2mplus})}, the first two of which correspond to the automata displayed in Table \textup{\ref{sparse2m}}. 
\end{theorem} 

We classify general series into three classes that we consider to have `easy formulas': 
$$ 
S \subset \hat{S} \subset \QS \subset  \F_2\pau{t},
$$
where $\hat S$ is the class of  series that are sparse up to multiplication with a rational function, and $\QS$ is the class of series that are in $\hat S$ up to composition with an automorphism of $\F_p(t)$. Whether or not a series is in a certain class can be studied both using Galois theory (Section \ref{sparsefield}) and combinatorics of automata (Section \ref{sparseauto}). Even for the `larger' automata with several hundred states, the combinatorial method can be automated relatively easily using the computer algebra representation (cf.\ Table \ref{proofnotin}). 
Among the series described above there occur examples at all levels of this hierarchy of complexity. 

\begin{theorem}[{Thm.\ \ref{enumarid} \& Table \ref{tableeqs}}] The series 
$\sigma_{\mathrm{T},1},\dots,\sigma_{\mathrm{T},4}, \us_{\mathrm{CS}}^{\circ 3}$ are in $S$\textup{;} the series $\us_{\mathrm{CS}},  \us_{\mathrm{CS}}^{\circ 2}$ are in $\hat{{S}}$ but not in $S$\textup{;} the series $ \us_{\mathrm{J}},  \us_{\mathrm{J}}^{\circ 3}$ are in $\QS$ but not in $\hat{{S}}$\textup{;} the series $\us_{\mathrm{K},m} (m\geq 3)$, $ \sigma_{V,1}$, $\sigma_{V,2}$, $\sigma_{V,3}$, $\us_{\mathrm{min}}$, $\us_{(1,5)}$, $\sigma_{(1,9)}$, $\us_8.$ are not in $\QS$.
\end{theorem}
 
Finally, in Section \ref{nonran} we briefly discuss the synchronisation properties of some of our automata, in relation to a `structured/random' decomposition of automatic sequences in \cite{BKM}. 

\subsection{Some open problems}
\begin{itemize}[leftmargin=*]
\item  We have provided one example of an embedding of a non-cyclic $p$-group (the Klein four-group $V$) into $\No(\F_p)$ (for $p=2$), with the break sequences of the nontrivial elements of $V$ being $(1), (1)$ and $(5)$. Study the possible break sequences for embeddings of  $V$ into $\No(\F_2)$, and more generally for embeddings of arbitrary finite $p$-groups into $\No(\F_p)$ (cf.\ Proposition \ref{lem:Kleinfourbr}).
\item Is there a sparse series of order $2$ with break sequence $(11)$? This is equivalent to asking whether Klopsch's series $t/\!\sqrt[11]{1+t^{11}} \in \No(\F_2)$ is conjugate to a sparse series. More generally, is every element of finite order in $\No(\F_2)$ sparse (or in $\hat S$ or $\QS$) up to conjugation? 
\item Provide an automaton-theoretic characterisation of series that are sparse up to multiplication with a rational function, in a manner analogous to how \cite{SYZS} gives a necessary and sufficient condition for a series to be sparse in terms of properties of a corresponding automaton.
\item As the automaton method allows us to extend the catalogue of known elements of finite order in $\No(\F_p)$, one may argue that it is advantageous to manipulate elements of finite order in $\No(\F_p)$ in their automatic form directly, ignoring any explicit form for the coefficients of the corresponding power series. Thus, it would make sense to study `$p$-automata of finite order' as a subject of its own.  How to characterise an automaton that represents a series of finite order? 
\item If it exists, describe an algorithm that finds all automata on at most $N$ states that represent series of finite order. For any \emph{given} finite order this is easy (see Section \ref{algoNvertices}), so an affirmative solution of this problem would most likely require finding a bound on the order of a series in terms of the number of states of an automaton that generates it.
\end{itemize}

\begin{notation}
We will use the notation $\sigma$ and $\sigma(t)$ for elements of $\No(\F_p)$ interchangeably, and also use $\sigma$ for the corresponding element of the Galois group of an extension of fields of formal Laurent series. We will also write `$\sigma(t)$' when $\sigma$ is considered as an element of a Galois group and $t$ is a specified uniformiser.  
\end{notation}

\section{Detailed method: finding an algebraic equation}\label{section2}

\subsection{Extensions of Laurent series fields and elements of $\No(\F_p)$} 
Let $k=\F_p\lau{z}$ be a field of formal Laurent series with corresponding valuation $v_z$, and let $K/k$ be a cyclic totally ramified Galois extension of degree $p^n$. Let $t$ be a uniformiser for $K$ with corresponding valuation $v_t$, so that  $K=\F_p\lau{t}$. Any $\sigma\in\Gal(K/k)$ is an automorphism of $\F_p\lau{t}$ fixing $\F_p\lau{z}$, and it automatically preserves the valuation.  It follows that
$\sigma(t)=a_1 t+a_2t^2+a_3t^3+\cdots$
for some $a_i\in\F_p$; since the order of $\sigma$ is a power of $p$, we have $a_1=1$, meaning that $\sigma$ is an element of $\No(\F_p)$. In this way, elements of order $p^n$ in $\No(\F_p)$ arise from totally ramified cyclic $p^n$-extensions of fields of Laurent series. 

We first explicitly describe cyclic $p^n$-extensions using Witt vectors and then discuss how to detect whether they are totally ramified. 
By Artin--Schreier theory any abelian extension $K/k$ of order $p^n$ can be decomposed as a tower of field extensions 
\begin{equation} \label{AStower} k=K_0\subsetneq K_1\subsetneq \cdots\subsetneq K_n=K 
\end{equation}  
with $K_{i+1}=K_i(\alpha_i)$ for $0\leq i\leq n-1$ and $K_{i+1}/K_i$  an Artin--Schreier extension with $\alpha_i^p-\alpha_i \in K_i$. 
In the opposite direction Witt vectors allow one to guarantee that such an iterative procedure produces a \emph{cyclic} extension $K/k$. 

Any $\sigma \in \No(\F_p)$ of order $p^n$ arises from such a construction: Harbater  \cite[\S 2]{Harbater} proved that every such $\sigma$ describes 
 the action of a generator of the Galois group on the completed local ring at a totally ramified point of a global ${\Z}/{p^n}{\Z}$-Galois cover of $\mathbf{P}^1$ having a unique ramification point. The choice of a uniformiser at the ramified point (i.e.\ the choice of an isomorphism of the completed local ring with $\F_p\pau{t}$) corresponds to a conjugation of the representing power series.  It follows that any $\us$ of order $p^n$  is conjugate to an algebraic power series; note that the conjugating power series is an element of $\No(\F_p)$, but is not necessarily algebraic over $\F_p(t)$. 
  
\begin{remark}
Harbater proved the result for perfect fields; it holds for arbitrary finite groups by the general theory of Harbater--Katz--Gabber covers \cite[1.4.1]{KG}, compare \cite[\S 4.3]{BCPS};   
for a cohomological characterisation of the occurring Galois covers, see \cite{Konto}. 
\end{remark}

\begin{remark} \label{remcft} 
There exist alternative methods for the explicit construction of  equations for the Galois extensions. One may use explicit local class field theory, using the theory of formal groups/moduli of Lubin and Tate \cite{LT}. An essentially equivalent global method is to use explicit global class field theory of function fields, employing torsion of the Carlitz module \cite{RS}, and then localising at a totally ramified place. 
This shows, at least theoretically, that the resulting series can be described by recursion relations or automata and immediately leads to a recursive algorithm to compute the coefficients of the power series.
 In Remark \ref{Carlitz} and Subsection \ref{subsecCarlitz}, we describe how to find series of order $4$ and $8$ in this way. In particular, we use this method to construct a complete set of representatives for all conjugacy classes of order $8$ elements with minimal break sequence. 
We have performed more experiments implementing these methods and observed that they tend to lead to automata with more states compared to the above method. A possible reason is that class field theory methods give Ore-style equations that in algorithms produce state spaces of size doubly exponential in the degree of the equation (cf.\ Subsection \ref{boundcomplexitychristol} below).
\end{remark}

\subsection{Witt vectors and construction of $p^n$-extensions}
Let $k$ be a field of characteristic $p>0$ and let $n \geq 1$ be an integer. Let $W_n(k)$ denote the ring of ($n$-truncated $p$-typical) Witt vectors over $k$. As a set $W_n(k)$ is equal to $k^n$, and we write its elements as vectors of length $n$. The zero and identity element of $W_n(k)$ are $0=(0,\ldots,0)$ and $1=(1,0,\ldots,0)$. Addition and multiplication of two elements $a,b\in W_n(k)$ are defined by polynomial expressions in the coordinates $a_0,\ldots,a_{n-1},b_0,\ldots,b_{n-1}$ of $a$ and $b$ (see e.g.\ Example \ref{re} and \ref{re2} below that we will use later). The ring $W_n(k)$ comes with a Frobenius endomorphism $\text{Frob}\colon W_n(k)\to W_n(k)$ mapping the element $(a_0,\ldots,a_{n-1})$ to $(a_0^p,\ldots,a_{n-1}^p)$. The map $\wp:=\text{Frob}-\text{Id}$ is an endomorphism of the underlying abelian group of $W_n(k)$.
Writing $k^{\mathrm{sep}}$ for a separable closure of $k$, for any given $\beta\in W_n(k)$ there exists some $\alpha\in W_n(k^{\mathrm{sep}})$ such that $\wp(\alpha)=\beta$. Such $\alpha$ is unique up to addition of an element of $\ker\wp=W_n(\F_p)$ and the extension $k(\wp^{-1}(\beta)):=k(\alpha_0,\ldots,\alpha_{n-1})$ of $k$ is independent of the choice of $\alpha$. Note that $W_1(k)$ is just the field $k$.

\begin{theorem}[{Witt; cf.\ \cite[p.\ 107, Thm.\ 5]{Lorenz}}]\label{ASW}
Let $k$ denote a field of characteristic $p>0$, let $k^{\mathrm{sep}}$ denote a separable closure of $k$, and let $n$ denote any positive integer. For any field $K$ with $k\subseteq K \subseteq k^{\mathrm{sep}}$, $K/k$ is a cyclic Galois extension of degree $p^n$ if and only if there exists a $\beta\in W_n(k)$ with $\beta_0\notin\wp(k)$ such that $K=k(\wp^{-1}(\beta))$. If $\alpha\in W_n(k^{\mathrm{sep}})$ satisfies $\wp(\alpha)=\beta$, then $k(\wp^{-1}(\beta))=k(\alpha_0,\ldots,\alpha_{n-1})$ and a generator $\sigma$ of the Galois group $\Gal(K/k)$ is determined by the equations 
\begin{equation} \label{galact} \sigma(\alpha_i)=(\alpha+1)_i,  \qquad i=0, \dots, n-1.\end{equation}
\end{theorem}

\begin{example} \label{re}  
We consider the ring of Witt vectors $W_2(k)$ of length two over a field $k$ of characteristic $2$. For $a=(a_0,a_1),b=(b_0,b_1)\in W_2(k)$ the formulas for addition and multiplication are
\begin{equation*}
a+b=(a_0+b_0,a_1+b_1+a_0b_0)\quad\text{ and }\quad a\cdot b=(a_0b_0,a_0^2b_1+a_1b_0^2), 
\end{equation*}
and the map $\wp$ is given by
$ \wp(a) 
=(a_0^2+a_0,a_1^2+a_1+a_0^2+a_0^3).
$ 
Observe that this implies that $-(a_0,a_1) = (a_0,a_1+a_0^2)$. 
According to Theorem \ref{ASW}, an extension $K/k$ is a cyclic Galois extension of degree 4 if and only if $K=k(\alpha_0,\alpha_1)$, where $\alpha_0,\alpha_1$ satisfy\def\arraystretch{1.1}
\begin{equation*}
\left\{ \begin{array}{l}
\alpha_0^2+\alpha_0 = \beta_0; \\
\alpha_1^2+\alpha_1 = \beta_1 + \beta_0 \alpha_0
   \end{array} \right.\def\arraystretch{1}
\end{equation*}
for some $\beta_0,\beta_1 \in k$ with $\beta_0$ not of the form $x^2+x$ for $x \in k$.
The Galois group of $K/k$ is generated by the field automorphism $\sigma$ defined on the generators $\alpha_0,\alpha_1$ by
\begin{equation} \label{w4} 
\left\{ \begin{array}{l} \sigma(\alpha_0)= \alpha_0+1; \\ \sigma(\alpha_1)= \alpha_1 + \alpha_0. \end{array} \right.
\end{equation}

\end{example}

\begin{example} \label{re2}  
We consider the ring of Witt vectors $W_3(k)$ of length three over a field $k$ of characteristic $2$. For $a=(a_0,a_1,a_2),b=(b_0,b_1,b_2)\in W_3(k)$ the formula for addition is 
\begin{equation*}
a+b=(a_0+b_0,a_1+b_1+a_0b_0, a_2+b_2+a_1 b_1 + a_0 a_1 b_0 + a_0 b_0 b_1 + a_0^3 b_0 + a_0 b_0^3)
\end{equation*}
and for multiplication is 
\begin{equation*}
a\cdot b=(a_0b_0,a_0^2b_1+a_1b_0^2, a_1^2 b_1^2 + a_0^4b_2 + a_2 b_0^4 + a_0^2 a_1 b_0^2 b_1). 
\end{equation*}
By Theorem \ref{ASW}, cyclic degree-$8$ extensions $K/k$ of a field $k$ of characteristic $2$  are of the form $K=k(\alpha_0,\alpha_1,\alpha_2)$, where 
\def\arraystretch{1.1}
\begin{equation} \label{re2eq}
\left\{ \begin{array}{l}
\alpha_0^2+\alpha_0 = \beta_0; \\
\alpha_1^2+\alpha_1 = \beta_1 + \beta_0 \alpha_0; \\
\alpha_2^2+\alpha_2 = \beta_2 + \alpha_1 \beta_1 +\alpha_0 \alpha_1 \beta_0 + \alpha_0 \beta_0 \beta_1 + \alpha_0^3 \beta_0 + \alpha_0 \beta_0^3,
   \end{array} \right.\def\arraystretch{1}
\end{equation}
with $\beta_0$ not of the form $x^2+x$ for $x \in k$. The Galois group of $K/k$ is generated by the field automorphism defined on the generators $\alpha_0,\alpha_1,\alpha_2$ by

\begin{equation} \label{w8} 
\left\{ \begin{array}{l} \sigma(\alpha_0)= \alpha_0+1; \\ \sigma(\alpha_1)= \alpha_1 + \alpha_0; \\ \sigma(\alpha_2)= \alpha_2 + \alpha_0 \alpha_1 + \alpha_0^3+\alpha_0. \end{array} \right.
\end{equation} 
\end{example}

\subsection{Ramification}
The ramification in an Artin--Schreier extension of $\F_p\lau{z}$ can be described using the following easy result (see e.g.\ \cite[III.(2.5)]{Fesenkobook}).

\begin{lemma} \label{ramlem} 
Let $k=\F_p\lau{z}$ and let $K=k(\alpha)$ be an extension of $k$ with $\alpha^p-\alpha = \gamma$ for some $\gamma \in k$. If $v_z(\gamma)$ is negative and not divisible by $p$, then  $K/k$ is a cyclic extension of degree $p$, and for any uniformiser $\pi$ of $K$ we have $v_{\pi}(\alpha) = v_z(\gamma)$\textup{;} for $x\in k$ we have $v_{\pi}(x) = p v_z(x)$. 
\end{lemma}

If we decompose a general cyclic totally ramified $p^n$-extension as a tower of Artin--Schreier extensions as in (\ref{AStower}) and we write $z_i$ for a uniformiser of $K_i$ (so $z_0=z$ and $z_n=t$), then $v_{z_{i+1}}(\alpha_i)=v_{z_i}(\alpha_i^p-\alpha_i)$ for $i=0,\dots,n-1$.

The general approach is now to take the following steps: 
\begin{enumerate} 
\item[(i)] Write down explicit equations for a cyclic $p^n$-extension in the variables $\alpha_i$ arising from the Witt construction, or other generators of the field (this may make equations simpler or help in applying Lemma \ref{ramlem} to check that the extension is totally ramified). 
\item[(ii)] Choose a unformiser $t$ as an algebraic function of the $\alpha_i$ (or the chosen field generators); using Lemma \ref{ramlem} allows us to control the valuations of rational functions in the field generators. 
\item[(iii)] Compute the action of a generator $\sigma$ of the Galois group on the uniformiser $t$ using the action in terms of Witt vectors given by Equation (\ref{galact}); this gives an equation for $\us(t)$ in terms of the $\alpha_i$ (or the chosen field generators).  
\end{enumerate}
These three steps lead to a set of algebraic equations from which one should eliminate all but $t$ and $\us(t)$, leading to an algebraic equation $F(t,X)=0$ with $F \in \F_p[t,X]$ satisfied by  $X=\us=\us(t)$. 
For elimination, one may use a Groebner basis algorithm (we used the implementation in \textsc{Singular} \cite{Singular}; in order to be able to eliminate all the variables it might be necessary to first make a primary decomposition of the ideal generated by the equations and extract a one-dimensional component). 

\begin{example} \label{exrun} We start describing what will be our `running example' for the next few sections, leading up to a particularly small (as it will turn out, the smallest possible one in terms of number of states) automaton for a series of order $4$ with `minimal' break sequence. 

Let $k=\F_2\lau{z}$, $\beta=(z^{-1},0)\in W_2(k)$, and write $\alpha=(x,y)\in W_2(k^{\mathrm{sep}})$ for a solution of $\wp(\alpha)=\beta$. Since $v_z(\wp(k))=2{\Z}\cup{\Z}_{\geq 0}$ we have $z^{-1}\notin\wp(k)$, and by Theorem \ref{ASW} the extension $K/k=\F_2\lau{z}(x,y)/\F_2\lau{z}$, with $x$ and $y$ satisfying
\begin{equation*}
\left\{ \begin{array}{l}
x^2+x=z^{-1}; \\
y^2+y=xz^{-1} = x^3+x^2, \\
   \end{array} \right.
\end{equation*}
is a cyclic Galois extension of degree 4. It is totally ramified; an example of a uniformiser $t$ for $K$  is given by $$t=(y+1)/(y+x^2).$$ 
Indeed, breaking up the extension into Artin--Schreier extensions as in Equation (\ref{AStower}), we have $$ k=K_0 = \F_2\lau{z_0} \subsetneq K_1 = K_0(x) = \F_2\lau{z_1} \subsetneq K_2 = K_1(y) = \F_2 \lau{z_2} = K $$
with $z_0=z, z_1, z_2$ uniformisers of the fields in the tower of extensions. So $v_{z_0}(z^{-1})=-1$, $v_{z_1}(x)=-1$ and $v_{z_1}(z)=2$. Hence $v_{z_1}(x^3+x^2)=-3$, so $K_1/K_0$ 
is totally ramified. Then $v_{z_2}(y)=-3$, $v_{z_2}(x)=-2$ and $v_{z_2}(z)=4$, so $K_2/K_1$ 
is also totally ramified. Hence $t$ is a uniformiser for $K$ since 
$$ v_{z_2}(t)=v_{z_2}(y+1)-v_{z_2}(y+x^2)=1.$$
Formula (\ref{w4}) shows that a generator $\sigma$ of the Galois group is determined by the equations
\begin{equation*}
\left\{ \begin{array}{l}
\sigma(x)=x+1; \\
\sigma(y)=y+x, \\
   \end{array} \right.
\end{equation*}
the other generator is given by $\tau=\sigma^{\circ 3}$. We compute $$\tau(t)=\sigma^{\circ 3}\left(\frac{y+1}{y+x^2}\right)=
\frac{y+x}{y+x^2+x}.$$
To find an algebraic equation for $\tau=\tau(t)$ over $\F_2(t)$, we need to eliminate $x$ and $y$ from the three equations 
\begin{equation*}
\left\{ \begin{array}{ll}
y^2+y=x^3+x^2 &  \mbox{[equation of extension]}; \\
(y+x^2)t=y+1 &  \mbox{[definition of uniformiser]}; \\ 
  (y+x^2+x)\tau(t)=y+x & \mbox{[action of $\tau$ on uniformiser]},
   \end{array} \right.
\end{equation*} 
from which we get that $X=\tau=\tau(t) \in \F_2\pau{t}$ satisfies the (irreducible) equation
\begin{equation} \label{exsmall}
F(t,X)=(t+1)^3X^3+(t^3+t)X^2+(t^3+t+1)X+t^3+t=0. 
\end{equation}
This equation has a unique solution of the form $t+O(t^2)$, as can be seen, e.g.\ from the corresponding $t$-adic Newton polygon;  its initial coefficients are given by $t+t^2+t^4+t^5+O(t^6)$.
\end{example} 

\subsection{Break sequence} By computing the first few coefficients of $\us \in \No(
\F_p)$ of order $p^n$ (using the algebraic equation for $\us$ over $\F_p(t)$), it is easy to determine the lower break sequence of $\sigma$. If one has an explicit upper bound for the number of inequivalent series with given break sequences, we can enumerate all classes of such series by `trying' enough equations, which sometimes works in practice. Such bounds are implicit in \cite[Theorem 2.2]{Lubin} and have been made explicit in a few cases (cf.\ the discussion in Sections \ref{section3} \& \ref{section5}). Alternatively, using explicit local class field theory constructions as in \cite{Lubin} we are guaranteed to obtain representatives of all the conjugacy classes.

A method of Kanesaka and Sekiguchi directly computes the upper break sequence in terms of the Witt vector data for a given extension of $k:=\F_p\lau{z}$ \cite[Thm.\ 5]{KS}, which we rephrase as follows.  

\begin{definition} Fix a positive integer $n$. Call a vector $a=(a_i) \in \bigoplus_{\mathbf N} W_n(\F_p)$
of Witt vectors of length $n$ (with finitely many nonzero entries) \emph{suitable} if $a_i=0$ for $p {\mid} i$ and for at least one $i$ we have $a_i \in W_n(\F_p)^*$ (i.e.\ the zero component of $a_i$ is not zero). 
If 
\begin{equation} \label{form}
\beta = (\beta_0,\ldots,\beta_{n-1}) := \sum_{i \geq 0} a_i (z^{-i},0,\dots,0)  + \wp(b) \in W_n(k)
\end{equation} 
for a suitable $a = (a_i)$ and any $b \in W_n(k)$,  define 
\begin{equation*} \label{rho} \rho_n(\beta):= p^{-1} \max\{ i \ord(a_i) : a_i \neq 0 \},
\end{equation*} where $\ord(a_i)$ is the order of $a_i$ in the additive group $W_n(\F_p)$ (that itself is of exponent $p^n$). This is well-defined, since one can show that if a vector $\beta$ admits such a representation, then the corresponding suitable vector is uniquely determined (since the vectors $(z^{-i},0,\dots,0)$ are independent modulo $\wp(W_n(k))$). Also note that $\rho_n(\beta)$ is independent of $b\in W_n(k)$.

Define, for $m \leq n$, the truncation map $ \lfloor (x_0,\dots,x_{n-1}) \rfloor_m := (x_0,\dots,x_{m-1}). $ The truncation of a vector of the form as in Equation (\ref{form}) in $W_n(k)$  is of that same form in $W_m(k)$. 
\end{definition}

\begin{proposition}[\cite{KS}] For $k=\F_p\lau{z}$ and a positive integer $n$, choose $\beta$ of the form as in Equation \textup{(\ref{form})}
for a suitable vector $a=(a_i)$, some $b \in W_n(k)$, and assume $\beta_0 \neq 0$. Then the extension $k(\wp^{-1}(\beta))/k$ is a totally ramified cyclic extension of degree $p^n$, and the upper break sequence of a generator of the corresponding Galois group is  
$\langle\rho_1(\lfloor \beta \rfloor_1),\ldots,
\rho_n(\lfloor \beta \rfloor_{n}) \rangle.$ 
\end{proposition}
Although in this paper, we usually use lower break sequences, the above result is most naturally formulated in terms of upper break sequences; as remarked before, these can be easily changed into each other using Formula (\ref{convert}). The above result allows one to fix not just $p^n$, but also the break sequence from the start, by choosing a suitable Witt vector $\beta\in W_n(k)$. Note that we get the same extension for every $b\in W_n(k)$, but it will be convenient to rewrite certain natural choices of $\beta$ using nonzero $b$. 

\begin{example} We give some examples of constructions with break sequences that we will use later.  \label{exjap}  \mbox { } 
\begin{itemize}
\item[(a)] Choose $\beta=(z^{-1},0,\dots,0) \in W_n(\F_p\lau{z})$ of length $n$, so all $a_i=0$ for $i \neq 1$ and $a_1=(1,0,\dots,0)$. Now $a_1$ is of order $p^n$ in $W_n(\F_p)$ and the break sequence, called the \emph{minimal} one,  is $$\left\langle p^i \right\rangle_{i=0}^{n-1}= \left(\frac{p^{2i+1}+1}{p+1}\right)_{i=0}^{n-1}.$$ 
\item[(b)] For $\beta=(z^{-1},z^{-pm})\in W_2(\F_p\lau{z})$, with $m>p$ coprime to $p$, rewrite $$\beta = a_1 (z^{-1},0) + a_m (z^{-m},0)$$ with $a_1 = (1,0)$ and $a_m=(0,1)$. Now $\ord(a_1) = p^2$ and $\ord(a_m) = p$ in $W_2(\F_p)$, so we find the upper break sequence $$\langle 1, m \rangle = (1,pm-p+1).$$ 
\item[(c)] For $\beta=(z^{-1},z^{-m})\in W_2(\F_p\lau{z})$ with $m>p$ coprime to $p$, we get the same break sequence as the previous example, since $$(z^{-1},z^{-m}) = (z^{-1},z^{-pm}) - \wp((0,z^{-m})).$$ \end{itemize}
\end{example}
 
 \section{Detailed method: Computing $p$-automata using proofs of Christol's theorem}\label{section2a}

\subsection{Abstract algorithm} The following theorem of Christol relates algebraic power series to $p$-auto\-ma\-tic sequences (see \cite{Christol,CKMR}): 
\begin{theorem}[Christol] \label{christol} 
A power series $\us = \sum_{k\geq 0}a_kt^k\in\F_p\pau{t}$ is algebraic over $\F_p(t)$ if and only if the sequence $(a_k)_{k\geq 0}$ is $p$-automatic.
\end{theorem}
For our applications it is important that there are constructive proofs of this theorem: given an algebraic equation $F(t,X)=0$ with $F(t,X) \in \F_p[t,X]$, the proofs can be turned into algorithms that compute $p$-automata representing the different solutions $X=\sigma \in \F_p\pau{t}$. These algorithms start from a finite $\F_p$-vector space $V$ with a distinguished nonzero vector $s_0 \in V$ and a set $\Lambda$ of `Cartier-style' operators $\Lambda_r \colon V \rightarrow V$ for $r \in \{0,\dots,p-1\}$. From these data, they produce the directed graph structure of an automaton.  
A finite computation (using Hensel's Lemma) then fills in the vertex labels for the different solutions. For three such proofs/algorithms, we briefly indicate the triples $(V, s_0, \Lambda)$ and point to other sources for proofs of correctness, optimised implementations and complexity analysis. 

It follows from the proofs that for a given irreducible equation all solutions can be represented by automata with the same directed graph structure (including edge labels, but excluding vertex labels). Hence the desired algorithm can be broken down into two parts: first, the computation of that directed graph, and second, computing the correct output labels corresponding to the different solutions.

We will make the following assumptions and use the following notations throughout: 
\begin{itemize} 
 \item $F(t,X) \in \F_p[t,X]$ is irreducible,
 \begin{itemize} 
 \item[$\diamond$] $d=\deg_X F$, 
 \item[$\diamond$] $h=\deg_t F$, 
 \item[$\diamond$] $m=\ord_t \mathrm{Res}_X\left(F(t,X),\frac{\partial F}{\partial X}(t,X)\right)$ denotes the $t$-valuation of the resultant of $F$ and its derivative in $X$, and 
 \item[$\diamond$] $g$ denotes the geometric genus of the normalisation $\mathcal X$ of the projective curve  corresponding to the plane affine curve $F(t,X)=0$. 
\end{itemize}
\item For $0\leq r<p$, the \emph{Cartier operator} $\mathcal C_r$ acting on  formal power series in $\F_p\pau{x_1,\dots,x_k}$ is defined by $$\mathcal C_r(\sum a_{i_1,i_2,
 \dots,i_k} x_1^{i_1} \cdots x_k^{i_k}) := \sum a_{pi_1+r,pi_2+r, \dots, pi_k+r} x_1^{i_1} \cdots x_k^{i_k}.$$
\end{itemize} 

For the first part---the construction of the directed graph underlying the automaton---the proofs are based on constructing a graph from the specific set of data $(V, s_0, \Lambda)$, as follows: 
 \begin{algorithm}[Labeled Directed Graph Structure] \mbox{ } \label{algoLDG} 
\begin{quote} \begin{enumerate}
 \item[{\footnotesize \tt Input}] 
 A finite $\F_p$-vector space $V$, $s_0 \in V$, 
 and maps $\Lambda=\{\Lambda_r\colon V \rightarrow V \mbox{ for } 0 \leq r < p\}$.
  \item[{\footnotesize \tt Output}] A finite directed graph with edge labels. 
  \end{enumerate} 
   \end{quote} 
Write $\Gamma$ for the monoid generated by the maps $\Lambda_r$ with $0\leq r<p$. Compute the set of vertices $S$ as the orbit of $s_0$ under the action of $\Gamma$ (by applying the maps $\Lambda_r$ until no new elements appear), let the vertex $s_0$ be labelled `Start', and put a directed edge between $s_1$ and $s_2$ with label $r$ precisely if $s_2=\Lambda_r(s_1)$. 
\hfill $\lrcorner$ 
\end{algorithm} 
 
 The second part can always be dealt with in the following way: 
 \begin{algorithm}[Vertex Labels] \mbox{ } \label{algo2} 
\begin{quote} \begin{enumerate}
 \item[{\footnotesize \tt Input}] A polynomial $F(t,X) \in \F_p[t,X]$ and the directed graph structure (including edge labels) of automata representing all solutions $X=\sigma \in \F_p\pau{t}$ of $F(t,X)=0$.
  \item[{\footnotesize \tt Output}] A finite list of automata corresponding to all these solutions.
  \end{enumerate} 
 \end{quote} 
  For an integer $i$, consider the truncated equation \begin{equation} \label{trunc} F(t,\sigma_0) = O(t^{i+1}) \mbox{ with } \sigma_0 = a_0 + a_1 t +a_2 t^2 +  \dots + a_{i} t^{i}. \end{equation} 
 \begin{enumerate} 
 \item Solve
the truncated Equation (\ref{trunc}) with $i=2m$ 
for all the (finitely many) possible $\sigma_0$. 
Hensel's Lemma implies that for each such $\sigma_0$ there is a unique solution $X=\sigma \in \F_p\pau{t}$ of $F(t,X)=0$ with $\sigma(t) = \sigma_0(t)+ O(t^{m+1})$  (see e.g.\ the introduction of  \cite{BCCD}). 
\item For each fixed $\sigma_0$, run through the  automaton following all base-$p$ expansions of the integers $j=0,1,2,\dots$ and give the final vertex of the walk corresponding to the base-$p$ expansion of $j$ the label $a_j$. For this, it may be necessary to compute  the coefficients $a_j$ of the solution of $F(t,X)=0$ corresponding to $\sigma_0$ for some $j>2m$, which can be done by solving the truncated equation inductively for $i=2m+1,\dots,j$, and use the leading zeros condition.  \hfill $\lrcorner$ 
\end{enumerate}
\end{algorithm} 
As we will indicate below, sometimes the vertex labels can be determined in a more efficient way, depending on the method used to compute the directed graph structure.  

\begin{center}
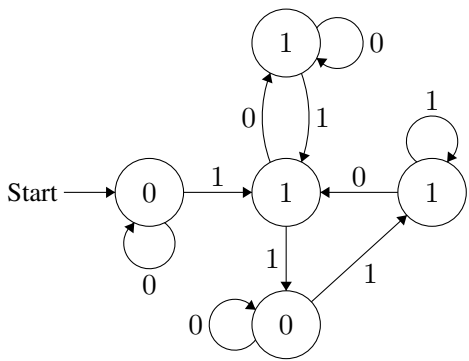
\begin{figure}
\begin{tikzpicture}[scale=0.15]
\tikzstyle{every node}+=[inner sep=0pt]
\draw [black] (24.3,-28.2) circle (3);
\draw (24.3,-28.2) node {$0$};
\draw [black] (36.3,-28.2) circle (3);
\draw (36.3,-28.2) node {$1$};
\draw [black] (36.3,-15) circle (3);
\draw (36.3,-15) node {$1$};
\draw [black] (49.1,-28.2) circle (3);
\draw (49.1,-28.2) node {$1$};
\draw [black] (36.3,-39.9) circle (3);
\draw (36.3,-39.9) node {$0$};
\draw [black] (25.623,-30.88) arc (54:-234:2.25);
\draw (24.3,-35.45) node [below] {$0$};
\fill [black] (22.98,-30.88) -- (22.1,-31.23) -- (22.91,-31.82);
\draw [black] (16.8,-28.2) -- (21.3,-28.2);
\draw (16.3,-28.2) node [left] {$\text{\small{Start}}$};
\fill [black] (21.3,-28.2) -- (20.5,-27.7) -- (20.5,-28.7);
\draw [black] (27.3,-28.2) -- (33.3,-28.2);
\fill [black] (33.3,-28.2) -- (32.5,-27.7) -- (32.5,-28.7);
\draw [black] (34.906,-25.553) arc (-159.74496:-200.25504:11.419);
\fill [black] (34.91,-17.65) -- (34.16,-18.22) -- (35.1,-18.57);
\draw (33.7,-21.6) node [left] {$0$};
\draw [black] (38.98,-13.677) arc (144:-144:2.25);
\draw (43.55,-15) node [right] {$0$};
\fill [black] (38.98,-16.32) -- (39.33,-17.2) -- (39.92,-16.39);
\draw [black] (37.639,-17.676) arc (19.32709:-19.32709:11.857);
\fill [black] (37.64,-25.52) -- (38.38,-24.93) -- (37.43,-24.6);
\draw (38.81,-21.6) node [right] {$1$};
\draw [black] (36.3,-31.2) -- (36.3,-36.9);
\fill [black] (36.3,-36.9) -- (36.8,-36.1) -- (35.8,-36.1);
\draw (35.8,-34.05) node [left] {$1$};
\draw [black] (38.51,-37.88) -- (46.89,-30.22);
\fill [black] (46.89,-30.22) -- (45.96,-30.39) -- (46.63,-31.13);
\draw (43.71,-34.54) node [below] {$1$};
\draw [black] (46.1,-28.2) -- (39.3,-28.2);
\fill [black] (39.3,-28.2) -- (40.1,-28.7) -- (40.1,-27.7);
\draw (42.7,-27.7) node [above] {$0$};
\draw [black] (47.777,-25.52) arc (234:-54:2.25);
\draw (49.1,-20.95) node [above] {$1$};
\fill [black] (50.42,-25.52) -- (51.3,-25.17) -- (50.49,-24.58);
\draw [black] (33.62,-41.223) arc (-36:-324:2.25);
\draw (29.05,-39.9) node [left] {$0$};
\fill [black] (33.62,-38.58) -- (33.27,-37.7) -- (32.68,-38.51);
\draw (30.3,-27.7) node [above] {$1$};
\end{tikzpicture}
\caption{A 2-automaton representing the element $\sigma_{\mathrm{min}}$ of $\No(\F_2)$ of order $4$ with lower break sequence $(1,3)$, corresponding to Equation \eqref{exsmall}.} 
\label{exrunpic}
\end{figure}
\end{center}

\begin{examplecontinued} \label{exrunvert} 
Suppose we know that the directed graph structure of the solutions for Example \ref{exrun} is as given in Figure \ref{exrunpic}, but the possible vertex labels are still unknown. In this case, we have $m=6$, and we are looking for a solution $\sigma$ with $
\sigma=t+O(t^2)$ (already known to exist). Substituting a tentative solution, we compute its initial coefficients: $\us_{\mathrm{min}} = t + t^2 + t^4 + t^5 + t^7 + O(t^8)$. Using the coefficients of $t^0, t^1, t^3, t^7$, the vertex labels are fixed uniquely, except for the label of the vertex reached from the start vertex by following the path $01$. However, the assumption of leading zeros invariance fixes this value to be the same as that of the vertex reached by following the path $1$. The resulting unique vertex labels are given in Figure \ref{exrunpic}. 
\end{examplecontinued}

\subsection{Three methods of constructing the input data}\label{input} What is different in various proofs/algorithms is the construction of $V, s_0$ and $\Lambda$ used as input for the construction of the directed graph. We briefly describe three possible approaches to this. 
 
  \subsubsection*[{\ref{input}.a Using spaces of differential forms}]{{\underline{\ref{input}.a Using spaces of differential forms}}} This method is based on a proof by David Speyer and Andrew Bridy \cite{Bridybounds}. The fact that the algorithm is correct is explained in \cite[\S 3]{Bridybounds}.  A plug-and-play implementation of this algorithm is not available at the current time, but the built-in algorithms for function fields in {\sc Magma} \cite{Magma} include K\"ahler differentials and Cartier operators, making it relatively easy to implement the computations (but not the visualisations). The file \cite{LDG} contains a description of a Magma routine that produces output that can be easily visualised in Mathematica and manipulated using \cite{Rowland}. 
  
Let $\Omega$ denote the $\F_p$-vector space of K\"ahler differentials on $\mathcal X$ and $K$ the function field of $\mathcal X$. Writing $\eta \in \Omega$ as $\eta = (u_0^p+u_1^p t + \dots + u_{p-1}^p t^{p-1}) dt$ for unique $u_i \in K$, define the Cartier operator $\mathcal{C}\colon \Omega\to \Omega$ by the formula $\mathcal C(\eta) := u_{p-1} dt$. 
Set $\omega:=Xdt \in \Omega$ and define the effective divisor $D:=(\omega)_\infty+(t)_\infty$, the sum of polar divisors of the differential $\omega$ and the function $t$. 
In this case:
\begin{itemize} 
\item $V = \Omega(D)$ is the $\F_p$-vector space of differential forms on $\mathcal X$ with divisor $\geq -D$ (of finite dimension $\leq h+3d+g-1$ over $\F_p$ by Riemann--Roch, see \cite[proof of Cor.\ 3.10]{Bridybounds}).
\item $s_0=\omega$.
\item For any $r=0,\dots,p-1$, define $\Lambda_r$ as $\Lambda_r(\eta):= \mathcal C (t^{p-1-r} \eta).$ The maps $\Lambda_r$ map $V$ to itself (see \cite[proof of Cor.\ 3.10]{Bridybounds}).
\end{itemize}

\begin{examplecontinued} \label{excont1}
Continuing the previous Example \ref{exrun}, we find (using  {\sc Magma}) that the curve corresponding to Equation (\ref{exsmall}) is of genus $g=1$, 
the space $\Omega((Xdt)_\infty+(t)_\infty)$ is of dimension $8$ and the subset $S=\Gamma(Xdt)$ has $5$ elements corresponding to the vertices in the automaton. Representing these by the vectors 
{\footnotesize  $$S = \{ (1, 1, 0, 1, 0, 0, 1, 0),
    (1, 0, 0, 0, 0, 0, 0, 0),
    (0, 1, 0, 0, 0, 0, 1, 0),
    (1, 1, 0, 1, 0, 0, 0, 0),
    (1, 0, 0, 0, 0, 1, 0, 0)\}, $$}where the third vector is the start vertex, the action of the operators $\Lambda_0$ and $\Lambda_1$ is given by right multiplication with the following explicit $8 \times 8$ matrices over $\F_2$:  
{\footnotesize $$    
\Lambda_0 = \left(
\begin{array}{cccccccc}
 1 & 0 & 0 & 0 & 0 & 0 & 0 & 0 \\
 0 & 1 & 0 & 0 & 0 & 0 & 1 & 0 \\
 0 & 1 & 0 & 0 & 0 & 0 & 0 & 0 \\
 0 & 0 & 0 & 1 & 0 & 0 & 1 & 0 \\
 0 & 1 & 1 & 0 & 1 & 0 & 1 & 1 \\
 0 & 1 & 0 & 1 & 0 & 0 & 1 & 0 \\
 0 & 0 & 0 & 0 & 0 & 0 & 0 & 0 \\
 0 & 0 & 1 & 0 & 1 & 0 & 0 & 1 \\
\end{array}
\right) \mbox{ and }
\Lambda_1 = \left(
\begin{array}{cccccccc}
 1 & 0 & 0 & 0 & 0 & 1 & 0 & 0 \\
 1 & 0 & 0 & 0 & 0 & 0 & 0 & 0 \\
 0 & 1 & 0 & 0 & 0 & 0 & 1 & 0 \\
 1 & 1 & 0 & 1 & 0 & 1 & 1 & 0 \\
 0 & 0 & 0 & 1 & 0 & 0 & 1 & 0 \\
 0 & 0 & 0 & 0 & 0 & 0 & 0 & 0 \\
 0 & 1 & 0 & 1 & 0 & 0 & 1 & 0 \\
 0 & 0 & 0 & 0 & 0 & 0 & 0 & 0 \\
\end{array}
\right).
$$
}The resulting automaton is the one in Figure \ref{exrunpic}.
\end{examplecontinued}

  \subsubsection*[{\ref{input}.b Using equations in Ore form}]{\underline{\ref{input}.b Using equations in Ore form}} This method is based on the proof from \cite{CKMR}. The fact that the algorithm is correct follows, e.g.\ from tracing through the proof of Christol's theorem in \cite[Thm.\ 12.2.5]{AS} using \cite[12.2.4]{AS} for the expression for the corresponding $p$-kernel and the construction of the automaton corresponding to such a kernel as in the proof of the equivalence of `$p$-automatic' and `finite $p$-kernel', see e.g.\ \cite[Thm.\ 6.6.2]{AS}. (The vector space described there is slightly larger, but the arguments show that the space defined below also works.) An implementation is described in \cite[Rem.\ 4.7]{RYBordeaux} and an actual implementation was done by Rowland in \cite{Rowland} (compare \cite{Rowlandproc}).
 
 One first computes a new polynomial $G(t,X) \in \F_p[t,X]$ in `Ore form', i.e.\ $G(t,X) = \sum_{i=0}^d B_i X^{p^i}$ with $B_i \in \F_p[t], B_0 \neq 0$, whose solution set in $X$ contains the $\F_p$-vector space spanned by the solution set of $F$ in $X$. 
Then the data are defined as follows: 
\begin{itemize}
\item $V$ is the set of linear combinations of elements from $\{X, X^p,\ldots,X^{p^{d-1}}\}$ with coefficients being elements from $\F_p[t]$ of degree at most $$N:=\max(\deg B_0, \max\{\left\lceil \frac{\deg B_i + (p^i-2) \deg B_0}{p-1}\right\rceil -1 \mid 1\leq i \leq d\}).$$ 
\item $s_0:=B_0 X$.
\item For $0\leq r<p$ and $D_k \in \F_p[t]$ of degree at most $ N$, define    
\begin{equation*}
 \Car_r\left(\sum_{k=0}^{d-1}D_kX^{p^k}\right):=\sum_{k=1}^{d-1}\mathcal C_r(D_k-D_0B_k B_0^{p^k-2})X^{p^{k-1}}-\mathcal C_r(D_0B_d B_0^{p^d-2})X^{p^{d-1}}. \label{caroperator} \end{equation*} The bound $N$ on the degrees of $D_k$ is chosen so that $s_0$ belongs to $V$ and  the operators $\Car_r$ map $V$ to itself (for this, note that for a polynomial $D\in \F_p[t]$ we have $\deg \mathcal C_r(D) \leq \lfloor \frac{\deg D}{p}\rfloor$).
\end{itemize}

One may circumvent the use of Algorithm \ref{algo2}: for the solution $\sigma_0$ whose truncation was fixed in (\ref{trunc}) (with $\ell:=\ord_t B_0\geq 1$) we can directly compute the labels of the vertices, as follows. Write 
\begin{equation}
\frac{\sigma_0}{B_0}=b_1 t^{-(\ell-1)}+b_2 t^{-(\ell-2)}+\cdots+b_{\ell-1}t^{-1}+b_\ell+O(t)
\end{equation}
with $b_i\in\F_p$; then the vertex corresponding to $\sum\limits_{k=0}^{d-1}D_k X^{p^k}\in V$, where $D_k = \sum_{j \geq 0} [D_k]_j t^j$ with $[D_k]_j\in \F_p$, has vertex label equal to $\sum\limits_{k=0}^{d-1}\, \sum\limits_{\substack{0\leq i\leq N\\ p^k{\mid}i}}\, [D_k]_i \cdot b_{\ell-i/p^k}$.

\begin{examplecontinued} \label{excont2}
The series $\tau$ from the previous Example \ref{exrun} satisfies the following equation in Ore form: 
$$G(t,X)=(t^8+1)X^8+(t^8+t^4+t^2+1)X^4+(t^7+t^6+t^5+t^4+t^2)X ^2+(t^7+t^5)X =0.$$ Now $\dim V = 150$ and $S$ consists of the following five elements, resulting in the automaton in Figure \ref{exrunpic}: 
{\footnotesize\begin{align*} s_0 &= (t^7+t^5)X, \\ 
s_1&=(t^6+t^3)X+(t^{14}+t^{13}+t^{11}+t^{10}+t^9+t^7)X^2+(t^{28}+t^{27}+t^{26}+t^{25}+t^{20}+t^{19}+t^{18}+t^{17})X^4, \\ 
s_2 &= (t^7+t^6+t^5)X+(t^{13}+t^{11}+t^{10}+t^8)X^2+(t^{28}+t^{26}+t^{20}+t^{18})X^4, \\ 
s_3 &=
t^2X+(t^{13}+t^8+t^7+t^6)X^2+(t^{26}+t^{24}+t^{18}+t^{16})X^4, \\
s_4 & =(t^6+t^4)X.
\end{align*}} 
\end{examplecontinued}

  \subsubsection*[{\ref{input}.c Using diagonals of two-variable power series}]{\underline{\ref{input}.c Using diagonals of two-variable power series}} This method splits the problem into two cases (`non-singu\-lar' and `general') and is based on a theorem of Furstenberg \cite[Prop.\ 2]{Furstenberg} in combination with the proof in \cite{Christol} and an observation in \cite{AdamBell}. In the special case, the algorithm is described in \cite[Algorithms 1 \& 2]{RYBordeaux}. The general algorithm is implemented in \cite{Rowland}. It is somewhat different from the preceding two methods: the non-singular case follows the setup considered before, in that it produces a triple $(V, s_0, \Lambda)$. The general case, however,  might produce a different automaton for every solution. 

 \underline{\emph{Special case}}. Suppose $G \in \F_p[t,X]$ is \emph{non-singular}, meaning that $G(0,0)=0$ and $c:=\partial G/\partial X(0,0)$ is nonzero. We search solutions $\sigma \in \F_p\fl t \fr$ of $G(t,\sigma)=0$ with $\sigma(0)=0$. In this case, by Hensel's lemma, there is a 
\emph{unique} such solution $\sigma$; Furstenberg's theorem says that 
$$ \sigma(t) = \Delta \left(\frac{P(t,X)}{Q(t,X)} \right) (t)\quad \text{with }   P(t,X):= c^{-1}X \frac{\partial G}{\partial X} (tX,X) \mbox{ and }Q(t,X):= c^{-1}X^{-1}G(tX,X), $$
where the \emph{diagonal} $\Delta G$ of a two-variable power series $G(t,X)=\sum a_{r,s} t^r X^s \in \F_p\fl t, X \fr$ is defined as the one-variable power series $(\Delta G)(t):=\sum a_{r,r} t^r\in \F_p\fl t \fr$. To avoid confusion: in the definition of $P$, the derivative is that of $G(t,X)$ w.r.t.\ $X$, after which the result is evaluated at $(tX,X)$, and the constant $c^{-1}$ is introduced so that $Q(0,0)=1$. 
The relevant data are: 

\begin{itemize}
\item $V$ is the space of polynomials in $\F_p[t,X]$ of degree at most  $\max(\deg_t P, \deg_t Q)$ in $t$ and of degree at most $\max(\deg_X P, \deg_X Q)$ in $X$.
\item $s_0:=P(t,X)$.
\item For $0\leq r<p$, $\Lambda_r(s):=\mathcal C_r(sQ^{p-1})$.
\end{itemize}

In this case, Algorithm \ref{algo2} may be avoided: $v \in V$ is a two-variable polynomial, and the corresponding (unique) vertex label is the value of this polynomial at $(0,0)$.

\underline{\emph{General case}}. 
Following \cite[\S 3.1]{AdamBell}, compute the finite list of all possible polynomials $q\in \F_p[t]$ of degree $\leq 2m$ such that  $F(t,q(t))=O(t^{2m+1})$. For each such $q$, set  $s=m+\ord_t (\frac{\partial F}{\partial X}(t,q(t)))$, $G(t,X)=t^{-s} F(t,t^m X+ q(t))$. 
Now $G$ is non-singular; apply the previous case to construct an automaton for the  (unique) power series solution $\tau(t)$ of $G(t,X)=0$ with $\tau(0)=0.$  Modify the automaton producing $\tau$ to an automaton producing a power series solution $\sigma=q+t^m\tau$ of $F(t,X)=0$ using standard constructions with automata (see e.g.\ \cite[Thm.\ 5.4.1 \& Cor.\ 6.8.5]{AS}, which have constructive proofs).

\begin{examplecontinued}\label{excont3}
For Example \ref{exrun}, the polynomial is non-singular and we have 
$$  \begin{array}{l} P(t,X) = t^3X^6+t^2X^5+(t^3+t)X^4+X^3+tX^2+X,\\ Q(t,X)= 
t^3 X^5+\left(t^3+t^2\right) X^4+\left(t^3+t\right) X^3+\left(t^3+t+1\right) X^2+t X+t+1. \end{array}  $$
The space $V$ is of dimension $28$ 
and $V$ consists of $6$ elements: 
{\footnotesize \begin{align*}  &\begin{array}{ll}
s_0 = P =  t^3 X^6+t^2 X^5+(t^3+t) X^4+X^3+t X^2+X, & s_1 = \Lambda_0(s_0) = t^3 X^5+(t^3+t) X^3+t X, \\ s_2 = \Lambda_1(s_0) = t^2 X^4+t^2 X^3+(t+1) X^2+t X+1, & s_3 = \Lambda_0(s_2) = t^2 X^4+X^2+1, \\ s_4 = \Lambda_1(s_2) = t^2 X^4+(t^2+t+1) X^2+X, & s_5 = \Lambda_1(s_4) = t^2 X^4+(t^2+1) X^2+1, 
 \end{array}\\
& \mbox{ with } \Lambda_0(s_1)=s_1, \Lambda_1(s_1) = s_2, \Lambda_0(s_3)=s_3, \Lambda_1(s_3) = s_2,  \Lambda_0(s_4)=s_4, \Lambda_0(s_5) = s_2,  \Lambda_1(s_5)=s_5.
\end{align*} }This leads to an automaton with $6$ states, but the states corresponding to $s_0$ and $s_1$ have the same outgoing edges and the same output labels, and hence can be merged into one state without affecting the automatic sequence produced by the automaton. Doing so leads  again to the automaton in Figure \ref{exrunpic}.
\end{examplecontinued}

\subsection{Bounds on the complexity} \label{boundcomplexitychristol} 
The exact complexity of the algorithms does not appear to be known, but upper bounds on the number of states $\# S$ have been given in terms of $d$ and $h$. In essentially all the known examples, these are obtained by first bounding the dimension of the vector space $V$, and then using the trivial inequality $\# S \leq p^{\dim V}$. In practice, it is often the case that the set $S$ is much smaller than the vector space $V$ (as seen, e.g.\ in Examples \ref{excont1}--\ref{excont3}). We will show in Proposition \ref{d=h} that $d=h$ for series of finite compositional order, and then we have the following upper bounds:
\begin{itemize}
\item Differential forms:  $\log_p \# S \leq 4d+g-1 \leq d(d+2) \approx d^2$ (\cite[Cor.\ 3.10]{Bridybounds} and the  inequality $g\leq (d-1)(h-1)$ of Castelnuovo--Riemann \cite[Cor.\ 3.11.4]{Stichtenoth}); 
\item Ore polynomials: $\log_p \# S \leq d^3p^d(p^{d}-1)/(p-1)\approx d^3 p^{2d-1}$
(using the upper bound $dhp^d$ for the height of the  Ore form equation from \cite[Lem.\ 8.1]{AB12}); 
\item Diagonals (non-singular case):   $\log_p (\# S-1) \leq d(d+1) \approx d^2$ (for this bound it is shown that all states in $S$ except possibly for $s_0$ lie in a vector subspace of $V$ of dimension $d(h+1)$ \cite[Rem.\ 4.7]{RYBordeaux}, \cite[Thm.\ 3.1]{AY}; the latter reference also contains an argument that shows that in the general case, the diagonal method gives a similar upper bound asymptotically in $d$ as the differential forms method).
\end{itemize}
In our running example, $\# S$ is $5$ or $6$, and the respective bounds on $\# S$ are $2^{12},2^{1512}$ and $2^{12}+1$. For more information on the exact complexity of our examples (that appear to require far fewer states than the theoretical general bounds), we refer to Section \ref{cxproperties}.  

\subsection{Our application} Our construction using Witt vectors produces a polynomial $F(t,X) \in \F_p[t,X]$ of which we first check irreducibility (if the polynomial were not irreducible, we would factor it and work with the factors). We know the polynomial has at least one solution $\sigma(t) = t+O(t^2) \in \F_p\pau{t}$, and we search only for such solutions. Most of the time, we can prove that there will be a unique solution of this form, and we then know that this $\sigma$ has the desired finite order under composition. In some cases, we find more than one solution, but in these cases, we can identify the correct series in a different way. 
For actual computations, we relied on implementations of all three algorithms; see the section `How computations and visualisations were done' at the end of the paper  for details.

The results obtained in our running example \ref{exrun}, \ref{exrunvert}, and either one of \ref{excont1}, \ref{excont2} or \ref{excont3}  may be summarised as follows: 

\begin{proposition} \label{sigmaminprop} 
The series $\sigma_{\mathrm{min}}$ corresponding to the automaton in Figure \textup{\ref{exrunpic}} is of order $4$ in $\No(\F_2)$ and has break sequence $(1,3)$ and initial coefficients $\us_{\mathrm{min}}= t+t^2+t^4+t^5+O(t^6)$. \qed
\end{proposition}

\section{An enumeration algorithm for automata on at most $N$ states representing finite order series}\label{algoNvertices}

\subsection{An abstract algorithm} Before we start applying our construction in concrete cases, we discuss an enumeration algorithm for finding all `small' (in terms of number of states) minimal automata representing an element in $\No(\F_2)$ of given finite order. The theoretical algorithm, which can readily be generalised to $p$-automata and order $p^n$ elements in $\No(\F_p)$, consists of two parts.

\begin{algorithm}[Compositional Power Automaton] \mbox{} \label{algoRprep} 
\begin{quote} \begin{enumerate}
 \item[{\footnotesize \tt Input}] A $2$-automaton $A$ and an integer $n \geq 0$.
 \item[{\footnotesize \tt Output}] If $\sigma$ denotes the series corresponding to $A$, a $2$-automaton $A_n$ corresponding to the series $\sigma^{\circ 2^n}$.
  \end{enumerate} 
 \end{quote} 

\begin{enumerate} 
\item Find a polynomial $F(t,X)\in\F_2[t,X]$ with $F(t,\sigma)=0$. This can be done by following the proof of Christol's Theorem \ref{christol} (in the direction different from the one used in Section \ref{section2a})---from the automaton, determine the $2$-kernel using \cite[Thm.\ 6.6.2]{AS} and then follow the first part of the proof in \cite[Thm.\ 12.2.5]{AS}.  
\item Composing with $\sigma(t)$ on the right gives  $F(\sigma(t),\sigma^{\circ 2}(t))=0$. Eliminate $Y$  from $F(t,Y)=F(Y,X)=0$ to produce an algebraic equation $F_1(t,X)=0$ satisfied by $X=\sigma^{\circ 2}$. Repeat this procedure to produce an algebraic equation $F_n(t,X)=0$ for $\sigma^{\circ 2^n}$. 
\item Construct an automaton $A_n$ for $\sigma^{\circ 2^n}$ from the equation $F_n(t,X)=0$, using the methods of Section \ref{section2a}.  \hfill $\lrcorner$ 
\end{enumerate}
\end{algorithm} 

We will use the well-known fact that to each automaton $A$ corresponds a unique minimal deterministic finite automaton $\hat A$ with the same corresponding series, and that $\hat A$ can be computed from $A$ by an algorithm, see e.g.\ \cite[\S 2.4]{Linz}. In particular, one can check by an algorithm whether or not two automata $A$ and $B$ correspond to the same series---this happens precisely when $\hat A = \hat B$. 

\begin{algorithm}[Enumeration Bounded Size Automata of Fixed Compositional Order] \mbox{} \label{algoR} 
\begin{quote} \begin{enumerate}
 \item[{\footnotesize \tt Input}] Integers $n\geq 0$ and $N\geq 1$.
 \item[{\footnotesize \tt Output}] A finite list of all minimal 2-automata on  at most $ N$ states representing an element of finite order $2^n$ in $\No(\F_2)$.
  \end{enumerate} 
 \end{quote} 

\begin{enumerate} 
\item Go over all 2-automata on  at most $ N$ states and eliminate those for which the corresponding power series is not of the form $\sigma=t+O(t^{2})$.
\item Remove duplicates from the list by comparing their minimal automata. 
\item For each remaining automaton $A$ use Algorithm \ref{algoRprep} to compute the automaton $A_n$. 
\item Compute the minimal automaton $\hat A_n$ corresponding to $A_n$ and check whether it equals the $3$-state minimal automaton generating the series $t$, depicted in Figure \ref{figt}. \hfill $\lrcorner$ 
\end{enumerate}
\end{algorithm} 
\begin{center} 
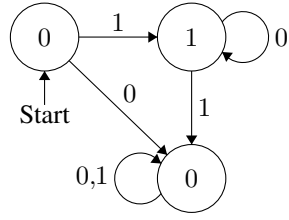
\begin{figure} 
\begin{tikzpicture}[scale=0.15, baseline=-47]
\tikzstyle{every node}+=[inner sep=0pt]
\draw [black] (27.3,-25.2) circle (3);
\draw (27.3,-25.2) node {{\small $0$}};
\draw [black] (40.2,-25.2) circle (3);
\draw (40.2,-25.2) node {{\small $1$}};
\draw [black] (40.2,-37.7) circle (3);
\draw (40.2,-37.7) node {{\small $0$}};
\draw [black] (27.3,-31.2) -- (27.3,-28.2);
\draw (27.3,-31.2) node [below] {{\small Start}};
\fill [black] (27.3,-28.2) -- (26.8,-29) -- (27.8,-29);
\draw [black] (30.3,-25.2) -- (37.2,-25.2);
\fill [black] (37.2,-25.2) -- (36.4,-24.7) -- (36.4,-25.7);
\draw (33.75,-24.7) node [above] {{\small $1$}};
\draw [black] (42.88,-23.877) arc (144:-144:2.25);
\draw (47.45,-25.2) node [right] {{\small $0$}};
\fill [black] (42.88,-26.52) -- (43.23,-27.4) -- (43.82,-26.59);
\draw [black] (37.52,-39.023) arc (-36:-324:2.25);
\draw (32.95,-37.7) node [left] {{\small $0,\hspace*{-0.8mm}1$}};
\fill [black] (37.52,-36.38) -- (37.17,-35.5) -- (36.58,-36.31);
\draw [black] (29.45,-27.29) -- (38.05,-35.61);
\fill [black] (38.05,-35.61) -- (37.82,-34.7) -- (37.12,-35.41);
\draw (34.77,-30.97) node [above] {{\small $0$}};
\draw [black] (40.2,-28.2) -- (40.2,-34.7);
\fill [black] (40.2,-34.7) -- (40.7,-33.9) -- (39.7,-33.9);
\draw (40.7,-31.45) node [right] {{\small $1$}};
\end{tikzpicture} 
\caption{Automaton for the power series $t$.}  \label{figt} 
\end{figure} 
\end{center} 

We do not know of an algorithm that lists all automata of size at most $ N$ corresponding to series of arbitrary but finite compositional order. 

\subsection{A practical implementation with application} A practical implementation of a  more optimal algorithm in \textup{{C}\texttt{++}} was given by Groot Koerkamp \cite{Ragnar} and produces a list of candidates for automata on at most $5$ states representing series of order $2$ and $4$. Running that algorithm, we find a unique candidate automaton corresponding to a series of order $4$. Since we already know from Proposition \ref{sigmaminprop} that $\sigma_{\mathrm{min}}$ is an order-$4$ series which is represented by an automaton with $5$ states, this proves the following. 

\begin{proposition}[{Groot Koerkamp, \cite{Ragnar}}] \label{propsigmamin} 
The \emph{unique} minimal (leading zeros invariant) $2$-automaton \emph{with at most $5$ states} representing a power series of compositional order $4$ is the one corresponding to the series $\sigma_{\mathrm{min}}$ and depicted in Figure \textup{\ref{exrunpic}}. \qed
\end{proposition}

\section{Construction and classification of some order-$4$ elements} 

\label{section3}

\subsection{Order $4$, break sequence $(1,3)=\langle 1,2 \rangle$} Below are two known explicit  power series with this order and break sequence: the one discovered by Chinburg and Symonds \cite{ChinburgSymonds} and its compositional inverse, computed by Scherr and Zieve \cite[Remark 1.4]{BCPS}: 
\begin{align} 
&\sigma_{\mathrm{CS}} := t+t^2 + \sum_{k \geq 0} \sum_{\ell=0}^{2^k-1} t^{6\cdot 2^k+2 \ell} = t+t^2+O(t^6);  \label{cseq} \\
&\sigma^{\circ 3}_{\mathrm{CS}} = \sum_{k \geq 0} \left( t^{3\cdot 2^k-2} + t^{4 \cdot 2^k-2}\right)  = t+t^2+t^4+O(t^6). \label{cs3eq}
\end{align} 
An unpublished result of Lubin (\cite{Lubinunpublished}, see \cite[Thm.\ 2.2]{HK} for a proof) implies that there are precisely two conjugacy classes of such elements in $\No(\F_2)$. We now present a slightly more detailed lemma that allows us to distinguish between these conjugacy classes based on the first few coefficients alone.
\begin{lemma} \label{recognise} Let $\sigma \in \No(\F_2)$ be an automorphism of order $4$ with break sequence $(1,3)=\langle 1,2 \rangle$, and write $\us=\sum_{i=1}^{\infty} a_i t^i$ with $a_i\in\F_2$. Then $a_1=a_2=1$, $a_3=0$, and exactly one of the following cases holds\textup{:}\begin{enumerate}  \item[\textup{(}a\textup{)}] $a_4=a_5$ and $\sigma$ is conjugate to $\sigma_{\mathrm{CS}}$\textup{;}\item[\textup{(}b\textup{)}] $a_4 \neq a_5$ and $\sigma$ is conjugate to $\sigma^{\circ 3}_{\mathrm{CS}}$.
\end{enumerate}
\end{lemma}
\begin{proof} We have $a_1=1$ since $\sigma\in\No(\F_2)$, and $a_2=1$, $a_3=0$ since $\sigma$ has 
lower break sequence $(1,3)$; for the latter statement, compute the power series  $\sigma^{\circ 2}=t+(1+a_3)t^4+O(t^5)$. The only possibilities for such series up to $O(t^6)$ are hence the four truncated series $\sigma = t + t^2 + a_4 t^4 + a_5 t^5 +O(t^6)$ with $a_4,a_5 \in \F_2$. Two of these correspond to (\ref{cseq}) and (\ref{cs3eq}), and for the other two, we observe that conjugating by $\phi:t\mapsto t+t^3$ gives 
\begin{align*} 
\phi^{-1} \circ \sigma_{\mathrm{CS}} \circ \phi &= t+t^2+t^4+t^5+ O(t^6);  \\
\phi^{-1} \circ \sigma^{\circ 3}_{\mathrm{CS}} \circ \phi &=  t+t^2+t^5+O(t^6).  
\end{align*} 
The quoted result of Lubin in \cite[Thm.\ 2.2]{HK} implies that there are precisely two conjugacy classes of power series with break sequence $(1,3)=\langle 1,2\rangle$. 
To finish the proof it is therefore enough to show that any automorphisms $\sigma, \tau \in \No(\F_2)$ with $$\sigma=t+t^2+O(t^6) \qquad \text{and} \qquad \tau=t+t^2+t^4+O(t^6)$$ are not conjugate in $\No(\F_2)$. Suppose this is the case, and let $\psi \in \No(\F_2)$ be such that $\sigma\circ \psi=\psi\circ \tau$. This implies that  
\begin{equation}\label{eqn:wiewiorka} 
\psi(t)+\psi(t)^2+O(t^6)=\psi(t+t^2+t^4)+O(t^6).
\end{equation} 
Writing $\psi(t)=t + \sum_{i=2}^{\infty}b_i t^i$ with $b_i\in \F_2$ and comparing the coefficients of $t^4$ and $t^5$ in \eqref{eqn:wiewiorka} gives $$b_2^2+b_4=1+b_2+b_3+b_4\qquad \text{and}\qquad b_5=b_3+b_5,$$ which gives a contradiction since $b_2\in\F_2$.
\end{proof}

\begin{corollary} \label{31class} 
The series $\us_{\mathrm{CS}}$ and $\us^{\circ 3}_{\mathrm{CS}}$ form a full set of representatives for the conjugacy classes of elements of order $4$ with break sequence $(1,3)=\langle 1,2 \rangle$ in $\No(\F_2)$. \qed
\end{corollary}

The following different power series of order $4$ and break sequence $(1,3)$ was found earlier by Jean in \cite{Jeanthesis} as a solution to the equation $(t+1) \us^2 + (t^2+1) \us + t = 0$: 
\begin{equation} 
 \sigma_{\mathrm{J}}  := \sum_{k\geq 0}\frac{t^{2^k}}{(t+1)^{3\cdot 2^k-1}}  
= t+t^2+t^5 +O(t^6). \label{Jeq} 
 \end{equation} 
Lemma \ref{recognise} implies that it is conjugate to $\sigma^{\circ 3}_{\mathrm{CS}}$. 

\begin{center}
\begin{table}[h]
\begin{tabular}{cc}
\hline \\
\begin{tikzpicture}[scale=0.14]
\tikzstyle{every node}+=[inner sep=0pt]
\draw [black] (25.8,-4.8) circle (3);
\draw (25.8,-4.8) node {$0$};
\draw [black] (39.7,-16.8) circle (3);
\draw (39.7,-16.8) node {$1$};
\draw [black] (39.7,-29.6) circle (3);
\draw (39.7,-29.6) node {$0$};
\draw [black] (25.8,-16.8) circle (3);
\draw (25.8,-16.8) node {$0$};
\draw [black] (11.6,-16.8) circle (3);
\draw (11.6,-16.8) node {$1$};
\draw [black] (11.6,-29.6) circle (3);
\draw (11.6,-29.6) node {$1$};
\draw [black] (25.8,-29.6) circle (3);
\draw (25.8,-29.6) node {$0$};
\draw [black] (39.7,-19.8) -- (39.7,-26.6);
\fill [black] (39.7,-26.6) -- (40.2,-25.8) -- (39.2,-25.8);
\draw (39.2,-23.2) node [left] {$1$};
\draw [black] (42.38,-28.277) arc (144:-144:2.25);
\draw (46.95,-29.6) node [right] {$0,\!1$};
\fill [black] (42.38,-30.92) -- (42.73,-31.8) -- (43.32,-30.99);
\draw [black] (28.07,-6.76) -- (37.43,-14.84);
\fill [black] (37.43,-14.84) -- (37.15,-13.94) -- (36.5,-14.7);
\draw (33.76,-10.31) node [above] {$1$};
\draw [black] (25.8,-7.8) -- (25.8,-13.8);
\fill [black] (25.8,-13.8) -- (26.3,-13) -- (25.3,-13);
\draw (26.3,-10.8) node [right] {$0$};
\draw [black] (22.8,-16.8) -- (14.6,-16.8);
\fill [black] (14.6,-16.8) -- (15.4,-17.3) -- (15.4,-16.3);
\draw (18.7,-16.3) node [above] {$1$};
\draw [black] (11.6,-19.8) -- (11.6,-26.6);
\fill [black] (11.6,-26.6) -- (12.1,-25.8) -- (11.1,-25.8);
\draw (11.1,-23.2) node [left] {$0$};
\draw [black] (13.83,-27.59) -- (23.57,-18.81);
\fill [black] (23.57,-18.81) -- (22.64,-18.97) -- (23.31,-19.72);
\draw (17.69,-22.71) node [above] {$1$};
\draw [black] (24.557,-26.878) arc (-162.48069:-197.51931:12.219);
\fill [black] (24.56,-26.88) -- (24.79,-25.96) -- (23.84,-26.27);
\draw (23.49,-23.2) node [left] {$0$};
\draw [black] (23.12,-30.923) arc (324:36:2.25);
\draw (18.55,-29.6) node [left] {$0$};
\fill [black] (23.12,-28.28) -- (22.77,-27.4) -- (22.18,-28.21);
\draw [black] (9.084,-31.213) arc (-29.60769:-317.60769:2.25);
\draw (4.24,-30.46) node [left] {$0$};
\fill [black] (8.79,-28.58) -- (8.34,-27.75) -- (7.85,-28.62);
\draw [black] (8.92,-18.123) arc (-36:-324:2.25);
\draw (4.35,-16.8) node [left] {$1$};
\fill [black] (8.92,-15.48) -- (8.57,-14.6) -- (7.98,-15.41);
\draw [black] (17.2,-4.8) -- (22.8,-4.8);
\draw (16.7,-4.8) node [left] {$\text{{\small Start}}$};
\fill [black] (22.8,-4.8) -- (22,-4.3) -- (22,-5.3);
\draw [black] (42.38,-15.477) arc (144:-144:2.25);
\draw (46.95,-16.8) node [right] {$0$};
\fill [black] (42.38,-18.12) -- (42.73,-19) -- (43.32,-18.19);
\draw [black] (26.779,-19.631) arc (13.46785:-13.46785:15.325);
\fill [black] (26.78,-19.63) -- (26.48,-20.53) -- (27.45,-20.29);
\draw (27.7,-23.2) node [right] {$1$};
\node at (26,-43) {Automaton for $\us_{\mathrm{CS}}$};
\end{tikzpicture}
&
\begin{tikzpicture}[scale=0.155]
\tikzstyle{every node}+=[inner sep=0pt]
\draw [black] (33.2,-32.2) circle (3);
\draw (33.2,-32.2) node {$0$};
\draw [black] (20,-23.9) circle (3);
\draw (20,-23.9) node {$1$};
\draw [black] (46.1,-23.9) circle (3);
\draw (46.1,-23.9) node {$0$};
\draw [black] (20,-4.2) circle (3);
\draw (20,-4.2) node {$1$};
\draw [black] (46.1,-4.2) circle (3);
\draw (46.1,-4.2) node {$1$};
\draw [black] (33.2,-23.9) circle (3);
\draw (33.2,-23.9) node {$0$};
\draw [black] (33.2,-11.2) circle (3);
\draw (33.2,-11.2) node {$0$};
\draw [black] (33.2,-38.2) -- (33.2,-35.2);
\draw (33.2,-38.7) node [below] {$\text{{\small Start}}$};
\fill [black] (33.2,-35.2) -- (32.7,-36) -- (33.7,-36);
\draw [black] (35.72,-30.58) -- (43.58,-25.52);
\fill [black] (43.58,-25.52) -- (42.63,-25.54) -- (43.17,-26.38);
\draw (40.65,-28.55) node [below] {$0$};
\draw [black] (30.66,-30.6) -- (22.54,-25.5);
\fill [black] (22.54,-25.5) -- (22.95,-26.35) -- (23.48,-25.5);
\draw (27.6,-27.55) node [above] {$1$};
\draw [black] (46.1,-20.9) -- (46.1,-7.2);
\fill [black] (46.1,-7.2) -- (45.6,-8) -- (46.6,-8);
\draw (46.6,-14.05) node [right] {$1$};
\draw [black] (43.1,-4.2) -- (23,-4.2);
\fill [black] (23,-4.2) -- (23.8,-4.7) -- (23.8,-3.7);
\draw (33.05,-3.7) node [above] {$0$};
\draw [black] (20,-7.2) -- (20,-20.9);
\fill [black] (20,-20.9) -- (20.5,-20.1) -- (19.5,-20.1);
\draw (19.5,-14.05) node [left] {$0,\!1$};
\draw [black] (48.78,-2.877) arc (144:-144:2.25);
\draw (53.35,-4.2) node [right] {$1$};
\fill [black] (48.78,-5.52) -- (49.13,-6.4) -- (49.72,-5.59);
\draw [black] (43.1,-23.9) -- (36.2,-23.9);
\fill [black] (36.2,-23.9) -- (37,-24.4) -- (37,-23.4);
\draw (39.65,-23.4) node [above] {$0$};
\draw [black] (30.2,-23.9) -- (23,-23.9);
\fill [black] (23,-23.9) -- (23.8,-24.4) -- (23.8,-23.4);
\draw (26.6,-23.4) node [above] {$1$};
\draw [black] (33.2,-20.9) -- (33.2,-14.2);
\fill [black] (33.2,-14.2) -- (32.7,-15) -- (33.7,-15);
\draw (33.7,-17.55) node [right] {$0$};
\draw [black] (22.16,-21.82) -- (31.04,-13.28);
\fill [black] (31.04,-13.28) -- (30.11,-13.47) -- (30.81,-14.19);
\draw (26.1,-15.03) node [below] {$1$};
\draw [black] (21.323,-26.58) arc (54:-234:2.25);
\draw (20,-31.15) node [below] {$0$};
\fill [black] (18.68,-26.58) -- (17.8,-26.93) -- (18.61,-27.52);
\draw [black] (35.88,-9.877) arc (144:-144:2.25);
\draw (40.45,-11.2) node [right] {$0,\!1$};
\fill [black] (35.88,-12.52) -- (36.23,-13.4) -- (36.82,-12.59);
\node at (35,-43) {Automaton for $\us_{\mathrm{CS}}^{\circ 3}$};
\end{tikzpicture}
\\
\hline \\
\begin{tikzpicture}[scale=0.16]
\tikzstyle{every node}+=[inner sep=0pt]
\draw [black] (41.1,-21) circle (3);
\draw (41.1,-21) node {$0$};
\draw [black] (29.6,-21) circle (3);
\draw (29.6,-21) node {$1$};
\draw [black] (51.9,-21) circle (3);
\draw (51.9,-21) node {$0$};
\draw [black] (29.6,-33.3) circle (3);
\draw (29.6,-33.3) node {$1$};
\draw [black] (51.9,-33.3) circle (3);
\draw (51.9,-33.3) node {$0$};
\draw [black] (18.1,-21) circle (3);
\draw (18.1,-21) node {$1$};
\draw [black] (63.2,-21) circle (3);
\draw (63.2,-21) node {$0$};
\draw [black] (18.1,-33.3) circle (3);
\draw (18.1,-33.3) node {$1$};
\draw [black] (63.2,-33.3) circle (3);
\draw (63.2,-33.3) node {$0$};
\draw [black] (41.1,-13.6) -- (41.1,-18);
\draw (41.1,-13.1) node [above] {$\text{{\small Start}}$};
\fill [black] (41.1,-18) -- (41.6,-17.2) -- (40.6,-17.2);
\draw [black] (38.1,-21) -- (32.6,-21);
\fill [black] (32.6,-21) -- (33.4,-21.5) -- (33.4,-20.5);
\draw (35.35,-20.5) node [above] {$1$};
\draw [black] (54.9,-21) -- (60.2,-21);
\fill [black] (60.2,-21) -- (59.4,-20.5) -- (59.4,-21.5);
\draw (57.55,-20.5) node [above] {$0$};
\draw [black] (26.6,-21) -- (21.1,-21);
\fill [black] (21.1,-21) -- (21.9,-21.5) -- (21.9,-20.5);
\draw (23.85,-20.5) node [above] {$0$};
\draw [black] (20.15,-23.19) -- (27.55,-31.11);
\fill [black] (27.55,-31.11) -- (27.37,-30.18) -- (26.64,-30.87);
\draw (24.38,-25.68) node [right] {$1$};
\draw [black] (61.17,-23.21) -- (53.93,-31.09);
\fill [black] (53.93,-31.09) -- (54.84,-30.84) -- (54.1,-30.16);
\draw (57.01,-25.69) node [left] {$1$};
\draw [black] (49.27,-22.45) -- (32.23,-31.85);
\fill [black] (32.23,-31.85) -- (33.17,-31.9) -- (32.69,-31.03);
\draw (34.75,-26.65) node [above] {$1$};
\draw [black] (32.23,-22.45) -- (49.27,-31.85);
\fill [black] (49.27,-31.85) -- (48.81,-31.03) -- (48.33,-31.9);
\draw (46.75,-26.65) node [above] {$1$};
\draw [black] (26.6,-33.3) -- (21.1,-33.3);
\fill [black] (21.1,-33.3) -- (21.9,-33.8) -- (21.9,-32.8);
\draw (23.85,-32.8) node [above] {$0$};
\draw [black] (54.9,-33.3) -- (60.2,-33.3);
\fill [black] (60.2,-33.3) -- (59.4,-32.8) -- (59.4,-33.8);
\draw (57.55,-32.8) node [above] {$0$};
\draw [black] (16.777,-18.32) arc (234:-54:2.25);
\draw (18.1,-13.75) node [above] {$0$};
\fill [black] (19.42,-18.32) -- (20.3,-17.97) -- (19.49,-17.38);
\draw [black] (61.877,-18.32) arc (234:-54:2.25);
\draw (63.2,-13.75) node [above] {$0$};
\fill [black] (64.52,-18.32) -- (65.4,-17.97) -- (64.59,-17.38);
\draw [black] (19.423,-35.98) arc (54:-234:2.25);
\draw (18.1,-40.55) node [below] {$0,\!1$};
\fill [black] (16.78,-35.98) -- (15.9,-36.33) -- (16.71,-36.92);
\draw [black] (64.523,-35.98) arc (54:-234:2.25);
\draw (63.2,-40.55) node [below] {$0,\!1$};
\fill [black] (61.88,-35.98) -- (61,-36.33) -- (61.81,-36.92);
\draw [black] (61.07,-35.41) arc (-49.10662:-130.89338:22.409);
\fill [black] (61.07,-35.41) -- (60.14,-35.56) -- (60.79,-36.31);
\draw (46.4,-41.38) node [below] {$1$};
\draw [black] (49.721,-35.359) arc (-50.34459:-129.65541:23.068);
\fill [black] (20.28,-35.36) -- (20.58,-36.25) -- (21.21,-35.48);
\draw (35,-41.17) node [below] {$1$};
\draw [black] (44.1,-21) -- (48.9,-21);
\fill [black] (48.9,-21) -- (48.1,-20.5) -- (48.1,-21.5);
\draw (46.5,-20.5) node [above] {$0$};
\node at (40,-48) {Automaton for $\us_{\mathrm{J}}$};
\end{tikzpicture}
&
\begin{tikzpicture}[scale=0.16]
\tikzstyle{every node}+=[inner sep=0pt]
\draw [black] (9.9,-35.7) circle (3);
\draw (9.9,-35.7) node {$1$};
\draw [black] (36.7,-35.7) circle (3);
\draw (36.7,-35.7) node {$0$};
\draw [black] (21.1,-35.7) circle (3);
\draw (21.1,-35.7) node {$1$};
\draw [black] (52.6,-35.7) circle (3);
\draw (52.6,-35.7) node {$0$};
\draw [black] (28.9,-25.7) circle (3);
\draw (28.9,-25.7) node {$1$};
\draw [black] (44.5,-25.7) circle (3);
\draw (44.5,-25.7) node {$1$};
\draw [black] (28.9,-16.2) circle (3);
\draw (28.9,-16.2) node {$1$};
\draw [black] (44.5,-16.2) circle (3);
\draw (44.5,-16.2) node {$0$};
\draw [black] (9.9,-9.7) circle (3);
\draw (9.9,-9.7) node {$0$};
\draw [black] (36.7,-9.7) circle (3);
\draw (36.7,-9.7) node {$1$};
\draw [black] (23.8,-9.7) circle (3);
\draw (23.8,-9.7) node {$0$};
\draw [black] (18.1,-35.7) -- (12.9,-35.7);
\fill [black] (12.9,-35.7) -- (13.7,-36.2) -- (13.7,-35.2);
\draw (15.5,-36.2) node [below] {$0$};
\draw [black] (49.6,-35.7) -- (39.7,-35.7);
\fill [black] (39.7,-35.7) -- (40.5,-36.2) -- (40.5,-35.2);
\draw (44.65,-36.2) node [below] {$0$};
\draw [black] (24.1,-35.7) -- (33.7,-35.7);
\fill [black] (33.7,-35.7) -- (32.9,-35.2) -- (32.9,-36.2);
\draw (28.9,-36.2) node [below] {$1$};
\draw [black] (50.189,-37.484) arc (-56.04265:-123.95735:33.906);
\fill [black] (12.31,-37.48) -- (12.7,-38.35) -- (13.25,-37.52);
\draw (31.25,-43.77) node [below] {$1$};
\draw [black] (10.619,-32.8) arc (193.80362:-94.19638:2.25);
\draw (16.02,-30.05) node [above] {$0,\!1$};
\fill [black] (12.64,-34.51) -- (13.54,-34.8) -- (13.3,-33.83);
\draw [black] (35.57,-38.466) arc (5.51717:-282.48283:2.25);
\draw (31.1,-41.63) node [left] {$0,\!1$};
\fill [black] (33.82,-36.49) -- (32.97,-36.07) -- (33.07,-37.06);
\draw [black] (28.9,-19.2) -- (28.9,-22.7);
\fill [black] (28.9,-22.7) -- (29.4,-21.9) -- (28.4,-21.9);
\draw (29.4,-20.95) node [right] {$0$};
\draw [black] (30.393,-23.112) arc (177.74705:-110.25295:2.25);
\draw (35.12,-20.61) node [right] {$0$};
\fill [black] (31.86,-25.31) -- (32.64,-25.84) -- (32.68,-24.84);
\draw [black] (27.79,-18.99) -- (22.21,-32.91);
\fill [black] (22.21,-32.91) -- (22.98,-32.36) -- (22.05,-31.99);
\draw (24.25,-25.07) node [left] {$1$};
\draw [black] (31.66,-26.87) -- (49.84,-34.53);
\fill [black] (49.84,-34.53) -- (49.29,-33.76) -- (48.9,-34.68);
\draw (39.78,-31.21) node [below] {$1$};
\draw [black] (41.74,-26.88) -- (23.86,-34.52);
\fill [black] (23.86,-34.52) -- (24.79,-34.67) -- (24.4,-33.75);
\draw (33.77,-31.21) node [below] {$0$};
\draw [black] (41.547,-25.24) arc (288.87797:0.87797:2.25);
\draw (38.38,-20.45) node [left] {$1$};
\fill [black] (43.07,-23.08) -- (43.28,-22.16) -- (42.34,-22.48);
\draw [black] (44.5,-19.2) -- (44.5,-22.7);
\fill [black] (44.5,-22.7) -- (45,-21.9) -- (44,-21.9);
\draw (44,-20.95) node [left] {$1$};
\draw [black] (45.65,-18.97) -- (51.45,-32.93);
\fill [black] (51.45,-32.93) -- (51.6,-32) -- (50.68,-32.38);
\draw (49.29,-25.03) node [right] {$0$};
\draw [black] (34.4,-11.62) -- (31.2,-14.28);
\fill [black] (31.2,-14.28) -- (32.14,-14.15) -- (31.5,-13.38);
\draw (31.79,-12.46) node [above] {$0$};
\draw [black] (39,-11.62) -- (42.2,-14.28);
\fill [black] (42.2,-14.28) -- (41.9,-13.38) -- (41.26,-14.15);
\draw (41.61,-12.46) node [above] {$1$};
\draw [black] (26.8,-9.7) -- (33.7,-9.7);
\fill [black] (33.7,-9.7) -- (32.9,-9.2) -- (32.9,-10.2);
\draw (30.25,-9.2) node [above] {$1$};
\draw [black] (20.8,-9.7) -- (12.9,-9.7);
\fill [black] (12.9,-9.7) -- (13.7,-10.2) -- (13.7,-9.2);
\draw (16.85,-9.2) node [above] {$0$};
\draw [black] (23.8,-3.5) -- (23.8,-6.7);
\draw (23.8,-3) node [above] {$\text{{\small Start}}$};
\fill [black] (23.8,-6.7) -- (24.3,-5.9) -- (23.3,-5.9);
\draw [black] (12.766,-10.546) arc (101.28243:-186.71757:2.25);
\draw (15.39,-15.74) node [right] {$0$};
\fill [black] (10.97,-12.49) -- (10.64,-13.37) -- (11.62,-13.18);
\draw [black] (9.9,-12.7) -- (9.9,-32.7);
\fill [black] (9.9,-32.7) -- (10.4,-31.9) -- (9.4,-31.9);
\draw (10.4,-22.7) node [right] {$1$};
\node at (31,-48) {Automaton for $\us_{\mathrm{J}}^{\circ 3}$};
\end{tikzpicture} \\ 
\hline\\
\end{tabular}
\caption{Automata corresponding to series of Chinburg--Symonds and Jean and their inverses.} 
\label{sigmatau}
\end{table}
\end{center}

Let us show how the power series of Chinburg--Symonds and Jean fit into our construction, and present the corresponding automata, using the same totally ramified cyclic  extension $\F_2\lau{z}(x,y)/\F_2\lau{z}$ of degree $4$ as in Example \ref{exrun}, but choosing different uniformisers $t$.
\begin{enumerate}
\item First, let $t=yx^{-2}$. After elimination, we find the (irreducible) equations 
\begin{align*}
t^2 X^2 + X + t^2 +t &=0 ; \\
(t^2+1) X^2 + X + t&=0 
\end{align*} for $\sigma$ and $\tau$, respectively. Looking at the valuations of the coefficients, we see that these equations have unique solutions of the form $t+O(t^2)$. The corresponding automata are given in the top right ($\sigma$) and the top left  ($\tau$) of Table \ref{sigmatau}. 
We now briefly indicate how these automata can be used to construct explicit formulas for $\us$ and $\tau$, showing that $\sigma = \sigma_{\mathrm{CS}}^{\circ 3}$ and $\tau = \sigma_{\mathrm{CS}}$. 
\begin{itemize}
\item Write $\tau=\sum_{i\geq 1}a_it^i$ with $a_i\in\F_2$. We will use the automaton corresponding to $\tau$ to determine for which $i\geq 1$ we have $a_i=1$. For such $i$, starting at the start vertex and walking through the automaton following the successive digits of $i$ in base $2$ (beginning with the least significant digit), we end up in a vertex with label $1$. Since we can disregard any leading zeros, this vertex has an incoming edge with label $1$. For $\tau$ note that this property holds precisely for those $i$ for which the base-$2$ expansion is either $1$, $10$ or of the form $11d_k\cdots d_1 0$ for some $k\geq 0$, $d_1,\dots,d_k\in\{0,1\}$, i.e.\ for $i$ equal to $1$, $2$ or such that $6\cdot 2^k \leq i < 8\cdot 2^k$ for some $k\geq 0$. 
 It follows that $\tau$ is given by the formula in \eqref{cseq}. 

\item For the power series $\us=\sum_{i\geq 1}b_it^i$ we see that the positive integers $i$ for which $b_i=1$ are precisely those which have a base-$2$ expansion of the form $1$, $100$, $1^k10$ or $101^k10$ with $k\geq 0$, and these are exactly the base-$2$ expansions of the numbers $1$, $4$, $4\cdot 2^k-2$ and $12\cdot 2^k-2$. This proves the formula for $\us$ given in \eqref{cs3eq}. 
\end{itemize} 
The fact that we can find such an explicit expression appears to be quite special. This relates to the fact that the automaton is `sparse' in the sense of Section \ref{aridsection} below. 
The automaton for $\tau$ is not sparse, but the base-$2$ expansion of the occurring powers has an explicit `closed' form. It turns out that this series is sparse up to multiplication by a rational function. 

\item Second, let $t=xy^{-1}$. Then we find the (irreducible) equations 
\begin{align*}
(t+1) X^2 + (t^2+1) X + t &= 0; \\
t X^2 + (t^2+1) X + t^2 + t &=0
\end{align*} for $\us$ and $\tau$, respectively. From formula \eqref{Jeq} we deduce that $\sigma_{\mathrm{J}}$ satisfies the same algebraic equation as $\sigma$, and since this equation has a unique solution of the form $t+O(t^2)$, we have $\sigma_{\mathrm{J}}=\sigma$. Solving the equations for $\sigma$ and $\tau$ by automata, we find that $\us$ correspond to the bottom left, and $\tau$ to the bottom right automaton depicted in Table \ref{sigmatau}. Converting the automata into explicit series as above, we find (after some rewriting) that
\begin{align*}
\sigma_{\mathrm{J}}=\us &=t+(t^7+t^2)\sum_{k\geq 0}t^{8k}+\sum_{k,\ell\geq 0}\left(t^{4\cdot 2^k(4\ell+1)+1}+t^{4\cdot 2^k(4\ell+3)}\right)  \\
&=t+ \frac{t^7+t^2}{t^8+1}+ \sum _{k\geq 2} \frac{t^{ 3\cdot 2^{k}}+t^{2^k+1}}{t^{4\cdot 2^k}+1}, 
\end{align*} and 
\begin{align} \sigma_{\mathrm{J}}^{\circ 3}=\tau & = 
t+(t^{11}+t^5)\sum_{k\geq 0}t^{16k}+\sum_{k\geq 1,\, \ell\geq 0}\left(t^{2^k(2\ell+1)}+t^{4\cdot 2^k(4\ell+1)-1}+t^{4\cdot 2^k(4\ell+3)+1}\right) \nonumber \\ 
& = t+\frac{t^{11}+t^5}{t^{16}+1} + \frac{t^2}{t^2+1} +\sum_{k\geq 3,\,\ell\geq 0}\left(t^{2^k(4\ell+1)-1}+t^{2^k(4\ell+3)+1}\right).\label{J3eq} 
\end{align}
On the other hand, from the algebraic equation for $\tau$ (which has a unique solution of the form $t+O(t^2)$), we can find directly another explicit form for $\tau$: the series $\tilde\tau:=\frac{t}{t^2+1}\cdot\tau$ 
satisfies $\tilde{\tau} = t^2/(t+1)^3 +\tilde{\tau}^2$, and hence (iteratively) $\tilde{\tau} = \sum_{k \geq 0} (t^2/(t+1)^3)^{2^k}$, leading to the formula  \begin{equation} \sigma_{\mathrm{J}}^{\circ 3} =\tau = \sum_{k\geq 0}\frac{t^{2\cdot 2^k-1}}{(t+1)^{3\cdot 2^k-2}}. \label{J3eq2}\end{equation} The series $\sigma$ and $\tau$ are further closed forms of elements of order $4$ in $\No(\F_2)$ with break sequence $(1,3)$ and conjugate to $\sigma_{\mathrm{CS}}^{\circ 3}$ and $\sigma_{\mathrm{CS}}$, respectively. 
\end{enumerate} 

\begin{figure}
\begin{tikzpicture}[scale=0.15]
\tikzstyle{every node}+=[inner sep=0pt]
\draw [black] (27.6,-13.5) circle (3);
\draw (27.6,-13.5) node {$0$};
\draw [black] (13.7,-25.6) circle (3);
\draw (13.7,-25.6) node {$0$};
\draw [black] (27.6,-25.6) circle (3);
\draw (27.6,-25.6) node {$1$};
\draw [black] (13.7,-38.3) circle (3);
\draw (13.7,-38.3) node {$0$};
\draw [black] (27.6,-38.3) circle (3);
\draw (27.6,-38.3) node {$1$};
\draw [black] (41.2,-25.6) circle (3);
\draw (41.2,-25.6) node {$1$};
\draw [black] (41.2,-38.3) circle (3);
\draw (41.2,-38.3) node {$0$};
\draw [black] (29.84,-15.49) -- (38.96,-23.61);
\fill [black] (38.96,-23.61) -- (38.69,-22.7) -- (38.03,-23.45);
\draw (35.41,-19.06) node [above] {$1$};
\draw [black] (41.2,-28.6) -- (41.2,-35.3);
\fill [black] (41.2,-35.3) -- (41.7,-34.5) -- (40.7,-34.5);
\draw (40.7,-31.95) node [left] {$1$};
\draw [black] (43.717,-36.69) arc (150.34157:-137.65843:2.25);
\draw (48.56,-37.45) node [right] {$0,\!1$};
\fill [black] (44.01,-39.32) -- (44.46,-40.15) -- (44.95,-39.28);
\draw [black] (22.1,-13.5) -- (24.6,-13.5);
\draw (21.6,-13.5) node [left] {$\text{{\small Start}}$};
\fill [black] (24.6,-13.5) -- (23.8,-13) -- (23.8,-14);
\draw [black] (25.34,-15.47) -- (15.96,-23.63);
\fill [black] (15.96,-23.63) -- (16.89,-23.48) -- (16.24,-22.73);
\draw (19.64,-19.06) node [above] {$0$};
\draw [black] (13.7,-28.6) -- (13.7,-35.3);
\fill [black] (13.7,-35.3) -- (14.2,-34.5) -- (13.2,-34.5);
\draw (13.2,-31.95) node [left] {$0$};
\draw [black] (11.02,-39.623) arc (-36:-324:2.25);
\draw (6.45,-38.3) node [left] {$0$};
\fill [black] (11.02,-36.98) -- (10.67,-36.1) -- (10.08,-36.91);
\draw [black] (24.6,-25.6) -- (16.7,-25.6);
\fill [black] (16.7,-25.6) -- (17.5,-26.1) -- (17.5,-25.1);
\draw (20.65,-25.1) node [above] {$1$};
\draw [black] (24.92,-39.623) arc (-36:-324:2.25);
\draw (20.35,-38.3) node [left] {$0$};
\fill [black] (24.92,-36.98) -- (24.57,-36.1) -- (23.98,-36.91);
\draw [black] (26.251,-35.63) arc (-160.85086:-199.14914:11.22);
\fill [black] (26.25,-35.63) -- (26.46,-34.71) -- (25.52,-35.04);
\draw (25.13,-31.95) node [left] {$0$};
\draw [black] (29.114,-28.177) arc (21.92551:-21.92551:10.104);
\fill [black] (29.11,-28.18) -- (28.95,-29.11) -- (29.88,-28.73);
\draw (30.34,-31.95) node [right] {$1$};
\draw [black] (11.02,-26.923) arc (-36:-324:2.25);
\draw (6.45,-25.6) node [left] {$1$};
\fill [black] (11.02,-24.28) -- (10.67,-23.4) -- (10.08,-24.21);
\draw [black] (15.91,-36.28) -- (25.39,-27.62);
\fill [black] (25.39,-27.62) -- (24.46,-27.79) -- (25.13,-28.53);
\draw (19.64,-31.46) node [above] {$1$};
\draw [black] (43.88,-24.277) arc (144:-144:2.25);
\draw (48.45,-25.6) node [right] {$0$};
\fill [black] (43.88,-26.92) -- (44.23,-27.8) -- (44.82,-26.99);
\end{tikzpicture}

\caption{Automaton corresponding to $\us_{\mathrm{CS}}^{\circ 2}\in\No(\F_2)$ of order $2$ with break sequence $(3)$.} 
\label{ka2}
\end{figure}

The element {$\sigma_{\mathrm{min}}$} in Proposition \ref{sigmaminprop} is conjugate to $\sigma_{\mathrm{CS}}$.   

\begin{remark} \label{Carlitz} We outline a construction of an automaton for such a series of order $4$ with minimal break sequence using the Carlitz module construction of abelian extensions of function fields, see e.g.\ \cite{Hayes} (this is a global class field theory version essentially equivalent to the local method based on Lubin--Tate theory used by Jean). 

Let $\rho \colon \F_2[z] \rightarrow \End(\mathbf{G}_a)$ denote the Carlitz module for $K:=\F_2(z)$ defined by $\rho_z(X)=zX+X^2$. Now the extension $K(\rho[z^3])/K$ given by adjoining the roots of $\rho_{z^3}(X)$ is Galois with Galois group $G = \left( \F_2[z]/z^3\right)^* \cong {\Z}/{4{\Z}},$ generated by the class of $z+1$ (of order $4$), where an element $g \in G$ acts on $\alpha \in K(\rho[z^3])$ by $g(\alpha):=\rho_g(\alpha)$. A minimal polynomial for the extension is 
$f:=\rho_{z^3}(X)/\rho_{z^2}(X)=X^4+(z^2+z)X^2+z^2X+z$, its splitting field is a cyclic degree-$4$ extension in which $z$ is totally ramified (and no other place ramifies, cf.\ \cite[Prop.\ 2.2, Thm.\ 3.2]{Hayes}), and a root $t$ is a uniformiser for the extension locally above $z$. The action of a generator of the Galois group is given by  $\sigma(t) = \rho_{z+1}(t) = t+zt+t^2$. 

Eliminating $z$, we find an equation 
$(t+1)X^2+(t^2+1)X+t=0$ for $\sigma$. This is exactly the equation for $\sigma_{\mathrm{J}}$, previously obtained using Witt vectors, and  solved by a series corresponding to the automaton in Table \ref{sigmatau} with $9$ states. 
\end{remark}

 \begin{remark} \label{rem:cs2}
If $\tau$ is an element of order $4$ with break sequence $(1,3)$, then $\tau^{\circ 2}$ has break sequence $(3)$, and hence is conjugate to the Klopsch's series $\sigma_{\mathrm{K},3}$ (see Example \ref{exklopschauto}). Taking $\tau=\sigma_{\mathrm{CS}}$ produces the power series $\sigma:=\sigma_{\mathrm{CS}}^{\circ 2} = t +t^4 +O(t^5)$, which satisfies $(t^2+1) X^2 + X + t^2 + t = 0$. The corresponding automaton is presented in Figure \ref{ka2}, leading to the following explicit formula for an element of order $2$ with break sequence $(3)$: 
\begin{equation*} \label{eqcs2} 
\us_{\mathrm{CS}}^{\circ 2} = t + \sum_{k \geq 0} \sum_{\ell = 0}^{2^k-1} t^{4 \cdot 2^k + 2 \ell} = t +\frac{1}{t^2+1}\sum_{k\geq 1} (t^{2\cdot 2^k}+t^{3\cdot 2^k}). 
\end{equation*} 
 \end{remark}

\subsection{Order $4$, break sequence $(1,5)= \langle 1,3 \rangle$} By Lubin's result (\cite{Lubinunpublished}, \cite[Thm.\ 2.2]{HK}),  there is a unique conjugacy class of such power series. No formula for such a series is known, but  following our philosophy, we can represent the solution by an automaton. 

\begin{proposition} \label{15class}  Up to conjugation, every element in $\No(\F_2)$ of order $4$ with break sequence $(1,5)= \langle 1,3 \rangle$ is given by the power series $\us_{(1,5)}$ corresponding to the automaton in Figure \textup{\ref{d13}} with $13$ states, with initial coefficients 
$$\sigma_{(1,5)}=t+t^2+t^3+t^4+t^6+O(t^{7}).$$
\end{proposition}

\begin{proof} 
Suitable algebraic equations are found from Witt's theory using Example \ref{re}; following Example \ref{exjap}, we start with the element $\beta:=(z^{-1},z^{-3})\in W_2(\F_2(\!(z)\!))$, 
and rewrite the resulting equation in terms of the variables $x:=\alpha_0$ and $y:=\alpha_1+\alpha_0^3+\alpha_0^2$ as 
\begin{equation}
\left\{ \begin{array}{l}
x^2+x=z^{-1}; \\
 y^2+y=x^5+x^3. \label{xy15} 
   \end{array} \right.
\end{equation}
(The variable $y$ is used instead of $\alpha_1$ since that choice allows us to use Lemma \ref{ramlem}.)  Writing $z_0=z, z_1, z_2$ for uniformisers of the fields in the tower of extensions $$K_0:=\F_2\lau{z} \subsetneq K_1=K_0(x) = \F_2\lau{z_1} \subsetneq K_2=K_1(y) = \F_2\lau{z_2},$$ we have 
$v_{z_1}(x) = v_{z_0} (z^{-1}) = -1$, so $v_{z_1}(x^5+x^3) = -5$, and hence $v_{z_2}(y) = -5$ and $v_{z_2}(x)=-2$. Hence the extensions are all totally ramified and we can choose $t=x^2y^{-1}$ as uniformiser for $K_2$ (since $v_{z_2}(t)=1$). 
A generator $\sigma$ for the Galois group of $K_2/K_0$ is determined by
\begin{equation*}
\left\{ \begin{array}{l}
\sigma(x)=x+1; \\
\sigma(y)=y+x^2+1, 
   \end{array} \right.
\end{equation*}
and with $t=x^2y^{-1}$ we compute that $\sigma(t) = (x^2+1)/(y+x^2+1)$. By eliminating $x$ and $y$ from these last two equations and the two equations in (\ref{xy15}), 
we find that $\us=\sigma(t)$ satisfies the following (irreducible) equation over $\F_2(t)$:  
\begin{equation} \label{eq15}
t^2X^3+(t+1)^3X+t^3+t=0.
\end{equation}
Considering the sum and product of the three solutions, we find that there is a unique solution with $\us=t+O(t^2)$. 
The corresponding automaton with initial coefficients $t + t^2 + t^3 + t^4 + t^6+O(t^7)$ and Equation (\ref{eq15}) produced by the algorithm 
is displayed in Figure \ref{d13}. 
\end{proof}

\begin{center}
\begin{figure}
\begin{tikzpicture}[scale=0.15]
\tikzstyle{every node}+=[inner sep=0pt]
\draw [black] (19.6,-38.5) circle (3);
\draw (19.6,-38.5) node {$0$};
\draw [black] (9.3,-38.5) circle (3);
\draw (9.3,-38.5) node {$1$};
\draw [black] (16.6,-17.2) circle (3);
\draw (16.6,-17.2) node {$1$};
\draw [black] (29.9,-17.2) circle (3);
\draw (29.9,-17.2) node {$1$};
\draw [black] (9.3,-9.8) circle (3);
\draw (9.3,-9.8) node {$1$};
\draw [black] (29.9,-49.2) circle (3);
\draw (29.9,-49.2) node {$1$};
\draw [black] (36.8,-26.2) circle (3);
\draw (36.8,-26.2) node {$0$};
\draw [black] (48.8,-17.2) circle (3);
\draw (48.8,-17.2) node {$0$};
\draw [black] (60.8,-9.8) circle (3);
\draw (60.8,-9.8) node {$0$};
\draw [black] (73.1,-38.5) circle (3);
\draw (73.1,-38.5) node {$1$};
\draw [black] (36.8,-38.5) circle (3);
\draw (36.8,-38.5) node {$0$};
\draw [black] (48.8,-30.4) circle (3);
\draw (48.8,-30.4) node {$0$};
\draw [black] (60.8,-24.1) circle (3);
\draw (60.8,-24.1) node {$0$};
\draw [black] (16.6,-38.5) -- (12.3,-38.5);
\fill [black] (12.3,-38.5) -- (13.1,-39) -- (13.1,-38);
\draw (14.45,-38) node [above] {$1$};
\draw [black] (19.6,-31.7) -- (19.6,-35.5);
\draw (19.6,-31.2) node [above] {$\text{{\small Start}}$};
\fill [black] (19.6,-35.5) -- (20.1,-34.7) -- (19.1,-34.7);
\draw [black] (22.28,-37.177) arc (144:-144:2.25);
\draw (26.85,-38.5) node [right] {$0$};
\fill [black] (22.28,-39.82) -- (22.63,-40.7) -- (23.22,-39.89);
\draw [black] (9.3,-35.5) -- (9.3,-12.8);
\fill [black] (9.3,-12.8) -- (8.8,-13.6) -- (9.8,-13.6);
\draw (8.8,-24.15) node [left] {$1$};
\draw [black] (12.12,-10.81) -- (27.08,-16.19);
\fill [black] (27.08,-16.19) -- (26.49,-15.44) -- (26.15,-16.39);
\draw (20.52,-12.97) node [above] {$0$};
\draw [black] (11.96,-39.88) -- (27.24,-47.82);
\fill [black] (27.24,-47.82) -- (26.76,-47) -- (26.3,-47.89);
\draw (18.61,-44.35) node [below] {$0$};
\draw [black] (29.9,-20.2) -- (29.9,-46.2);
\fill [black] (29.9,-46.2) -- (30.4,-45.4) -- (29.4,-45.4);
\draw (29.4,-33.2) node [left] {$0$};
\draw [black] (30.76,-46.33) -- (35.94,-29.07);
\fill [black] (35.94,-29.07) -- (35.23,-29.7) -- (36.19,-29.98);
\draw (32.58,-37.09) node [left] {$1$};
\draw [black] (34.97,-23.82) -- (31.73,-19.58);
\fill [black] (31.73,-19.58) -- (31.82,-20.52) -- (32.61,-19.91);
\draw (33.92,-20.29) node [right] {$1$};
\draw [black] (15.63,-20.04) -- (10.27,-35.66);
\fill [black] (10.27,-35.66) -- (11,-35.07) -- (10.06,-34.74);
\draw (13.71,-28.58) node [right] {$0$};
\draw [black] (19.19,-18.691) arc (87.79558:-200.20442:2.25);
\draw (21.36,-23.31) node [below] {$1$};
\fill [black] (16.99,-20.16) -- (16.46,-20.94) -- (17.46,-20.98);
\draw [black] (26.9,-17.2) -- (19.6,-17.2);
\fill [black] (19.6,-17.2) -- (20.4,-17.7) -- (20.4,-16.7);
\draw (23.25,-17.7) node [below] {$1$};
\draw [black] (58.25,-11.37) -- (51.35,-15.63);
\fill [black] (51.35,-15.63) -- (52.3,-15.63) -- (51.77,-14.78);
\draw (53.8,-13) node [above] {$0$};
\draw [black] (45.8,-17.2) -- (32.9,-17.2);
\fill [black] (32.9,-17.2) -- (33.7,-17.7) -- (33.7,-16.7);
\draw (39.35,-16.7) node [above] {$1$};
\draw [black] (12.3,-9.8) -- (57.8,-9.8);
\fill [black] (57.8,-9.8) -- (57,-9.3) -- (57,-10.3);
\draw [black] (71.92,-35.74) -- (61.98,-12.56);
\fill [black] (61.98,-12.56) -- (61.84,-13.49) -- (62.76,-13.1);
\draw (67.68,-23.2) node [right] {$1$};
\draw [black] (74.423,-41.18) arc (54:-234:2.25);
\draw (73.1,-45.75) node [below] {$0$};
\fill [black] (71.78,-41.18) -- (70.9,-41.53) -- (71.71,-42.12);
\draw [black] (32.81,-48.48) -- (70.19,-39.22);
\fill [black] (70.19,-39.22) -- (69.29,-38.93) -- (69.53,-39.9);
\draw (52.21,-44.42) node [below] {$0$};
\draw [black] (36.8,-29.2) -- (36.8,-35.5);
\fill [black] (36.8,-35.5) -- (37.3,-34.7) -- (36.3,-34.7);
\draw (37.3,-32.35) node [right] {$0$};
\draw [black] (39.8,-38.5) -- (70.1,-38.5);
\fill [black] (70.1,-38.5) -- (69.3,-38) -- (69.3,-39);
\draw (54.95,-39) node [below] {$1$};
\draw [black] (50.048,-19.92) arc (17.84916:-17.84916:12.657);
\fill [black] (50.05,-27.68) -- (50.77,-27.07) -- (49.82,-26.76);
\draw (51.16,-23.8) node [right] {$0$};
\draw [black] (47.585,-27.664) arc (-162.67777:-197.32223:12.978);
\fill [black] (47.59,-19.94) -- (46.87,-20.55) -- (47.82,-20.85);
\draw (46.5,-23.8) node [left] {$1$};
\draw [black] (62.123,-26.78) arc (54:-234:2.25);
\draw (60.8,-31.35) node [below] {$0,\!1$};
\fill [black] (59.48,-26.78) -- (58.6,-27.13) -- (59.41,-27.72);
\draw [black] (51.46,-29.01) -- (58.14,-25.49);
\fill [black] (58.14,-25.49) -- (57.2,-25.42) -- (57.67,-26.31);
\draw (55.79,-27.75) node [below] {$0$};
\draw [black] (60.8,-12.8) -- (60.8,-21.1);
\fill [black] (60.8,-21.1) -- (61.3,-20.3) -- (60.3,-20.3);
\draw (60.3,-16.95) node [left] {$1$};
\draw [black] (39.29,-36.82) -- (46.31,-32.08);
\fill [black] (46.31,-32.08) -- (45.37,-32.11) -- (45.93,-32.94);
\draw (43.8,-34.95) node [below] {$0$};
\draw (35.52,-9.3) node [above] {$1$};
\end{tikzpicture}
 \caption{Automaton representing a power series $\us_{(1,5)}$  of order $4$ with break sequence $(1,5)$ (unique up to conjugation).} 
\label{d13}
\end{figure}
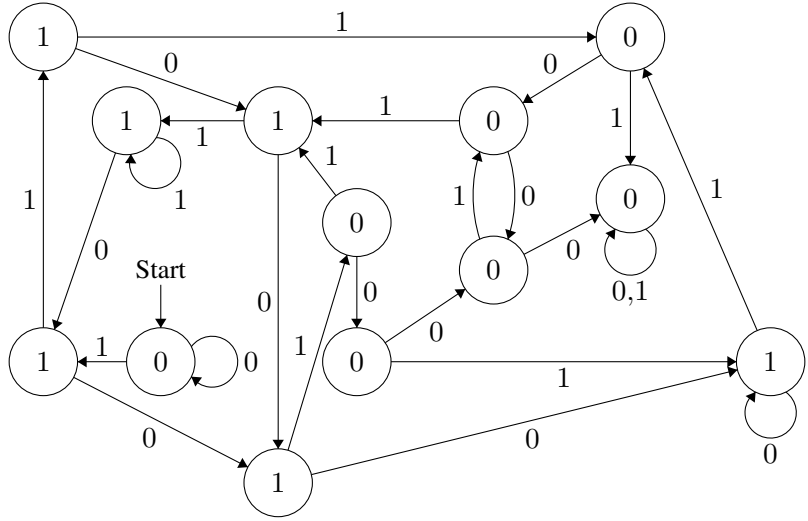
\end{center}

\subsection{Order $4$, break sequence $(1,9) = \langle 1, 5 \rangle$} Again by Lubin's result in loc.\ cit., there is a unique conjugacy class of such power series. A corresponding automaton is found as follows.

\begin{proposition} \label{19class} 
Up to conjugation, every element in $\No(\F_2)$ of order $4$ with break sequence $(1,9)= \langle 1,5 \rangle$ is given by the power series $\us_{(1,9)}$ corresponding to the automaton described as follows using the data in Table \textup{\ref{19fig}:} it has $110$ states, corresponding to the $110$ triples on the displayed ordered list, where the start vertex is the first triple on the list and a triple $(l,i,j)$ occurs on the list precisely if the following three conditions hold\textup{:} it has label $l$, there is a directed edge with label $0$ to the $i$-th triple on the list and there is a directed edge with label $1$ to the $j$-th triple on the list. The initial coefficients of $\us_{(1,9)}$ are 
$$\us_{(1,9)} =t+t^2+t^3+t^4+t^5+t^6+t^7+t^9+t^{11}+t^{12}+t^{13}+t^{17}+t^{18}+O(t^{19}). $$
\end{proposition}

\begin{proof}
Following Example \ref{exjap}(c), we start with $\beta=(z^{-1},z^{-10}) \in W_2(\F_2\lau{z})$. In the resulting equations $\wp(\alpha)=\beta$, change variables to  $x:=\alpha_0$ and $y:=\alpha_1+\alpha_0^{10}+\alpha_0^9+\alpha_0^6+\alpha_0^3+\alpha_0$ to find 
\begin{equation*}
\left\{ \begin{array}{l}
x^2+x=z^{-1}; \\ 
y^2+y=x^9+x.
\end{array} \right.
\end{equation*}
Writing $z_0=z$, $z_1$, $z_2$ for uniformisers of the fields in the tower of extensions $$K_0:=\F_2\lau{z}\subsetneq K_1=K_0(x)=\F_2\lau{z_1}\subsetneq K_2=K_1(y)=\F_2\lau{z_2},$$
we have $v_{z_1}(x)=-1$, so $v_{z_1}(x^9+x)=-9$, and hence $v_{z_2}(y)=-9$, $v_{z_2}(x)=-2$ and $v_{z_2}(z)=4$. Hence all extensions are totally ramified and we can choose $t=x^{-1}yz^2$ as uniformiser for $K_2$ (since $v_{z_2}(t)=1$). A generator $\sigma$ for the Galois group of $K_2/K_0$ is determined by
\begin{equation*}
\left\{ \begin{array}{l}
\sigma(x)=x+1; \\
\sigma(y)=y+x^4+x^2+x+1, 
   \end{array} \right.
\end{equation*}
By elimination of variables, we find that $\us=\sigma(t)$ satisfies the following (irreducible) equation over $\F_2(t)$: \begin{align*} \label{eq19} t^2\us^7  + t^3\us^6 &+ (t^5 + t^4 + t^2)X^5 + (t^5 + t^3)X^4+  \nonumber \\  & 
 (t^7+t^5+t^4+t^3+t)X^3 + t^5X^2 + (t^3 + t + 1)X + t=0. \end{align*} 
There is a unique solution of the form $t+O(t^2)$, and its initial coefficients are as indicated in the proposition; the corresponding 2-automaton can be found in Table \ref{19fig} and in \cite{Database} (the visual representation in Table \ref{19fig} is more of an illustration but can be manipulated directly in \cite{Database} using standard graph theory algorithms). 
\end{proof} 

\begin{table} 
\hrule
{\footnotesize
\begin{align*}
& ((0, 2, 3), (0, 7, 8), (1, 3, 69), (0, 5, 6), (0, 12, 13), (1, 85, 72), (0, 20, 16), (1, 88, 89), (0, 10, 11), (0, 24, 37), \\
& (1, 84, 74), (0, 24, 40), (1, 104, 76), (0, 15, 16), (0, 20, 37), (1, 8, 72), (0, 18, 19), (0, 30, 13), (1, 73, 74), \\
& (0, 48, 49), (0, 22, 23), (0, 48, 60), (1, 91, 92), (0, 20, 25), (1, 81, 94), (0, 27, 6), (0, 35, 28), (0, 29, 19), \\
& (0, 62, 28), (0, 31, 9), (0, 22, 32), (1, 88, 78), (0, 34, 11), (0, 31, 49), (0, 7, 36), (1, 105, 52), (1, 23, 61), \\
& (0, 7, 39), (1, 106, 14), (1, 95, 17), (0, 42, 36), (0, 20, 40), (0, 22, 36), (0, 45, 46), (0, 31, 60), (1, 40, 55), \\
& (0, 22, 39), (0, 50, 51), (0, 35, 54), (0, 50, 50), (0, 48, 9), (0, 53, 54), (0, 21, 50), (0, 59, 57), (0, 49, 56), \\
& (0, 51, 50), (0, 58, 57), (0, 62, 54), (0, 66, 4), (0, 38, 4), (0, 60, 56), (0, 21, 43), (0, 30, 54), (0, 43, 50), \\
& (0, 21, 51), (0, 21, 7), (0, 65, 4), (0, 47, 9), (1, 8, 98), (1, 70, 71), (1, 100, 92), (1, 75, 76), (1, 97, 98), \\
& (1, 79, 78), (1, 81, 82), (1, 69, 76), (1, 70, 78), (1, 83, 78), (1, 77, 80), (1, 102, 72), (1, 88, 90), (1, 103, 74), \\
& (1, 70, 89), (1, 39, 82), (1, 91, 80), (1, 77, 87), (1, 93, 94), (1, 79, 64), (1, 99, 52), (1, 101, 14), (1, 79, 68), \\
& (1, 36, 61), (1, 106, 68), (1, 107, 63), (1, 91, 96), (1, 108, 17), (1, 106, 41), (1, 16, 55), (1, 77, 92), (1, 70, 90), \\
& (1, 77, 96), (1, 106, 64), (1, 109, 26), (1, 23, 26), (1, 86, 52), (1, 86, 67), (1, 110, 17), (1, 105, 67), (1, 105, 44), \\
& (1, 105, 33))
\end{align*} 
}
\includegraphics[width=13cm]{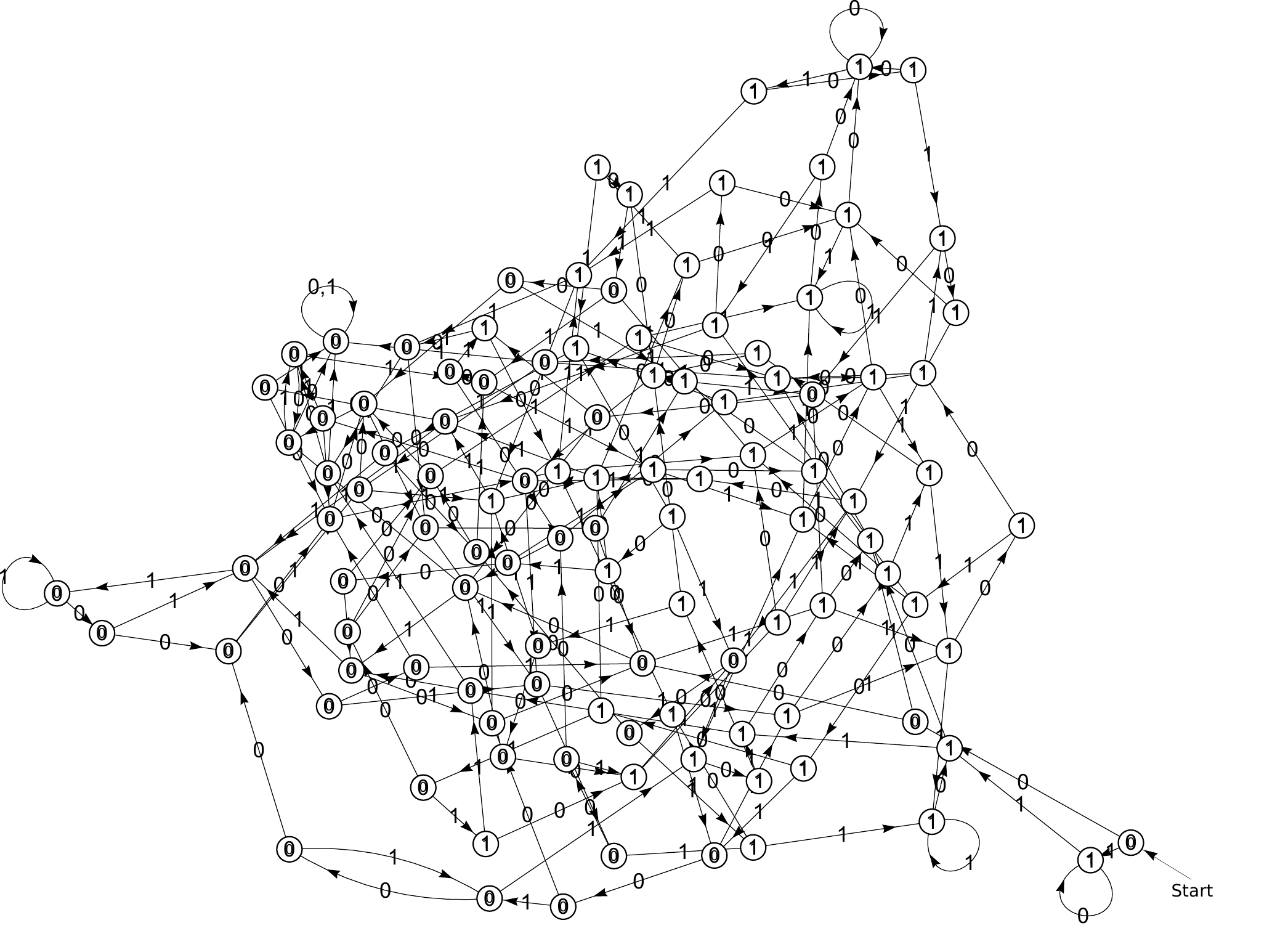}
\hrule 
\caption{Representation of the automaton for the power series $\us_{(1,9)}$ of order $4$ with break sequence $(1,9)$.} \label{19fig}
\end{table}

\section{Some new explicit formulas for power series of order $4$} 

The explicit power series $\us_{\mathrm{CS}}$ and its inverse are a full set of representatives for the conjugacy classes of order-$4$ elements with break sequence $(1,3)$. The series $\us_{\mathrm{J}}$ is another power series with a nice closed formula. We did a larger search for automata corresponding to such power series and found five more for which we could write down reasonably sized closed formulas. One of these is the inverse of Jean's series displayed in Equations (\ref{J3eq}), (\ref{J3eq2}). We list the other four in Table \ref{s5-s8}. 

\begin{table}
\begin{tabular}{l} 
\hline \\[0.5mm]
$\displaystyle{\sigma_{\mathrm{T},1} = t+\sum_{k\geq 2}\left(t^{2^k-2}+t^{2\cdot 2^k-1}+t^{4\cdot 2^k-5}\right)+\sum_{k,\ell\geq 2}t^{2^k(2^\ell-3)+1} } = t+t^2 + O(t^5)$. 
\\[7mm]
$\displaystyle{\sigma_{\mathrm{T},2} =  t+t^2+\sum_{k\geq 3}\left(t^{2^k-4}+t^{2^k-3}+t^{2^k-1}+t^{4\cdot 2^k-6}+t^{4\cdot 2^k-5}+t^{8\cdot 2^k-22}+t^{8\cdot 2^k-21}\right)} + $ \\  
$\displaystyle{\quad \quad  (t+1)\sum_{k,\ell\geq 3}t^{2^k(2^\ell-6)+2}+(t+1)\sum_{k,\ell,m\geq 2}t^{2^{k+\ell}(2^m-3)+2\cdot 2^k-2}}= t+t^2+t^4+t^5 + O(t^7)$. 
\\[7mm]
${\displaystyle \sigma_{\mathrm{T},3} = t+t^8+t^{44}+\sum_{k\geq 2}\left(t^{2^k-2}+t^{3\cdot 2^k-2}+t^{8\cdot 2^k-4}+t^{8\cdot 2^k+4}+t^{8\cdot 2^k+20}+t^{16\cdot 2^k+44}+t^{24\cdot 2^k-4}\right) } + $  \\
${\displaystyle  \quad \quad  \sum_{k,\ell\geq 2}\left(t^{2^k(2^\ell+3)-2}+t^{4\cdot 2^k(2^\ell+2)+4}+t^{8\cdot 2^k(2^\ell+3)-4}+t^{8\cdot 2^k(2^\ell+2)+12}\right)+}$  \\
${\displaystyle \quad \quad  \sum_{k,\ell\geq2,m\geq 1}\left(t^{2^{k+\ell}(2^m+1)+2^k-2}+t^{8\cdot 2^{k+\ell}(2^m+1)+8\cdot 2^k-4}\right)} = t+t^2 +t^6+t^8+t^{10}+ O(t^{13})\,.
$
\\[7mm]
${\displaystyle\sigma_{\mathrm{T},4} = t+t^4+t^8+t^{20}+\sum_{k\geq 2}\left(t^{2^k-2}+t^{8\cdot 2^k-4}+t^{8\cdot 2^k+20}+t^{16\cdot 2^k+12}+t^{16\cdot 2^k+44}\right)}+$ \\
${\displaystyle \quad \quad  \sum_{k,\ell\geq 2}\left(t^{2^k(2^\ell+1)-2}+t^{8\cdot 2^k(2^\ell+1)-4}+t^{4\cdot 2^k(2^\ell+2)+4}+t^{8\cdot 2^k(2^\ell+2)+12}+t^{2^k(2^\ell+3)-2}+t^{8\cdot 2^k(2^\ell+3)-4}\right)} + $\\
${\displaystyle \quad \quad  \sum_{k,\ell\geq2,m\geq 1}\left(t^{2^{k+\ell}(2^m+1)+2^k-2}+t^{8\cdot 2^{k+\ell}(2^m+1)+8\cdot 2^k-4}\right) = t+t^2 + t^4 +t^6+t^8+ O(t^{13}).
}$ 
\\[7mm] \hline \\[1mm] 
\end{tabular} 
\caption{Four explicit power series of order $4$ with break sequence $(1,3)$ (the representation is minimal in the sense that no monomial occurs twice in the same formula).} 
\label{s5-s8} 
\end{table}

We start with the equation from Example \ref{exrun}, but choose yet different uniformisers $t$. Recall that we write $\tau=\sigma^{\circ 3}$. 
\begin{enumerate} 
\item First, let $t=(1+x^2+y)/(x^2+xy)$. Then $\sigma=\sigma_{\mathrm{T},1}$ satisfies $$ t^2X^4+(t^4+t^2+t+1)X^2+(t^3+t^2+t)X+t^3=0$$ and $\tau=\sigma_{\mathrm{T},2}$ satisfies $$t^2X^4+(t+1)X^3+(t^4+t^2+t)X^2+(t^2+t)X+t^2=0.$$
Solving these (irreducible) equations by automata, we find that $\sigma_{\mathrm{T},1}$ and $\sigma_{\mathrm{T},2}$ correspond to the top left, respectively top right automaton depicted in Table \ref{sigmatau2}. It is relatively straightforward to convert the automata into explicit series following the method explained after Corollary \ref{31class}, and the result is 
shown in Table \ref{s5-s8} (including the initial coefficients). 
\item Secondly, let $t=xy/(x^3+y)$. Then $\us=\sigma_{\mathrm{T},3}$ satisfies $$t^4X^4+(t^2+1)X^3+(t^3+t)X^2+t^2X+t^3=0$$
and $\tau=\sigma_{\mathrm{T},4}$ satisfies the same equation as $\us$ (it turns out  that another solution is $\sigma_{\mathrm{T},3}^{\circ 2}=\sigma_{\mathrm{T},4}^{\circ 2}$). Solving this (irreducible) equation by automata, we find that $\us$ and $\tau$ correspond to the bottom left and bottom right automaton depicted in Table \ref{sigmatau2}. Converting the automata into explicit series as before, we find the formulas in Table \ref{s5-s8} (again including the initial coefficients).
\end{enumerate}

By the criterion in Lemma \ref{recognise}, we see easily that $\sigma_{\mathrm{T},2}, \sigma_{\mathrm{T},3}$ and $\sigma_{\mathrm{CS}}$ are conjugate, and so are $\sigma_{\mathrm{T},1}, \sigma_{\mathrm{T},4}$ and $\sigma_{\mathrm{CS}}^{\circ 3}$. 
\begin{center}
\begin{table}[t]
\begin{tabular}{cc}
\hline \\
\begin{tikzpicture}[scale=0.14]
\tikzstyle{every node}+=[inner sep=0pt]
\tikzstyle{every node}+=[inner sep=0pt]
\draw [black] (21.4,-14.3) circle (3);
\draw (21.4,-14.3) node {$0$};
\draw [black] (35.6,-14.3) circle (3);
\draw (35.6,-14.3) node {$1$};
\draw [black] (35.6,-25.6) circle (3);
\draw (35.6,-25.6) node {$0$};
\draw [black] (48.9,-14.3) circle (3);
\draw (48.9,-14.3) node {$1$};
\draw [black] (21.4,-25.6) circle (3);
\draw (21.4,-25.6) node {$0$};
\draw [black] (48.9,-34.7) circle (3);
\draw (48.9,-34.7) node {$1$};
\draw [black] (35.6,-34.7) circle (3);
\draw (35.6,-34.7) node {$1$};
\draw [black] (21.4,-34.7) circle (3);
\draw (21.4,-34.7) node {$1$};
\draw [black] (6.3,-34.7) circle (3);
\draw (6.3,-34.7) node {$0$};
\draw [black] (21.4,-17.3) -- (21.4,-22.6);
\fill [black] (21.4,-22.6) -- (21.9,-21.8) -- (20.9,-21.8);
\draw (20.9,-19.95) node [left] {$0$};
\draw [black] (24.4,-14.3) -- (32.6,-14.3);
\fill [black] (32.6,-14.3) -- (31.8,-13.8) -- (31.8,-14.8);
\draw (28.5,-13.8) node [above] {$1$};
\draw [black] (35.6,-17.3) -- (35.6,-22.6);
\fill [black] (35.6,-22.6) -- (36.1,-21.8) -- (35.1,-21.8);
\draw (35.1,-19.95) node [left] {$1$};
\draw [black] (38.6,-14.3) -- (45.9,-14.3);
\fill [black] (45.9,-14.3) -- (45.1,-13.8) -- (45.1,-14.8);
\draw (42.25,-13.8) node [above] {$0$};
\draw [black] (48.183,-17.201) arc (13.84113:-274.15887:2.25);
\draw (43.54,-19.95) node [below] {$0$};
\fill [black] (46.16,-15.5) -- (45.26,-15.2) -- (45.5,-16.17);
\draw [black] (48.9,-17.3) -- (48.9,-31.7);
\fill [black] (48.9,-31.7) -- (49.4,-30.9) -- (48.4,-30.9);
\draw (49.4,-24.5) node [right] {$1$};
\draw [black] (32.6,-25.6) -- (24.4,-25.6);
\fill [black] (24.4,-25.6) -- (25.2,-26.1) -- (25.2,-25.1);
\draw (28.5,-25.1) node [above] {$0$};
\draw [black] (35.6,-28.6) -- (35.6,-31.7);
\fill [black] (35.6,-31.7) -- (36.1,-30.9) -- (35.1,-30.9);
\draw (35.1,-30.15) node [left] {$1$};
\draw [black] (23.93,-27.22) -- (33.07,-33.08);
\fill [black] (33.07,-33.08) -- (32.67,-32.23) -- (32.13,-33.07);
\draw (29.5,-29.65) node [above] {$1$};
\draw [black] (45.9,-34.7) -- (38.6,-34.7);
\fill [black] (38.6,-34.7) -- (39.4,-35.2) -- (39.4,-34.2);
\draw (42.25,-34.2) node [above] {$0$};
\draw [black] (36.476,-31.843) arc (190.68691:-97.31309:2.25);
\draw (41.25,-29.3) node [above] {$1$};
\fill [black] (38.4,-33.66) -- (39.28,-34) -- (39.09,-33.02);
\draw [black] (32.6,-34.7) -- (24.4,-34.7);
\fill [black] (24.4,-34.7) -- (25.2,-35.2) -- (25.2,-34.2);
\draw (28.5,-34.2) node [above] {$0$};
\draw [black] (18.4,-34.7) -- (9.3,-34.7);
\fill [black] (9.3,-34.7) -- (10.1,-35.2) -- (10.1,-34.2);
\draw (13.85,-34.2) node [above] {$1$};
\draw [black] (24.22,-35.688) arc (98.42508:-189.57492:2.25);
\draw (27.15,-40) node [below] {$0$};
\fill [black] (22.33,-37.54) -- (21.95,-38.4) -- (22.94,-38.26);
\draw [black] (5.315,-31.879) arc (226.98026:-61.01974:2.25);
\draw (7.6,-27.34) node [above] {$0,\!1$};
\fill [black] (7.94,-32.2) -- (8.85,-31.96) -- (8.12,-31.28);
\draw [black] (18.83,-27.15) -- (8.87,-33.15);
\fill [black] (8.87,-33.15) -- (9.81,-33.17) -- (9.3,-32.31);
\draw (12.85,-29.65) node [above] {$0$};
\draw [black] (46.554,-36.568) arc (-54.12182:-125.87818:32.341);
\fill [black] (8.65,-36.57) -- (9,-37.44) -- (9.59,-36.63);
\draw (27.6,-43.2) node [below] {$1$};
\draw [black] (12.4,-14.3) -- (18.4,-14.3);
\draw (11.9,-14.3) node [left] {$\text{{\small Start}}$};
\fill [black] (18.4,-14.3) -- (17.6,-13.8) -- (17.6,-14.8);
\node at (29,-48) {Automaton of $\sigma_{\mathrm{T},1}$};
\end{tikzpicture}
&
\begin{tikzpicture}[scale=0.14]
\tikzstyle{every node}+=[inner sep=0pt]
\draw [black] (36.8,-4.3) circle (3);
\draw (36.8,-4.3) node {$0$};
\draw [black] (21.3,-4.3) circle (3);
\draw (21.3,-4.3) node {$0$};
\draw [black] (53,-4.3) circle (3);
\draw (53,-4.3) node {$1$};
\draw [black] (63.5,-27.1) circle (3);
\draw (63.5,-27.1) node {$0$};
\draw [black] (11.3,-27.1) circle (3);
\draw (11.3,-27.1) node {$1$};
\draw [black] (21.3,-27.1) circle (3);
\draw (21.3,-27.1) node {$1$};
\draw [black] (11.3,-49.5) circle (3);
\draw (11.3,-49.5) node {$0$};
\draw [black] (53,-27.1) circle (3);
\draw (53,-27.1) node {$0$};
\draw [black] (63.5,-49.5) circle (3);
\draw (63.5,-49.5) node {$1$};
\draw [black] (36.8,-27.1) circle (3);
\draw (36.8,-27.1) node {$0$};
\draw [black] (21.3,-49.5) circle (3);
\draw (21.3,-49.5) node {$1$};
\draw [black] (53,-49.5) circle (3);
\draw (53,-49.5) node {$0$};
\draw [black] (36.8,-15.8) circle (3);
\draw (36.8,-15.8) node {$0$};
\draw [black] (36.8,-49.5) circle (3);
\draw (36.8,-49.5) node {$1$};
\draw [black] (47.4,-36.1) circle (3);
\draw (47.4,-36.1) node {$1$};
\draw [black] (36.8,-36.1) circle (3);
\draw (36.8,-36.1) node {$1$};
\draw [black] (26.6,-36.1) circle (3);
\draw (26.6,-36.1) node {$0$};
\draw [black] (41.9,-8.8) -- (39.05,-6.28);
\draw (42.44,-10.25) node [right] {$\text{{\small Start}}$};
\fill [black] (39.05,-6.28) -- (39.32,-7.19) -- (39.98,-6.44);
\draw [black] (33.8,-4.3) -- (24.3,-4.3);
\fill [black] (24.3,-4.3) -- (25.1,-4.8) -- (25.1,-3.8);
\draw (29.05,-3.8) node [above] {$0$};
\draw [black] (39.8,-4.3) -- (50,-4.3);
\fill [black] (50,-4.3) -- (49.2,-3.8) -- (49.2,-4.8);
\draw (44.9,-4.8) node [below] {$1$};
\draw [black] (54.25,-7.02) -- (62.25,-24.38);
\fill [black] (62.25,-24.38) -- (62.36,-23.44) -- (61.46,-23.86);
\draw (57.53,-16.73) node [left] {$1$};
\draw [black] (20.1,-7.05) -- (12.5,-24.35);
\fill [black] (12.5,-24.35) -- (13.28,-23.82) -- (12.37,-23.42);
\draw (15.57,-14.72) node [left] {$1$};
\draw [black] (14.3,-27.1) -- (18.3,-27.1);
\fill [black] (18.3,-27.1) -- (17.5,-26.6) -- (17.5,-27.6);
\draw (16.3,-26.6) node [above] {$0$};
\draw [black] (11.3,-30.1) -- (11.3,-46.5);
\fill [black] (11.3,-46.5) -- (11.8,-45.7) -- (10.8,-45.7);
\draw (10.8,-38.3) node [left] {$1$};
\draw [black] (60.5,-27.1) -- (56,-27.1);
\fill [black] (56,-27.1) -- (56.8,-27.6) -- (56.8,-26.6);
\draw (58.25,-26.6) node [above] {$0$};
\draw [black] (63.5,-30.1) -- (63.5,-46.5);
\fill [black] (63.5,-46.5) -- (64,-45.7) -- (63,-45.7);
\draw (63,-38.3) node [left] {$1$};
\draw [black] (24.3,-27.1) -- (33.8,-27.1);
\fill [black] (33.8,-27.1) -- (33,-26.6) -- (33,-27.6);
\draw (29.05,-27.6) node [below] {$1$};
\draw [black] (12.623,-52.18) arc (54:-234:2.25);
\draw (11.3,-56.75) node [below] {$1$};
\fill [black] (9.98,-52.18) -- (9.1,-52.53) -- (9.91,-53.12);
\draw [black] (64.823,-52.18) arc (54:-234:2.25);
\draw (63.5,-56.75) node [below] {$1$};
\fill [black] (62.18,-52.18) -- (61.3,-52.53) -- (62.11,-53.12);
\draw [black] (60.885,-50.969) arc (-62.80055:-117.19945:40.441);
\fill [black] (23.91,-50.97) -- (24.4,-51.78) -- (24.86,-50.89);
\draw (42.4,-55.94) node [below] {$0$};
\draw [black] (21.3,-30.1) -- (21.3,-46.5);
\fill [black] (21.3,-46.5) -- (21.8,-45.7) -- (20.8,-45.7);
\draw (20.8,-38.3) node [left] {$0$};
\draw [black] (50.385,-50.969) arc (-62.83319:-117.16681:39.938);
\fill [black] (50.38,-50.97) -- (49.44,-50.89) -- (49.9,-51.78);
\draw (32.15,-55.87) node [below] {$0$};
\draw [black] (53,-30.1) -- (53,-46.5);
\fill [black] (53,-46.5) -- (53.5,-45.7) -- (52.5,-45.7);
\draw (53.5,-38.3) node [right] {$0$};
\draw [black] (36.8,-24.1) -- (36.8,-18.8);
\fill [black] (36.8,-18.8) -- (36.3,-19.6) -- (37.3,-19.6);
\draw (37.3,-21.45) node [right] {$0$};
\draw [black] (18.474,-48.528) arc (278.755:-9.245:2.25);
\draw (15.53,-44.23) node [above] {$0$};
\fill [black] (20.35,-46.67) -- (20.72,-45.8) -- (19.74,-45.95);
\draw [black] (53.944,-46.665) arc (189.3226:-98.6774:2.25);
\draw (58.77,-44.22) node [above] {$0$};
\fill [black] (55.82,-48.52) -- (56.69,-48.89) -- (56.53,-47.9);
\draw [black] (23.71,-6.09) -- (34.39,-14.01);
\fill [black] (34.39,-14.01) -- (34.05,-13.13) -- (33.45,-13.94);
\draw (28.05,-10.55) node [below] {$0$};
\draw [black] (50,-49.5) -- (39.8,-49.5);
\fill [black] (39.8,-49.5) -- (40.6,-50) -- (40.6,-49);
\draw (44.9,-49) node [above] {$1$};
\draw [black] (24.3,-49.5) -- (33.8,-49.5);
\fill [black] (33.8,-49.5) -- (33,-49) -- (33,-50);
\draw (29.05,-50) node [below] {$1$};
\draw [black] (39.09,-29.04) -- (45.11,-34.16);
\fill [black] (45.11,-34.16) -- (44.83,-33.26) -- (44.18,-34.02);
\draw (41.09,-32.09) node [below] {$1$};
\draw [black] (39.06,-17.77) arc (45.81063:9.33357:27.662);
\fill [black] (47.07,-33.12) -- (47.44,-32.25) -- (46.45,-32.41);
\draw (44.98,-23.65) node [right] {$1$};
\draw [black] (38.66,-47.15) -- (45.54,-38.45);
\fill [black] (45.54,-38.45) -- (44.65,-38.77) -- (45.43,-39.39);
\draw (42.66,-44.22) node [right] {$0$};
\draw [black] (44.4,-36.1) -- (39.8,-36.1);
\fill [black] (39.8,-36.1) -- (40.6,-36.6) -- (40.6,-35.6);
\draw (42.1,-36.6) node [below] {$0$};
\draw [black] (48.723,-38.78) arc (54:-234:2.25);
\draw (47.4,-43.35) node [below] {$1$};
\fill [black] (46.08,-38.78) -- (45.2,-39.13) -- (46.01,-39.72);
\draw [black] (38.123,-38.78) arc (54:-234:2.25);
\draw (36.8,-43.35) node [below] {$0$};
\fill [black] (35.48,-38.78) -- (34.6,-39.13) -- (35.41,-39.72);
\draw [black] (27.923,-38.78) arc (54:-234:2.25);
\draw (26.6,-43.35) node [below] {$0,\!1$};
\fill [black] (25.28,-38.78) -- (24.4,-39.13) -- (25.21,-39.72);
\draw [black] (34.98,-47.11) -- (28.42,-38.49);
\fill [black] (28.42,-38.49) -- (28.5,-39.43) -- (29.3,-38.82);
\draw (31.13,-44.2) node [left] {$1$};
\draw [black] (33.8,-36.1) -- (29.6,-36.1);
\fill [black] (29.6,-36.1) -- (30.4,-36.6) -- (30.4,-35.6);
\draw (31.7,-36.6) node [below] {$1$};
\draw [black] (26.345,-33.114) arc (-179.56992:-233.78576:19.324);
\fill [black] (26.35,-33.11) -- (26.84,-32.31) -- (25.84,-32.32);
\draw (27.71,-23.18) node [left] {$0$};
\draw [black] (50,-27.1) -- (39.8,-27.1);
\fill [black] (39.8,-27.1) -- (40.6,-27.6) -- (40.6,-26.6);
\draw (44.9,-27.6) node [below] {$1$};
\draw [black] (52.48,-7.25) -- (47.92,-33.15);
\fill [black] (47.92,-33.15) -- (48.55,-32.44) -- (47.57,-32.27);
\draw (49.48,-19.95) node [left] {$0$};
\node at (38,-61) {Automaton of $\sigma_{\mathrm{T},2}$};
\end{tikzpicture}
\\
\hline  \\
\begin{tikzpicture}[scale=0.13]
\tikzstyle{every node}+=[inner sep=0pt]
\draw [black] (56.4,-4.2) circle (3);
\draw (56.4,-4.2) node {$0$};
\draw [black] (33.9,-4.2) circle (3);
\draw (33.9,-4.2) node {$0$};
\draw [black] (33.9,-13.4) circle (3);
\draw (33.9,-13.4) node {$0$};
\draw [black] (33.9,-23.1) circle (3);
\draw (33.9,-23.1) node {$0$};
\draw [black] (21.1,-13.4) circle (3);
\draw (21.1,-13.4) node {$0$};
\draw [black] (11.5,-4.2) circle (3);
\draw (11.5,-4.2) node {$1$};
\draw [black] (11.5,-23.1) circle (3);
\draw (11.5,-23.1) node {$1$};
\draw [black] (21.1,-23.1) circle (3);
\draw (21.1,-23.1) node {$0$};
\draw [black] (33.9,-52.1) circle (3);
\draw (33.9,-52.1) node {$0$};
\draw [black] (11.5,-39.1) circle (3);
\draw (11.5,-39.1) node {$1$};
\draw [black] (21.1,-39.1) circle (3);
\draw (21.1,-39.1) node {$0$};
\draw [black] (11.5,-56.2) circle (3);
\draw (11.5,-56.2) node {$0$};
\draw [black] (26.9,-45.7) circle (3);
\draw (26.9,-45.7) node {$0$};
\draw [black] (26.9,-33.3) circle (3);
\draw (26.9,-33.3) node {$1$};
\draw [black] (56.4,-56.2) circle (3);
\draw (56.4,-56.2) node {$1$};
\draw [black] (44.9,-23.1) circle (3);
\draw (44.9,-23.1) node {$0$};
\draw [black] (44.9,-39.7) circle (3);
\draw (44.9,-39.7) node {$0$};
\draw [black] (53.4,-4.2) -- (36.9,-4.2);
\fill [black] (36.9,-4.2) -- (37.7,-4.7) -- (37.7,-3.7);
\draw (45.15,-3.7) node [above] {$0$};
\draw [black] (52.4,-9.1) -- (54.5,-6.52);
\draw (50,-9.59) node [below] {$\text{{\small Start}}$};
\fill [black] (54.5,-6.52) -- (53.61,-6.83) -- (54.38,-7.46);
\draw [black] (33.9,-7.2) -- (33.9,-10.4);
\fill [black] (33.9,-10.4) -- (34.4,-9.6) -- (33.4,-9.6);
\draw (33.4,-8.8) node [left] {$0$};
\draw [black] (33.9,-16.4) -- (33.9,-20.1);
\fill [black] (33.9,-20.1) -- (34.4,-19.3) -- (33.4,-19.3);
\draw (33.4,-18.25) node [left] {$1$};
\draw [black] (31.51,-21.29) -- (23.49,-15.21);
\fill [black] (23.49,-15.21) -- (23.83,-16.09) -- (24.43,-15.3);
\draw (28.5,-17.75) node [above] {$1$};
\draw [black] (18.93,-11.32) -- (13.67,-6.28);
\fill [black] (13.67,-6.28) -- (13.9,-7.19) -- (14.59,-6.47);
\draw (17.32,-8.32) node [above] {$1$};
\draw [black] (30.9,-4.2) -- (14.5,-4.2);
\fill [black] (14.5,-4.2) -- (15.3,-4.7) -- (15.3,-3.7);
\draw (22.7,-3.7) node [above] {$1$};
\draw [black] (8.82,-5.523) arc (-36:-324:2.25);
\draw (4.25,-4.2) node [left] {$1$};
\fill [black] (8.82,-2.88) -- (8.47,-2) -- (7.88,-2.81);
\draw [black] (33.9,-26.1) -- (33.9,-49.1);
\fill [black] (33.9,-49.1) -- (34.4,-48.3) -- (33.4,-48.3);
\draw (34.4,-37.6) node [right] {$0$};
\draw [black] (21.1,-16.4) -- (21.1,-20.1);
\fill [black] (21.1,-20.1) -- (21.6,-19.3) -- (20.6,-19.3);
\draw (20.6,-18.25) node [left] {$0$};
\draw [black] (11.5,-7.2) -- (11.5,-20.1);
\fill [black] (11.5,-20.1) -- (12,-19.3) -- (11,-19.3);
\draw (11,-13.65) node [left] {$0$};
\draw [black] (11.5,-26.1) -- (11.5,-36.1);
\fill [black] (11.5,-36.1) -- (12,-35.3) -- (11,-35.3);
\draw (11,-31.1) node [left] {$0$};
\draw [black] (21.1,-26.1) -- (21.1,-36.1);
\fill [black] (21.1,-36.1) -- (21.6,-35.3) -- (20.6,-35.3);
\draw (20.6,-31.1) node [left] {$0$};
\draw [black] (11.5,-42.1) -- (11.5,-53.2);
\fill [black] (11.5,-53.2) -- (12,-52.4) -- (11,-52.4);
\draw (11,-47.65) node [left] {$1$};
\draw [black] (19.63,-41.72) -- (12.97,-53.58);
\fill [black] (12.97,-53.58) -- (13.8,-53.13) -- (12.92,-52.64);
\draw (15.64,-46.44) node [left] {$1$};
\draw [black] (8.82,-40.423) arc (-36:-324:2.25);
\draw (4.25,-39.1) node [left] {$0$};
\fill [black] (8.82,-37.78) -- (8.47,-36.9) -- (7.88,-37.71);
\draw [black] (8.82,-57.523) arc (324:36:2.25);
\draw (4.25,-56.2) node [left] {$0$};
\fill [black] (8.82,-54.88) -- (8.47,-54) -- (7.88,-54.81);
\draw [black] (24.557,-47.555) arc (-23.90035:-311.90035:2.25);
\draw (19.55,-47.47) node [left] {$0$};
\fill [black] (24,-44.97) -- (23.47,-44.19) -- (23.07,-45.1);
\draw [black] (31.69,-50.08) -- (29.11,-47.72);
\fill [black] (29.11,-47.72) -- (29.37,-48.63) -- (30.04,-47.9);
\draw (31.41,-48.41) node [above] {$0$};
\draw [black] (22.58,-25.71) -- (25.42,-30.69);
\fill [black] (25.42,-30.69) -- (25.46,-29.75) -- (24.59,-30.24);
\draw (24.66,-26.98) node [right] {$1$};
\draw [black] (14,-24.76) -- (24.4,-31.64);
\fill [black] (24.4,-31.64) -- (24.01,-30.78) -- (23.46,-31.62);
\draw (18.2,-28.7) node [below] {$1$};
\draw [black] (26.9,-42.7) -- (26.9,-36.3);
\fill [black] (26.9,-36.3) -- (26.4,-37.1) -- (27.4,-37.1);
\draw (27.4,-39.5) node [right] {$1$};
\draw [black] (30.95,-52.64) -- (14.45,-55.66);
\fill [black] (14.45,-55.66) -- (15.33,-56.01) -- (15.15,-55.02);
\draw (22.2,-53.56) node [above] {$1$};
\draw [black] (36.15,-15.38) -- (42.65,-21.12);
\fill [black] (42.65,-21.12) -- (42.38,-20.21) -- (41.72,-20.96);
\draw (40.41,-17.76) node [above] {$0$};
\draw [black] (56.4,-7.2) -- (56.4,-53.2);
\fill [black] (56.4,-53.2) -- (56.9,-52.4) -- (55.9,-52.4);
\draw (55.9,-30.2) node [left] {$1$};
\draw [black] (29.27,-35.14) -- (54.03,-54.36);
\fill [black] (54.03,-54.36) -- (53.7,-53.47) -- (53.09,-54.26);
\draw (40.64,-45.25) node [below] {$1$};
\draw [black] (26.368,-30.359) arc (217.98032:-70.01968:2.25);
\draw (29.57,-26.26) node [above] {$0$};
\fill [black] (28.91,-31.09) -- (29.85,-30.99) -- (29.23,-30.2);
\draw [black] (14.5,-56.2) -- (53.4,-56.2);
\fill [black] (53.4,-56.2) -- (52.6,-55.7) -- (52.6,-56.7);
\draw (33.95,-56.7) node [below] {$1$};
\draw [black] (59.08,-54.877) arc (144:-144:2.25);
\draw (63.65,-56.2) node [right] {$0$};
\fill [black] (59.08,-57.52) -- (59.43,-58.4) -- (60.02,-57.59);
\draw [black] (45.88,-25.93) -- (55.42,-53.37);
\fill [black] (55.42,-53.37) -- (55.63,-52.45) -- (54.68,-52.77);
\draw (51.41,-38.91) node [right] {$1$};
\draw [black] (44.9,-26.1) -- (44.9,-36.7);
\fill [black] (44.9,-36.7) -- (45.4,-35.9) -- (44.4,-35.9);
\draw (44.4,-31.4) node [left] {$0$};
\draw [black] (54.68,-53.74) -- (46.62,-42.16);
\fill [black] (46.62,-42.16) -- (46.66,-43.1) -- (47.48,-42.53);
\draw (51.25,-46.59) node [right] {$1$};
\draw [black] (42.026,-38.881) arc (281.83529:-6.16471:2.25);
\draw (39.36,-33.71) node [left] {$0,\!1$};
\fill [black] (43.8,-36.92) -- (44.13,-36.04) -- (43.15,-36.24);
\draw [black] (18.181,-38.459) arc (285.34019:-2.65981:2.25);
\draw (15.27,-33.46) node [left] {$0$};
\fill [black] (19.83,-36.39) -- (20.1,-35.49) -- (19.14,-35.75);
\node at (35,-61) {Automaton of $\sigma_{\mathrm{T},3}$};
\end{tikzpicture}
&
\begin{tikzpicture}[scale=0.13]
\tikzstyle{every node}+=[inner sep=0pt]
\draw [black] (56.4,-4.2) circle (3);
\draw (56.4,-4.2) node {$0$};
\draw [black] (33.9,-4.2) circle (3);
\draw (33.9,-4.2) node {$0$};
\draw [black] (33.9,-13.4) circle (3);
\draw (33.9,-13.4) node {$0$};
\draw [black] (33.9,-23.1) circle (3);
\draw (33.9,-23.1) node {$1$};
\draw [black] (21.1,-13.4) circle (3);
\draw (21.1,-13.4) node {$0$};
\draw [black] (11.5,-4.2) circle (3);
\draw (11.5,-4.2) node {$1$};
\draw [black] (11.5,-23.1) circle (3);
\draw (11.5,-23.1) node {$1$};
\draw [black] (21.1,-23.1) circle (3);
\draw (21.1,-23.1) node {$0$};
\draw [black] (33.9,-52.1) circle (3);
\draw (33.9,-52.1) node {$1$};
\draw [black] (11.5,-39.1) circle (3);
\draw (11.5,-39.1) node {$1$};
\draw [black] (21.1,-39.1) circle (3);
\draw (21.1,-39.1) node {$0$};
\draw [black] (11.5,-56.2) circle (3);
\draw (11.5,-56.2) node {$1$};
\draw [black] (26.9,-45.7) circle (3);
\draw (26.9,-45.7) node {$1$};
\draw [black] (26.9,-33.3) circle (3);
\draw (26.9,-33.3) node {$0$};
\draw [black] (56.4,-56.2) circle (3);
\draw (56.4,-56.2) node {$1$};
\draw [black] (44.9,-23.1) circle (3);
\draw (44.9,-23.1) node {$0$};
\draw [black] (44.9,-39.7) circle (3);
\draw (44.9,-39.7) node {$0$};
\draw [black] (53.4,-4.2) -- (36.9,-4.2);
\fill [black] (36.9,-4.2) -- (37.7,-4.7) -- (37.7,-3.7);
\draw (45.15,-3.7) node [above] {$0$};
\draw [black] (52.4,-9.1) -- (54.5,-6.52);
\draw (50,-9.59) node [below] {$\text{{\small Start}}$};
\fill [black] (54.5,-6.52) -- (53.61,-6.83) -- (54.38,-7.46);
\draw [black] (33.9,-7.2) -- (33.9,-10.4);
\fill [black] (33.9,-10.4) -- (34.4,-9.6) -- (33.4,-9.6);
\draw (33.4,-8.8) node [left] {$0$};
\draw [black] (33.9,-16.4) -- (33.9,-20.1);
\fill [black] (33.9,-20.1) -- (34.4,-19.3) -- (33.4,-19.3);
\draw (33.4,-18.25) node [left] {$1$};
\draw [black] (31.51,-21.29) -- (23.49,-15.21);
\fill [black] (23.49,-15.21) -- (23.83,-16.09) -- (24.43,-15.3);
\draw (28.5,-17.75) node [above] {$1$};
\draw [black] (18.93,-11.32) -- (13.67,-6.28);
\fill [black] (13.67,-6.28) -- (13.9,-7.19) -- (14.59,-6.47);
\draw (17.32,-8.32) node [above] {$1$};
\draw [black] (30.9,-4.2) -- (14.5,-4.2);
\fill [black] (14.5,-4.2) -- (15.3,-4.7) -- (15.3,-3.7);
\draw (22.7,-3.7) node [above] {$1$};
\draw [black] (8.82,-5.523) arc (-36:-324:2.25);
\draw (4.25,-4.2) node [left] {$1$};
\fill [black] (8.82,-2.88) -- (8.47,-2) -- (7.88,-2.81);
\draw [black] (33.9,-26.1) -- (33.9,-49.1);
\fill [black] (33.9,-49.1) -- (34.4,-48.3) -- (33.4,-48.3);
\draw (34.4,-37.6) node [right] {$0$};
\draw [black] (21.1,-16.4) -- (21.1,-20.1);
\fill [black] (21.1,-20.1) -- (21.6,-19.3) -- (20.6,-19.3);
\draw (20.6,-18.25) node [left] {$0$};
\draw [black] (11.5,-7.2) -- (11.5,-20.1);
\fill [black] (11.5,-20.1) -- (12,-19.3) -- (11,-19.3);
\draw (11,-13.65) node [left] {$0$};
\draw [black] (11.5,-26.1) -- (11.5,-36.1);
\fill [black] (11.5,-36.1) -- (12,-35.3) -- (11,-35.3);
\draw (11,-31.1) node [left] {$0$};
\draw [black] (21.1,-26.1) -- (21.1,-36.1);
\fill [black] (21.1,-36.1) -- (21.6,-35.3) -- (20.6,-35.3);
\draw (20.6,-31.1) node [left] {$0$};
\draw [black] (11.5,-42.1) -- (11.5,-53.2);
\fill [black] (11.5,-53.2) -- (12,-52.4) -- (11,-52.4);
\draw (11,-47.65) node [left] {$1$};
\draw [black] (19.63,-41.72) -- (12.97,-53.58);
\fill [black] (12.97,-53.58) -- (13.8,-53.13) -- (12.92,-52.64);
\draw (15.64,-46.44) node [left] {$1$};
\draw [black] (8.82,-40.423) arc (-36:-324:2.25);
\draw (4.25,-39.1) node [left] {$0$};
\fill [black] (8.82,-37.78) -- (8.47,-36.9) -- (7.88,-37.71);
\draw [black] (8.82,-57.523) arc (324:36:2.25);
\draw (4.25,-56.2) node [left] {$0$};
\fill [black] (8.82,-54.88) -- (8.47,-54) -- (7.88,-54.81);
\draw [black] (24.557,-47.555) arc (-23.90035:-311.90035:2.25);
\draw (19.55,-47.47) node [left] {$0$};
\fill [black] (24,-44.97) -- (23.47,-44.19) -- (23.07,-45.1);
\draw [black] (31.69,-50.08) -- (29.11,-47.72);
\fill [black] (29.11,-47.72) -- (29.37,-48.63) -- (30.04,-47.9);
\draw (31.41,-48.41) node [above] {$0$};
\draw [black] (22.58,-25.71) -- (25.42,-30.69);
\fill [black] (25.42,-30.69) -- (25.46,-29.75) -- (24.59,-30.24);
\draw (24.66,-26.98) node [right] {$1$};
\draw [black] (14,-24.76) -- (24.4,-31.64);
\fill [black] (24.4,-31.64) -- (24.01,-30.78) -- (23.46,-31.62);
\draw (18.2,-28.7) node [below] {$1$};
\draw [black] (26.9,-42.7) -- (26.9,-36.3);
\fill [black] (26.9,-36.3) -- (26.4,-37.1) -- (27.4,-37.1);
\draw (27.4,-39.5) node [right] {$1$};
\draw [black] (30.95,-52.64) -- (14.45,-55.66);
\fill [black] (14.45,-55.66) -- (15.33,-56.01) -- (15.15,-55.02);
\draw (22.2,-53.56) node [above] {$1$};
\draw [black] (36.15,-15.38) -- (42.65,-21.12);
\fill [black] (42.65,-21.12) -- (42.38,-20.21) -- (41.72,-20.96);
\draw (40.41,-17.76) node [above] {$0$};
\draw [black] (56.4,-7.2) -- (56.4,-53.2);
\fill [black] (56.4,-53.2) -- (56.9,-52.4) -- (55.9,-52.4);
\draw (55.9,-30.2) node [left] {$1$};
\draw [black] (29.27,-35.14) -- (54.03,-54.36);
\fill [black] (54.03,-54.36) -- (53.7,-53.47) -- (53.09,-54.26);
\draw (40.64,-45.25) node [below] {$1$};
\draw [black] (26.368,-30.359) arc (217.98032:-70.01968:2.25);
\draw (29.57,-26.26) node [above] {$0$};
\fill [black] (28.91,-31.09) -- (29.85,-30.99) -- (29.23,-30.2);
\draw [black] (14.5,-56.2) -- (53.4,-56.2);
\fill [black] (53.4,-56.2) -- (52.6,-55.7) -- (52.6,-56.7);
\draw (33.95,-56.7) node [below] {$1$};
\draw [black] (59.08,-54.877) arc (144:-144:2.25);
\draw (63.65,-56.2) node [right] {$0$};
\fill [black] (59.08,-57.52) -- (59.43,-58.4) -- (60.02,-57.59);
\draw [black] (45.88,-25.93) -- (55.42,-53.37);
\fill [black] (55.42,-53.37) -- (55.63,-52.45) -- (54.68,-52.77);
\draw (51.41,-38.91) node [right] {$1$};
\draw [black] (44.9,-26.1) -- (44.9,-36.7);
\fill [black] (44.9,-36.7) -- (45.4,-35.9) -- (44.4,-35.9);
\draw (44.4,-31.4) node [left] {$0$};
\draw [black] (54.68,-53.74) -- (46.62,-42.16);
\fill [black] (46.62,-42.16) -- (46.66,-43.1) -- (47.48,-42.53);
\draw (51.25,-46.59) node [right] {$1$};
\draw [black] (42.026,-38.881) arc (281.83529:-6.16471:2.25);
\draw (39.36,-33.71) node [left] {$0,\!1$};
\fill [black] (43.8,-36.92) -- (44.13,-36.04) -- (43.15,-36.24);
\draw [black] (18.182,-38.458) arc (285.32151:-2.67849:2.25);
\draw (15.27,-33.46) node [left] {$0$};
\fill [black] (19.83,-36.39) -- (20.11,-35.49) -- (19.14,-35.75);
\node at (35,-61) {Automaton of $\sigma_{\mathrm{T},4}$};
\end{tikzpicture}
\\
\hline \\
\end{tabular} 
\caption{Automata corresponding to the order $4$, break sequence $(1,3)$ series in Table \ref{s5-s8}.} 
\label{sigmatau2}
\end{table}
\end{center}

\section{Construction and classification of some order-$8$ elements} 
\label{section5}

\subsection{Order $8$, break sequence $(1, 3, 11) = \langle 1, 2, 4 \rangle$} 

Up to now, no finite description of any element of $\No(\F_2)$ of order $8$ was known. Our method produces an example. 

\begin{proposition} \label{8class} An element $\sigma_8$ in $\No(\F_2)$ of order $8$ with break sequence $(1, 3, 11) = \langle 1, 2, 4 \rangle$ is given by the automaton described by the data in Table \textup{\ref{8data}:} it has $320$ states, corresponding to the $320$ triples on the displayed ordered list, where the start vertex is the first triple on the list and a triple $(\ell,i,j)$ occurs on the list precisely if the following three conditions hold\textup{:} its vertex label is $\ell$, there is a directed edge with label $0$ to the $i$-th triple on the list and there is a directed edge with label $1$ to the $j$-th triple on the list. The initial terms of $\us_8$ are 
$$\us_8 = t+t^2+t^5+t^6+t^{12}+ O(t^{13}). $$
\end{proposition}

We refrain from including a pictorial representation, but the automaton is stored in standard Mathematica form in \cite{Database}, making it easy to manipulate. 

\begin{proof} We refer to Example \ref{re2} on how to use Witt vectors of length $3$ to construct cyclic order-$8$ extensions. 
We choose $\beta=(z^{-1},0,0) \in W_3(\F_2\lau{z})$ and rewrite the resulting equations in (\ref{re2eq}) in terms of the variables $x:=\alpha_0$, $y:=\alpha_1$ and $w:=\alpha_2+\alpha_0^2 \alpha_1$ to find 
\begin{equation*} 
\left\{ \begin{array}{l}
x^2+x = z^{-1}; \\ 
y^2+y=x z^{-1}; \\   
w^2+w = x^4y+x^3y. 
\end{array} \right.
\end{equation*}
Choosing uniformisers $z_0=z,z_1,z_2,z_3$ for the intermediate fields in the tower of field extensions $$K_0=\F_2\lau{z} \subsetneq K_1=K_0(x) = \F_2\lau{z_1}\subsetneq K_2 = K_1(y) = \F_2\lau{z_2} \subsetneq K_3= K_2(w) = \F_2\lau{z_3} 
$$ 
and using Lemma \ref{ramlem} as in Example \ref{exrun}, we see that the extension $K_3/K_0$ is totally ramified.
We find the relevant valuations (following Lemma \ref{ramlem}): 
\begin{align*} & v_{z_1}(x)=-1, \  v_{z_1}(z)=2; \\ &  v_{z_2}(y)=-3; \  v_{z_2}(x)=-2, \ v_{z_2}(z)=4;  \\ &  
v_{z_3}(w)=-11, v_{z_3}(x)=-4, \ v_{z_3}(y)=-6, \ v_{z_3}(z)=8. \end{align*}        
We choose the uniformiser $t$ as $t=(w+y)/(x^3+y)$. Then indeed $v_{z_3}(t)=1$, and the action of the generator of the Galois group  on $\alpha_i$ is given by (\ref{w8}), which implies that for our choice of variables we have 
\begin{equation}
\left\{ \begin{array}{l}
\sigma(x)=x+1; \\ 
\sigma(y)=y+x; \\ 
\sigma(w)=w+xy+y,
\end{array} \right.
\end{equation}
and so by elimination we find the (irreducible) equation 
\begin{equation*} \label{eq8} 
t^6 X^6 + (t^6+t^2) X^4 + (t^6+t^5+t^4+t^3+t^2+1) X^2 + (t+1)^3 X + t^6 + t^5 + t^2 + t = 0
\end{equation*}
for $\us=\us_8$. 
The initial coefficients are as indicated, and we readily verify the lower break sequence $(1,3,11)$ from 
\begin{equation*} \us_{8} = t + t^2 + O(t^3),\ \us_{8}^{\circ 2} = t+t^4+ O(t^5),\ \us_{8}^{\circ 4} = t+t^{12} + O(t^{13}). \tag*{\qedhere} 
\end{equation*}
\end{proof} 

\subsection{Detecting conjugacy using local class field theory} 

\begin{proposition} \label{conj8}  The number of conjugacy classes of elements of order $8$ in $\No(\F_2)$ with `minimal' break sequence $(1, 3, 11) = \langle 1, 2, 4 \rangle$ is $4$.
\end{proposition}
\begin{proof}
We follow the method of Lubin \cite{Lubin}. For $k\geq 1$, write $U_k$ for the multiplicative group of units $U_k=1 + z^k \F_2\fl z \fr$. By \cite[Thm.\ 2.2]{Lubin}  elements of exact order $2^n$ in $\No(\F_2)$  up to conjugation correspond bijectively to continuous surjective characters $\eta \colon U_1 \rightarrow {\Z}/2^n{\Z}$ up to so-called strict equivalence (the bijection arises from the restriction of the local reciprocity map to $U_1$). Strict equivalence  of characters $\eta$ and $\eta'$ means that there exists $u \in \No(\F_2)$ with $\eta(u(z)/z)=0$ and $\eta'(x) = \eta(x \circ u)$ for all $x \in U_1$. Moreover, the upper break sequence $\langle b^{(0)},\ldots,b^{(n-1)}\rangle$ can be read off from the corresponding character: $\eta(U_{b^{(i)}}) = 2^{i}{\Z}{/}2^n{\Z}$ and $\eta(U_{b^{(i)}+1}) = 2^{i+1}{\Z}{/}2^n{\Z}$  \cite[Prop.\ 3.2]{Lubin}.

In our case of order $8$ elements with minimal break sequence, this implies that the corresponding characters factor through $U_1/U_5$ and map $U_3$ to $4{\Z}/8{\Z}$. Since we have an isomorphism of groups  \begin{align*} {\Z}/8{\Z} \times {\Z}{/}{2{\Z}} &\to U_1/U_5\\ (c,d) &\mapsto (1+z)^c (1+z^3)^d U_5,\end{align*}  there are eight such characters $\eta_{a,b}$ determined by $\eta_{a,b}(1+z) = a \in({\Z}/8{\Z})^*$ and $\eta_{a,b}(1+z^3)=4b$ with $b\in{\Z}/2{\Z}$. We need to determine which of these are strictly equivalent. Write any $u\in \No(\F_2)$ in the form $u(z) = z(1+z)^{\alpha}(1+z^3)^{\beta}u_5$ with $\alpha \in \{0,\dots,7\}, \beta \in \{0,1\}$ and $u_5 \in U_5$. We have $\eta_{a,b}(u(z)/z)=a\alpha+4b\beta \mbox{ mod } 8$, and hence $\eta_{a,b}(u(z)/z)=0$ if and only if $u(z) \equiv z \bmod z^6$ or $u(z) \equiv z+z^4+bz^5 \bmod z^6$. Suppose then that
\begin{equation} \label{eta} \eta_{a',b'}(x) = \eta_{a,b}(x \circ u), \end{equation} and evaluate both sides for $x=1+z$ and $x=1+z^3$, respectively. For the first choice of $u$, we immediately find that $a'=a$ and $b'=b$. For the second choice of $u$, for $x=1+z$, the left hand side of (\ref{eta}) evaluates to $a'$ and the right hand side to $\eta_{a,b}((1+z)^5U_5) = 5a$. For $x=1+z^3$, the left hand side is $b'$ and the right hand side $\eta_{a,b}((1+z^3)U_5)=b$.  

We conclude that the strict equivalence class of $\eta_{a,b}$ consists of $\eta_{a,b}$ and $\eta_{5a,b}$, and there are indeed four strict equivalence classes in total. 
\end{proof}

\begin{table} \label{8data}
\hrule
{\footnotesize 
\begin{align*}
& ((0, 2, 3), (0, 58, 59), (1, 82, 185), (0, 5, 3), (0, 65, 66), (0, 7, 8), (0, 136, 137), (1, 278, 43), (0, 10, 11), (0, 140, 141), \\ 
& (1, 281, 43), (0, 13, 8), (0, 147, 38), (0, 15, 11), (0, 151, 152), (0, 17, 18), (0, 76, 77), (1, 279, 117), (0, 20, 18), \\ 
& (0, 78, 79), (0, 22, 23), (0, 60, 61), (1, 280, 117), (0, 25, 23), (0, 70, 72), (0, 9, 27), (1, 89, 190), (0, 24, 27), (0, 30, 31), \\
& (0, 44, 41), (1, 87, 190), (0, 32, 31), (0, 32, 34), (1, 72, 189), (0, 33, 36), (1, 224, 160), (0, 33, 38), (1, 214, 154), \\ 
& (0, 40, 41), (0, 51, 109), (1, 84, 185), (0, 43, 3), (0, 115, 116), (0, 96, 101), (0, 46, 3), (0, 80, 68), (0, 35, 48), \\
& (1, 272, 112), (0, 37, 50), (1, 290, 45), (0, 35, 8), (0, 37, 53), (1, 282, 39), (0, 55, 18), (0, 99, 100), (0, 57, 27), \\
& (0, 142, 143), (0, 60, 93), (1, 238, 128), (0, 60, 106), (1, 236, 87), (0, 63, 64), (0, 58, 112), (1, 265, 48), (0, 151, 129), \\
& (1, 242, 66), (0, 65, 68), (1, 260, 59), (0, 70, 71), (0, 151, 179), (1, 293, 305), (1, 231, 97), (0, 74, 75), (0, 95, 199), \\
& (1, 232, 97), (0, 51, 26), (1, 268, 48), (0, 44, 192), (1, 267, 143), (0, 81, 82), (0, 195, 154), (1, 256, 143), (0, 81, 84), \\
& (1, 246, 273), (0, 86, 87), (0, 195, 162), (1, 273, 34), (0, 86, 89), (1, 149, 220), (0, 91, 92), (0, 47, 12), (1, 234, 87), \\
& (0, 94, 89), (0, 111, 107), (0, 96, 97), (0, 125, 123), (1, 23, 34), (0, 99, 79), (0, 125, 128), (1, 251, 273), (1, 222, 220), \\
& (0, 103, 101), (0, 52, 114), (0, 105, 89), (0, 85, 89), (0, 4, 107), (1, 3, 221), (0, 54, 109), (1, 221, 221), (0, 56, 107), \\
& (0, 129, 130), (0, 113, 114), (0, 138, 139), (1, 262, 303), (0, 131, 132), (1, 266, 303), (0, 118, 116), (0, 148, 149), \\
& (0, 120, 114), (0, 150, 50), (0, 62, 68), (0, 73, 123), (1, 264, 59), (0, 90, 123), (0, 126, 127), (0, 126, 75), (1, 296, 305), \\
& (1, 240, 66), (0, 153, 28), (1, 71, 189), (0, 153, 42), (1, 215, 154), (0, 131, 134), (1, 225, 160), (0, 129, 27), (0, 55, 164), \\
& (1, 213, 156), (0, 98, 172), (1, 288, 42), (0, 63, 16), (1, 291, 45), (0, 63, 6), (1, 270, 112), (0, 140, 145), (1, 283, 39), \\
& (0, 142, 11), (0, 52, 122), (0, 193, 187), (1, 212, 156), (0, 49, 104), (0, 148, 194), (1, 289, 42), (0, 58, 124), (0, 155, 121), \\
& (0, 95, 203), (0, 157, 121), (0, 33, 206), (0, 159, 46), (0, 69, 179), (0, 161, 46), (0, 144, 182), (0, 163, 122), (0, 47, 19), \\
& (0, 165, 122), (0, 99, 21), (0, 167, 124), (0, 133, 106), (0, 169, 124), (0, 83, 40), (0, 171, 26), (0, 176, 29), (0, 173, 28), \\
& (0, 65, 46), (0, 175, 26), (0, 184, 26), (0, 44, 189), (0, 178, 32), (0, 142, 168), (0, 86, 32), (0, 181, 29), (0, 210, 192), \\
& (0, 183, 29), (0, 211, 186), (0, 51, 45), (0, 154, 186), (0, 191, 192), (0, 129, 188), (0, 194, 192), (0, 179, 188), \\
& (0, 162, 186), (0, 198, 164), (0, 187, 193), (0, 193, 193), (0, 205, 166), (0, 148, 191), (0, 197, 162), (0, 67, 194), \\ 
& (0, 98, 177), (0, 200, 164), (0, 37, 180), (0, 202, 162), (0, 146, 208), (0, 204, 166), (0, 126, 108), (0, 49, 110), \\
& (0, 207, 168), (0, 135, 16), (0, 209, 168), (0, 88, 24), (0, 55, 156), (0, 52, 119), (1, 319, 100), (1, 294, 313), \\ 
& (1, 310, 71), (1, 316, 308), (1, 212, 164), (1, 218, 172), (1, 319, 79), (1, 218, 177), (1, 68, 187), (1, 66, 187), \\
& (1, 223, 158), (1, 318, 36), (1, 311, 145), (1, 304, 284), (1, 227, 203), (1, 297, 97), (1, 227, 199), (1, 230, 192), \\
& (1, 297, 101), (1, 230, 189), (1, 233, 190), (1, 233, 31), (1, 235, 154), (1, 244, 72), (1, 237, 160), (1, 241, 141), \\
& (1, 239, 196), (1, 271, 122), (1, 241, 170), (1, 257, 16), (1, 243, 129), (1, 248, 194), (1, 243, 179), (1, 246, 162), \\
& (1, 248, 191), (1, 246, 154), (1, 249, 187), (1, 249, 193), (1, 216, 199), (1, 252, 201), (1, 258, 42), (1, 217, 172), \\
& (1, 255, 174), (1, 277, 104), (1, 257, 6), (1, 260, 112), (1, 260, 124), (1, 258, 28), (1, 261, 93), (1, 261, 106), \\
& (1, 263, 12), (1, 292, 26), (1, 244, 44), (1, 228, 102), (1, 229, 14), (1, 247, 9), (1, 269, 104), (1, 294, 12), (1, 302, 222), \\
& (1, 317, 53), (1, 312, 134), (1, 271, 119), (1, 275, 119), (1, 218, 286), (1, 277, 110), (1, 317, 50), (1, 309, 84), \\
& (1, 277, 306), (1, 320, 127), (1, 315, 314), (1, 226, 128), (1, 295, 303), (1, 285, 30), (1, 245, 87), (1, 287, 30), \\
& (1, 258, 307), (1, 249, 221), (1, 259, 130), (1, 276, 48), (1, 219, 274), (1, 223, 8), (1, 292, 45), (1, 223, 48), \\
& (1, 294, 19), (1, 297, 39), (1, 280, 123), (1, 299, 112), (1, 252, 132), (1, 301, 43), (1, 247, 82), (1, 242, 68), \\
& (1, 250, 128), (1, 244, 71), (1, 152, 130), (1, 298, 307), (1, 284, 130), (1, 300, 307), (1, 245, 89), (1, 247, 84), \\
& (1, 241, 145), (1, 256, 11), (1, 253, 274), (1, 254, 274), (1, 259, 27), (1, 252, 134), (1, 318, 38), (1, 233, 34), \\
& (1, 280, 128), (1, 320, 75))
\end{align*} 
\hrule 
}
\caption{Representation of the automaton for the power series $\us_8$ of order $8$ with break sequence $(1,3,11)$.
} 
\end{table}

We state below an analogue of Lemma \ref{recognise} that allows us to distinguish between these four conjugacy classes based on the first few coefficients of the power series.

\begin{proposition} \label{recognise2} Let $\sigma \in \No(\F_2)$ be an automorphism of order $8$ with break sequence $(1,3,11)=\langle 1,2,4 \rangle$, and write $\us=\sum_{i=1}^{\infty} a_i t^i$ with $a_i\in\F_2$. Then $a_1=a_2=1$, $a_3=0$, $a_5\neq a_7$, and $\us$ is conjugate to a series $\sigma_{8,(b_4,b_{11})}$ of order $8$ that has initial coefficients  \begin{equation*} \label{b4b11} \sigma_{8,(b_4,b_{11})} = t+t^2+b_4t^4+t^7+b_{11} t^{11}+O(t^{12})
\end{equation*} for a unique choice of $b_4, b_{11} \in \F_2$. In particular, the conjugacy class of $\sigma$ depends only on $\sigma \bmod t^{12}$.

The series $\sigma_8$ is conjugate to $\sigma_{8,(1,1)}$ and $\sigma_8^{\circ 3}$ is conjugate to $\sigma_{8,(0,1)}$. These give representatives of two of the four conjugacy classes of minimally ramified series of order $8$. 
\end{proposition}
\begin{proof}  We will show that any such $\sigma$ is conjugate to some $\sigma_{8,(b_4,b_{11})}$ modulo $t^{12}$, and that the series $\sigma_{8,(b_4,b_{11})}$ are not conjugate modulo $t^{12}$ for the four different choices of $(b_4,b_{11})$. Since we know that there are $4$ conjugacy classes of series $\sigma$ satisfying the required assumptions, this shows that actual series $\sigma_{8,(b_4,b_{11})}$ of order $8$ with minimal break sequence do exist. 

We first note that $d(\sigma)=1$ implies $a_1=a_2=1$; computing $\sigma^{\circ 2}$, we get $\sigma^{\circ 2}=t+(1+a_3)t^4 + O(t^5)$, and $d(\sigma^{\circ 2})=3$ gives $a_3=0$; finally, $\sigma^{\circ 4}=t+(a_5+a_7)t^{12} + O(t^{13})$, and since $d(\sigma^{\circ 4})=11$, we get $a_5\neq a_7$. 

We will now prove that $\sigma$ is conjugate to $\sigma_{8,(b_4,b_{11})}$ for some  $b_4, b_{11} \in \F_2$. We do this by conjugating with selected elements of $\No(\F_2)$ in the following steps (in each step the symbols $a_i$ denote the coefficients of the `new' power series, obtained by performing the conjugations described in the previous steps):

{\bf Step I} (conjugating with $\chi_3\colon t \mapsto t+t^3$). We have $\chi_3^{\circ -1}=t+t^3+t^5+t^9+t^{11}+O(t^{12})$, yielding $$\chi_3 \circ \sigma \circ \chi_3^{\circ -1}=t+t^2+(1+a_4)t^4+(1+a_5)t^5+O(t^6),$$ so conjugating if necessary by $\chi_3$ we may and do assume that $a_5=0$; then $a_7=1$, since $a_5\neq a_7$.

{\bf Step II} (conjugating with $\chi_5\colon t \mapsto t+t^5$). We have $\chi_5^{\circ -1}=t+t^5+t^9+O(t^{12})$, yielding $$\chi_5 \circ \sigma \circ \chi_5^{\circ -1}=t+t^2+a_4t^4+(1+a_6)t^6+O(t^7),$$ 
so conjugating if necessary by $\chi_5$ we may and do assume that $a_6=0$.

{\bf Step III} (conjugating with $\chi_2\colon t \mapsto t+t^2$). We have $\chi_2^{\circ -1}=t+t^2+t^4+t^8+O(t^{12})$, yielding $$\chi_2 \circ \sigma \circ \chi_2^{\circ -1}=t+t^2+a_4t^4+t^7+(1+a_8)t^8+(1+a_9)t^9+(a_9+a_{10})t^{10}+(1+a_{11})t^{11} +O(t^{12}),$$ so conjugating if necessary by $\chi_2$ we may and do assume that $a_9=0$.

{\bf Step IV} (conjugating with $\chi_6\colon t \mapsto t+t^6$). We have $\chi_6^{\circ -1}=t+t^6+O(t^{12})$, yielding $$\chi_6 \circ \sigma \circ \chi_6^{\circ -1}=t+t^2+a_4t^4+t^7+(1+a_8)t^8+(1+a_{10})t^{10}+a_{11}t^{11}+O(t^{12}),$$ so conjugating if necessary by $\chi_6$ we may and do assume that $a_8=0$.

{\bf Step V} (conjugating with $\chi_4\colon t \mapsto t+t^4$). We have $\chi_4^{\circ -1}=t+t^4+O(t^{12})$, yielding $$\chi_4 \circ \sigma \circ \chi_4^{\circ -1}=t+t^2+a_4t^4+t^7+(1+a_{10})t^{10}+a_{11}t^{11}+O(t^{12}),$$ so conjugating if necessary by $\chi_4$ we may and do assume that $a_{10}=0$.

This ends the proof that $\sigma$ is conjugate  to $\sigma_{8,(b_4,b_{11})}$ for some  $b_4, b_{11} \in \F_2$.

We will now prove that the power series $\sigma_{8,(b_4,b_{11})}$ and $\sigma_{8,(c_4,c_{11})}$ are not conjugate in $\No(\F_2)$ unless $(b_4,b_{11})=(c_4,c_{11})$.  Indeed, suppose that $\sigma_{8,(b_4,b_{11})}$ and $\sigma_{8,(c_4,c_{11})}$ are conjugate, and let $\tau \in \No(\F_2)$ be a conjugating power series, so that $\sigma_{8,(b_4,b_{11})}\circ \tau = \tau \circ \sigma_{8,(c_4,c_{11})}$. Write $\tau=t+\sum_{i=2}^{\infty} d_i t^i$. Computing $\sigma_{8,(b_4,b_{11})}\circ \tau - \tau \circ \sigma_{8,(c_4,c_{11})}$, we get \begin{align*}\sigma_{8,(b_4,b_{11})}&\circ \tau - \tau \circ \sigma_{8,(c_4,c_{11})}=(d_3+b_4+c_4)t^4+d_3t^5+(d_5+d_3c_4)t^6 + \\& (d_2+d_6+d_7+d_2b_4+d_2c_4+d_3c_4+d_5c_4)t^8+(d_2+d_5+d_7+d_3c_4)t^9 + \\& (d_2+d_4+d_6+d_7+d_9+d_3c_4+d_7c_4)t^{10}+ 
(d_2+d_2d_3+d_7+b_{11}+c_{11})t^{11}+O(t^{12}).\end{align*}
Considering the coefficients at $t^5$, $t^6$ and $t^9$ gives $d_3=d_5=d_2+d_7=0$; looking then at the coefficients at $t^4$ and $t^{11}$ gives $b_4=c_4$ and $b_{11}=c_{11}$. 

Applying the algorithm from the above proof, we find that $\sigma_8$ is conjugate to $\sigma_{8,(1,1)}$ and $\sigma_8^{\circ 3}$ is conjugate to $\sigma_{8,(0,1)}$. (This requires computing more coefficients than we have specified in Steps I and II, but the computations are easy.)
\end{proof}

\begin{corollary} \label{cor:recognise2} Let $\sigma \in \No(\F_2)$ be an automorphism of order $8$ with break sequence $(1,3,11)=\langle 1,2,4 \rangle$. Then $\sigma$ and $\sigma^{\circ 5}$ are conjugate in $\No(\F_2)$, while $\sigma$ and $\sigma^{\circ 3}$ are not.\end{corollary}
\begin{proof} This follows from the proof of Proposition \ref{conj8}---if an element $\sigma$ corresponds to the character $\eta_{a,b}$, then for $k$ odd the element $\sigma^{\circ k}$ corresponds to $k\eta_{a,b}=\eta_{ka,kb}=\eta_{ka,b}$. Since $\eta_{a,b}$ and $\eta_{5a,b}$ are strictly equivalent, while $\eta_{a,b}$ and $\eta_{3a,b}$ are not, the claim follows.

It is also possible to give a direct proof using the method of Proposition \ref{recognise2}, as follows. Denote the relation of being conjugate by $\sim$. By Proposition \ref{recognise2}, we may assume without loss of generality that $\sigma = t+t^2+b_4t^4+t^7+b_{11} t^{11}+O(t^{12})$ for some $b_4, b_{11}\in \F_2$. Then \[ \sigma^{\circ 2}=t+t^4+t^8+t^9+(1+b_4) t^{10} + t^{11}+O(t^{12}),\qquad \sigma^{\circ 4}=t+O(t^{12}),\] and hence  $\sigma = \sigma^{\circ 5} + O(t^{12})$ and $$\sigma^{\circ 3}  = t+t^2+(1+b_4)t^4+t^7+t^9+b_4t^{10}+(1+b_{11})t^{11}+O(t^{12}).$$ Following the algorithm of the proof of Proposition \ref{recognise2} (and using the notation therein), we may conjugate $\sigma^{\circ 3}$ in turn by $\chi_2$, $\chi_6$ and in the case where $b_4=1$ also $\chi_4$ to arrive at $$\sigma^{\circ 3} \sim t+t^2+(1+b_4)t^4+t^7+b_{11}  t^{11} +O(t^{12}),$$ 
i.e.\ if $\sigma \sim \sigma_{8,(b_4,b_{11})}$, then $\sigma^{\circ 3} \sim  \sigma_{8,(b_4+1,b_{11})}$. Applying Proposition \ref{recognise2} again shows that $\sigma \sim \sigma^{\circ 5}$ and $\sigma \not\sim \sigma^{\circ 3}$.\end{proof}

\subsection{Finding representatives via explicit class field theory}\label{subsecCarlitz}

We have already constructed representatives of two out of four conjugacy classes of minimally ramified series of order $8$. In order to construct the representatives for the remaining conjugacy classes, we will extend the method using the Carlitz module from Remark \ref{Carlitz}.

Let $\rho$ be the Carlitz module for $K=\F_2(z)$. We know from \cite[Obs.\ 4 \& Sect.\ 5]{Lubin} that the characters $\eta \colon U_1 \to {\Z}{/}{8{\Z}}$ corresponding to minimally ramified order-$8$ elements factor through $U_5$, and the corresponding Galois extensions can be obtained as a subextension of  $K(\rho[z^5])/K$. The extension $K(\rho[z^5])/K$ has Galois group $$G = \left( \F_2[z]/z^5\right)^* \cong {\Z}/{8{\Z}} \times {\Z}/{2{\Z}} = \langle z+1\bmod z^5 \rangle \times \langle z^3+1 \bmod z^5 \rangle. $$ The group $G$ has two subgroups with quotient ${\Z}/{8{\Z}}$: $$H_1 = \langle z^3+1 \bmod z^5 \rangle \qquad \text{and} \qquad H_2=\langle z^4+z^3+1 \bmod z^5 \rangle.$$

The field $K(\rho[z^5])$ is generated by a root $\alpha$ of the degree-$16$ polynomial $\rho_{z^5}(X)/\rho_{z^4}(X)$. The fixed fields $L_1$ and $L_2$ of $H_1$ and $H_2$, respectively, are generated by the elements $$\beta_1:=\alpha\cdot \rho_{z^3+1}(\alpha) \qquad \text{and} \qquad \beta_2:=\alpha\cdot \rho_{z^4+z^3+1}(\alpha).$$ 
Recalling that $L_i/K$ has Galois group cyclic of order $8$ generated by $\sigma$ acting as $\sigma(\alpha) = z \alpha + \alpha + \alpha^2$, we can compute $\sigma(\beta_i)$ and we find that  
\begin{align*} 
& \begin{cases}
\beta_1=\alpha^9 + (z^4 + z^2 + z)\alpha^5 + (z^4 + z^3 + z^2)\alpha^3 + (z^3 + 1)\alpha^2; \\
 \sigma(\beta_1) = \alpha^{10} +  (z + 1)\alpha^9 + (z^4 + z^2 + z)\alpha^6 + (z^5 + z^4 + z^3 + z)\alpha^5 +(z^4 + z^3 + z^2 + 1)\alpha^4+ \\ 
 \end{cases} \\[-1mm]
 & \qquad\qquad\qquad   (z^5 + z^3 + z^2)\alpha^3 + (z^4 + z^3 + z^2 + z + 1)\alpha^2 + (z^2 + z)\alpha;
 \end{align*}
and
\begin{align*} 
& \begin{cases}
 \beta_2=\alpha^9 + (z^4 + z^2 + z)\alpha^5 + (z^4 + z^3 + z^2)\alpha^3 + (z^3 + 1)\alpha^2 + z\alpha; \\
 \sigma(\beta_2)=\alpha^{10} + (z + 1)\alpha^9 + (z^4 + z^2 + z)\alpha^6 + (z^5 + z^4 + z^3 + z)\alpha^5 + (z^4 + z^3 + z^2 + 1)\alpha^4 +\\
 \end{cases} \\[-1mm]
 & \qquad\qquad\qquad (z^5 + z^3 + z^2)\alpha^3 + (z^4 + z^3 + z^2 + z + 1)\alpha^2 + (z^2 + z)\alpha.
 \end{align*}
Since $z$ is the only ramified place and it is totally ramified in $K(\rho[z^5])$, the same is true in $L_i$. We can choose $t=\beta_i$  as a uniformiser for the place above $z$ in $L_i$. Elimination of $z$ and $\alpha$ leads to the following equation for the element $\sigma_{8,1}=\sigma_{8,1}(t)$ of order $8$ with $t=\beta_1$: 
\begin{align*} t X^6 +(t+1) X^5 &+   \left(t^5+t^3+t\right) X^4+\left(t^5+t^2+t\right) X^3 +    \\ &  \left(t^6+t^3+t\right)X^2+ t^4 X + t^6+t^5+t^4+t^3=  0; \end{align*}
and to the following equation for the element $\sigma_{8,2}=\sigma_{8,2}(t)$ of order $8$ with $t=\beta_2$: 
\begin{align*}
 t X^6 +(t+1) X^5 &+   \left(t^5+t^3\right) X^4+\left(t^5+t+1\right) X^3 +    \\ &  \left(t^6+t^5+t^4+t^3+t\right)X^2+ \left(t^4 +t^2\right) X + t^4+t^3=  0. 
\end{align*}
These equations define algebraic curves of geometric genus $7$, solved by the series $$\sigma_{8,1}(t)=t+t^2+t^5+t^{11}+O(t^{13}) \qquad \text{and} \qquad \sigma_{8,2}(t)=t+t^2+t^5+t^9+t^{11}+O(t^{13})$$ of order $8$, which are produced by automata with $668$ and $926$ states, respectively. Furthermore, $\sigma_{8,1}$ is conjugate to $\sigma_{8,(1,1)}$ and $\sigma_{8,2}$ is conjugate to $\sigma_{8, (1,0)}$ by the method from Proposition \ref{recognise2}.  
We may summarise the above discussion as follows:

\begin{proposition} \label{sigma81} There are four conjugacy classes of order-$8$ elements with break sequence $(1, 3, 11) = \langle 1, 2, 4 \rangle$ and their representatives are the series $\sigma_{8,1}$, $\sigma_{8,1}^{\circ 3}$ \textup{(}conjugate to $\sigma_8$ and $\sigma_8^{\circ 3}$, respectively\textup{)}, $\sigma_{8,2}$ and $\sigma_{8,2}^{\circ 3}$. The series $\sigma_{8,2}$ may be found in \cite{Database}. \qed
\end{proposition}

\begin{remark}
We have constructed order-$8$ elements by considering the Galois extension $K(\rho[z^5])/K$ with Galois group ${\Z}/8{\Z} \times {\Z}/2{\Z}$, and looking at its subextensions $L_i/K$ with Galois group ${\Z}/8{\Z}$. We could instead look at an extension $K(\rho[z^5])/M$ with Galois group ${\Z}/8{\Z}$. This would work, but would produce a non-minimally ramified series generated by an automaton with many more states---the automaton corresponding to $\sigma(t)=\rho_{1+z}(t)$ with $t=\alpha$ has 136600 states.
\end{remark}

\section{Embedding the Klein four-group in $\No(\F_2)$ using automata} 
Since every $p$-group embeds in $\No(\F_p)$, we may ask for a representation for generators of a given $p$-group through automata. We show how to do this for the easiest case, that of the Klein four-group $V={\Z}/{2}{\Z} \times {\Z}/{2}{\Z}$
for $p=2$, by describing two automata that correspond to two commuting power series of order two in characteristic two (with minimal admissible break sequences), answering a question that Klopsch asked us. 

\subsection{Embedding with small conductor} For a general field $\F$, define the Nottingham group $\No(\F)$ to be the group of power series $\sigma(t)\in \F[\![t]\!]$ of the form $t+O(t^2)$ under composition. The following lemma shows that it is easy to embed $V$ into the Nottingham group over any proper field extension $\F$ of $\F_2$ such that all nontrivial elements of $V$ have break sequence $(1)$ (i.e.\ have depth $1$), but one cannot do so over $\F_2$.

\begin{proposition}\label{thm:Klein1} There is an embedding of the Klein four-group $V={\Z}/{2}{\Z} \times {\Z}/{2}{\Z}$ in the Nottingham group $\No(\F)$ over a field $\F$ of characteristic two with all nontrivial elements of $V$ having break sequence $(1)$ if and only if $\F \neq \F_2$. 
\end{proposition}

Note that all nontrivial elements having break sequence $(1)$ means that the corresponding $V$-extension is \emph{weakly ramified}, i.e.\ has trivial second ramification group. A much more general statement that implies Lemma \ref{thm:Klein1} is given in \cite[Korollar 3.2]{CKCrelle}, but we give a short direct proof.

\begin{proof} Assume $\F \neq \F_2$ and let $U$ be a two-dimensional $\F_2$-vector subspace of $\F$.  Then the power series $t/(ut+1)=t+ut^2+O(t^3)$ taken over $u\in U$ form a subgroup of $\No(\F)$ isomorphic to the Klein four-group. 

For the converse, assume we have an embedding of $V=\{\mathrm{id},\sigma,\tau,\sigma \circ \tau\}$ into $ \No(\F_2)$ with nontrivial elements having break sequence $(1)$. Then $\sigma$ and $\tau$ are of the form $t+t^2+O(t^3)$, implying that $\sigma \circ \tau = t+O(t^3)$, a contradiction. 
\end{proof} 

There are further restrictions on possible depths of elements of the Klein four-group embedded in $\No(\F_2)$. In the next subsection, we will construct an embedding  with nontrivial elements having depths $1, 1$ and $5$. The next lemma shows that these are the minimal possible values.

\begin{proposition}\label{lem:Kleinfourbr} For every embedding of the Klein four-group $V$ in the Nottingham group $\No(\F_2)$ some nontrivial element of $V$ has depth at least $5$.\end{proposition} 
\begin{proof} Suppose the contrary. By Proposition \ref{thm:Klein1} some nontrivial element has  depth at least $2$. Every element of finite order has odd depth:  if $\sigma$ had even depth, writing $\sigma=t+t^k+O(t^{k+1})$ with $k$ odd, we would find by induction that $\sigma^{\circ 2^n}=t+t^{2^n(k-1)+1}+O(t^{2^n(k-1)+2})$ for all $n\geq 1$,  so $\sigma$ would not be of finite order.  
Also note that for every $k\geq 1$ the elements of depth at least $k$ form a subgroup. Thus, the only possible sequences of depths $<5$ of series in $\No(\F_2)$ representing nontrivial elements of $V$ are $1, 1, 3$ and $3, 3, 3$. The latter is impossible, since the product of two elements of depth $k$ has depth at least $k+1$. 

It remains to treat the case where the depths of the nontrivial elements are $1, 1, 3$. By Klopsch's theorem \cite{Klopsch} every element of order $2$ and depth $1$ is conjugate to $t/(t+1)$, so without loss of generality we may  assume that $V=\{\mathrm{id},\sigma,\tau,\sigma\circ  \tau\}$ with \begin{equation*} \sigma(t)=\frac{t}{t+1}\qquad\mbox{and}\qquad \tau(t)=t+t^2+\sum_{i\geq 3}a_i t^i.\end{equation*} We will reach a contradiction by computing up to order $O(t^9)$. We have \begin{align*} \tau^{\circ 2}(t)&=t+(1+a_3)t^4+(a_3a_4+a_5)t^6+(a_3+a_3a_4+a_4a_5+a_6+a_3a_6+a_7)t^8+O(t^9).\end{align*} Since $\tau^{\circ 2}=\mathrm{id}$, this gives $a_3=1$, $a_4=a_5$, and $a_7=1$. Substituting these values allows us to compute \begin{align*} (\sigma\circ\tau)(t)&=t+(1+a_4)t^4+(1+a_4)t^5+(a_4+a_6)t^6+(1+a_4)t^7+\\ & \qquad \qquad (1+a_4+a_6+a_8)t^8+O(t^9);\\ (\tau\circ\sigma)(t)&=t+(1+a_4)t^4+(1+a_4)t^5+(a_4+a_6)t^6+(1+a_4)t^7+(a_6+a_8)t^8+O(t^9).\end{align*} Since $\sigma\circ\tau =\tau\circ \sigma$, this gives $a_4=1$, and  shows that the depth of $\sigma\circ \tau$ is at least $5$. 
\end{proof}

\subsection{Using automata} We now show how to use automata to embed the Klein four-group $V$ into $\No(\F_2)$. We start with the $V$-extension $\F_2\lau{z}(x,y)$ of $\F_2\lau{z}$ given by 
$x^2+x=z^{-1}$ and $y^2+y=z^{-3}$ with two generators $\sigma_{V,1},\sigma_{V,2}$ of $V$ acting as 
$$ \left\{ \begin{array}{l} \sigma_{V,1}(x)=x+1; \\ \sigma_{V,1}(y)=y \end{array} \right.   \qquad \mbox{and} \qquad \left\{ \begin{array}{l}  \sigma_{V,2}(x)=x; \\ \sigma_{V,2}(y)=y+1. \end{array}\right. $$
Since $\sigma_{V,1},\sigma_{V,2}$ are different, of order two and commute, they generate the group $V$. 
Set $w=y+x^3+x^2+x$. We may regard $\F_2\lau{z}(x,y)$ as the extension $\F_2\lau{z}(x,y)=\F_2\lau{z}(x,w)$ of $\F_2\lau{z}$ given by $$\left\{ \begin{array}{l} x^2+x=z^{-1}; \\  w^2+w=x^5+x \end{array} \right.$$ with the two generators $\sigma_{V,1}$ and $\sigma_{V,2}$ acting on $x$ and $w$ as 
$$ \left\{ \begin{array}{l} \sigma_{V,1}(x)=x+1; \\ \sigma_{V,1}(w)=w+x^2+x+1\end{array} \right.   \qquad \mbox{and} \qquad \left\{ \begin{array}{l}  \sigma_{V,2}(x)=x; \\ \sigma_{V,2}(w)=w+1. \end{array} \right.$$
 Writing $z_0=z$, $z_1$,  $z_2$   for uniformisers of the fields in the tower of field extensions $$K_0:=\F_2\lau{z}\subsetneq K_1=K_0(x)=\F_2\lau{z_1}\subsetneq K_2=K_1(w)=\F_2\lau{z_2},$$ we have $v_{z_1}(x)=-1$, $v_{z_1}(x^5+x)=-5$, and hence $v_{z_2}(w)=-5$ and $v_{z_2}(x)=-2$. Choosing a uniformiser $t=x^2w^{-1}$ (note that $v_{z_2}(t)=1$), we find by elimination of the variables $z,x,w$ that $ \sigma_{V,1}= \sigma_{V,1}(t)$ and $\sigma_{V,2}=\sigma_{V,2}(t)$ satisfy, respectively, 
\begin{align*}
t^4 X^4+t^3X^3+ X^2+(t+1) X+t^2+t&=0; \\
(t^4+1)X^4+tX^2+t^2X+t^4&=0. 
\end{align*}
This is solved with respective initial coefficients
$$
 \sigma_{V,1}= t+t^2 + O(t^3) \qquad \text{and} \qquad \sigma_{V,2}= t + t^6+O(t^{7}). 
$$
 The corresponding automata have $18$ and  $14$ states, respectively. 

\begin{proposition} \label{propkleinfour} 
The series $ \sigma_{V,1}$ and $\sigma_{V,2}$  have break sequences $(1)$ and $(5)$ and satisfy $ \sigma_{V,1}^{\circ 2} = \sigma_{V,2}^{\circ 2}=t$ and $ \sigma_{V,1} \circ \sigma_{V,2} = \sigma_{V,2} \circ \sigma_{V,1}$, and hence exhibit an explicit embedding of the Klein four-group ${\Z}/{2}{\Z} \times {\Z}/{2}{\Z}$ into $\No(\F_2)$. The corresponding automata are depicted in Table \textup{\ref{Klein}}. \qed
\end{proposition}

 For completeness, writing $\sigma_{V,3}=\sigma_{V,1}\circ\sigma_{V,2}$ for the third nontrivial element of $V$, we find that $\sigma_{V,3}$ satisfies 
\begin{equation*} t^4X^4+(t+1)^3X^3+(t^3+t^2+t)X^2+(t+1)^3X+t^3+t=0
\end{equation*} 
with initial coefficients $\sigma_{V,3}=t + t^2 + t^3 +O(t^{5})$, leading to an automaton with $25$ states. The automaton is stored in standard Mathematica form in \cite{Database}.

\begin{remark} 
In principle, since any finite $p$-group can be realised explicitly as the  Galois group of an extension of $\F_2\lau{z}$, the Galois-theoretic method can be used to find equations satisfied by generators of any finite $p$-group embedded into $\No(\F_p)$, and thus to represent them explicitly by automata. 

The examples in the current paper do not constitute the computational limit of the method. For example, we can give an embedding of ${\Z}/4{\Z} \times {\Z}/2{\Z}$ into $\No(\F_2)$ with two generators being produced by automata with $128$ states, the order-$4$ element being minimally ramified and the order-$2$ element having depth $7$; we can also obtain an order-$9$ element in $\No(\F_3)$ with break sequence $(1,7)=\langle1,3\rangle$  produced by an automaton with $3634$ states, etc. However, we refrain from further expanding the catalogue of examples.
\end{remark}

\begin{center}
\begin{table}[t]
\begin{tabular}{cc}
\hline \\
\begin{tikzpicture}[scale=0.1]
\tikzstyle{every node}+=[inner sep=0pt]
\draw [black] (4.1,-3.1) circle (3);
\draw (4.1,-3.1) node {$1$};
\draw [black] (4.1,-55.9) circle (3);
\draw (4.1,-55.9) node {$0$};
\draw [black] (16.8,-3.1) circle (3);
\draw (16.8,-3.1) node {$1$};
\draw [black] (16.8,-55.9) circle (3);
\draw (16.8,-55.9) node {$0$};
\draw [black] (25.8,-3.1) circle (3);
\draw (25.8,-3.1) node {$1$};
\draw [black] (36.8,-3.1) circle (3);
\draw (36.8,-3.1) node {$1$};
\draw [black] (25.8,-12.8) circle (3);
\draw (25.8,-12.8) node {$1$};
\draw [black] (36.8,-12.8) circle (3);
\draw (36.8,-12.8) node {$0$};
\draw [black] (36.8,-43) circle (3);
\draw (36.8,-43) node {$0$};
\draw [black] (64.4,-17.9) circle (3);
\draw (64.4,-17.9) node {$1$};
\draw [black] (64.4,-30.7) circle (3);
\draw (64.4,-30.7) node {$1$};
\draw [black] (73.7,-12.8) circle (3);
\draw (73.7,-12.8) node {$0$};
\draw [black] (36.8,-52.2) circle (3);
\draw (36.8,-52.2) node {$1$};
\draw [black] (73.7,-55.9) circle (3);
\draw (73.7,-55.9) node {$0$};
\draw [black] (64.4,-43) circle (3);
\draw (64.4,-43) node {$0$};
\draw [black] (64.4,-52.2) circle (3);
\draw (64.4,-52.2) node {$1$};
\draw [black] (36.8,-23.3) circle (3);
\draw (36.8,-23.3) node {$0$};
\draw [black] (36.8,-33.1) circle (3);
\draw (36.8,-33.1) node {$0$};
\draw [black] (7.1,-55.9) -- (13.8,-55.9);
\fill [black] (13.8,-55.9) -- (13,-55.4) -- (13,-56.4);
\draw (10.45,-56.4) node [below] {$0$};
\draw [black] (4.1,-52.9) -- (4.1,-6.1);
\fill [black] (4.1,-6.1) -- (3.6,-6.9) -- (4.6,-6.9);
\draw (3.6,-29.5) node [left] {$1$};
\draw [black] (7.1,-3.1) -- (13.8,-3.1);
\fill [black] (13.8,-3.1) -- (13,-2.6) -- (13,-3.6);
\draw (10.45,-2.6) node [above] {$0$};
\draw [black] (7.6,-51.9) -- (6.08,-53.64);
\draw (10,-51.41) node [above] {{$\text{{\small Start}}$}};
\fill [black] (6.08,-53.64) -- (6.98,-53.37) -- (6.23,-52.71);
\draw [black] (16.8,-6.1) -- (16.8,-52.9);
\fill [black] (16.8,-52.9) -- (17.3,-52.1) -- (16.3,-52.1);
\draw (16.3,-29.5) node [left] {$1$};
\draw [black] (28.8,-3.1) -- (33.8,-3.1);
\fill [black] (33.8,-3.1) -- (33,-2.6) -- (33,-3.6);
\draw (31.3,-2.6) node [above] {$0$};
\draw [black] (19.8,-3.1) -- (22.8,-3.1);
\fill [black] (22.8,-3.1) -- (22,-2.6) -- (22,-3.6);
\draw (21.3,-2.6) node [above] {$0$};
\draw [black] (25.8,-6.1) -- (25.8,-9.8);
\fill [black] (25.8,-9.8) -- (26.3,-9) -- (25.3,-9);
\draw (25.3,-7.95) node [left] {$1$};
\draw [black] (36.8,-6.1) -- (36.8,-9.8);
\fill [black] (36.8,-9.8) -- (37.3,-9) -- (36.3,-9);
\draw (36.3,-7.95) node [left] {$1$};
\draw [black] (33.8,-12.8) -- (28.8,-12.8);
\fill [black] (28.8,-12.8) -- (29.6,-13.3) -- (29.6,-12.3);
\draw (31.3,-12.3) node [above] {$1$};
\draw [black] (17.41,-52.96) -- (25.19,-15.74);
\fill [black] (25.19,-15.74) -- (24.53,-16.42) -- (25.51,-16.62);
\draw (20.56,-34.01) node [left] {$1$};
\draw [black] (34.28,-44.63) -- (19.32,-54.27);
\fill [black] (19.32,-54.27) -- (20.26,-54.26) -- (19.72,-53.42);
\draw (25.8,-48.95) node [above] {$0$};
\draw [black] (26.83,-15.62) -- (35.77,-40.18);
\fill [black] (35.77,-40.18) -- (35.97,-39.26) -- (35.03,-39.6);
\draw (30.54,-28.69) node [left] {$1$};
\draw [black] (39.44,-4.52) -- (61.76,-16.48);
\fill [black] (61.76,-16.48) -- (61.29,-15.66) -- (60.81,-16.54);
\draw (51.6,-10) node [above] {$0$};
\draw [black] (62.18,-19.92) -- (39.02,-40.98);
\fill [black] (39.02,-40.98) -- (39.95,-40.81) -- (39.27,-40.07);
\draw (51.61,-30.94) node [below] {$1$};
\draw [black] (61.992,-19.67) arc (-25.94611:-313.94611:2.25);
\draw (57.03,-19.33) node [left] {$0$};
\fill [black] (61.53,-17.07) -- (61.03,-16.27) -- (60.59,-17.17);
\draw [black] (72.32,-15.46) -- (65.78,-28.04);
\fill [black] (65.78,-28.04) -- (66.6,-27.56) -- (65.71,-27.1);
\draw (68.36,-20.61) node [left] {$1$};
\draw [black] (39.8,-12.8) -- (70.7,-12.8);
\fill [black] (70.7,-12.8) -- (69.9,-12.3) -- (69.9,-13.3);
\draw (55.25,-12.3) node [above] {$0$};
\draw [black] (64.4,-27.7) -- (64.4,-20.9);
\fill [black] (64.4,-20.9) -- (63.9,-21.7) -- (64.9,-21.7);
\draw (64.9,-24.3) node [right] {$1$};
\draw [black] (28.52,-14.06) -- (61.68,-29.44);
\fill [black] (61.68,-29.44) -- (61.16,-28.65) -- (60.74,-29.55);
\draw (46.08,-21.24) node [above] {$0$};
\draw [black] (62.03,-32.54) -- (39.17,-50.36);
\fill [black] (39.17,-50.36) -- (40.11,-50.26) -- (39.49,-49.47);
\draw (49.59,-40.95) node [above] {$0$};
\draw [black] (73.7,-15.8) -- (73.7,-52.9);
\fill [black] (73.7,-52.9) -- (74.2,-52.1) -- (73.2,-52.1);
\draw (74.2,-34.35) node [right] {$0$};
\draw [black] (65.062,-40.086) arc (194.92496:-93.07504:2.25);
\draw (70.41,-37.27) node [above] {$0,1$};
\fill [black] (67.12,-41.75) -- (68.02,-42.03) -- (67.76,-41.06);
\draw [black] (61.561,-51.266) arc (279.51902:-8.48098:2.25);
\draw (59.06,-46) node [left] {$0$};
\fill [black] (63.41,-49.38) -- (63.78,-48.51) -- (62.79,-48.67);
\draw [black] (64.4,-49.2) -- (64.4,-46);
\fill [black] (64.4,-46) -- (63.9,-46.8) -- (64.9,-46.8);
\draw (64.9,-47.6) node [right] {$1$};
\draw [black] (39.8,-52.2) -- (61.4,-52.2);
\fill [black] (61.4,-52.2) -- (60.6,-51.7) -- (60.6,-52.7);
\draw (50.6,-51.7) node [above] {$0$};
\draw [black] (71.95,-53.47) -- (66.15,-45.43);
\fill [black] (66.15,-45.43) -- (66.22,-46.37) -- (67.03,-45.79);
\draw (69.64,-48.07) node [right] {$0$};
\draw [black] (39.8,-43) -- (61.4,-43);
\fill [black] (61.4,-43) -- (60.6,-42.5) -- (60.6,-43.5);
\draw (50.6,-43.5) node [below] {$1$};
\draw [black] (70.7,-55.9) -- (19.8,-55.9);
\fill [black] (19.8,-55.9) -- (20.6,-56.4) -- (20.6,-55.4);
\draw (45.25,-56.4) node [below] {$1$};
\draw [black] (33.85,-52.75) -- (19.75,-55.35);
\fill [black] (19.75,-55.35) -- (20.63,-55.7) -- (20.45,-54.72);
\draw (26.29,-53.46) node [above] {$1$};
\draw [black] (36.8,-20.3) -- (36.8,-15.8);
\fill [black] (36.8,-15.8) -- (36.3,-16.6) -- (37.3,-16.6);
\draw (37.3,-18.05) node [right] {$1$};
\draw [black] (18.37,-53.34) -- (35.23,-25.86);
\fill [black] (35.23,-25.86) -- (34.39,-26.28) -- (35.24,-26.8);
\draw (26.16,-38.32) node [left] {$0$};
\draw [black] (36.8,-26.3) -- (36.8,-30.1);
\fill [black] (36.8,-30.1) -- (37.3,-29.3) -- (36.3,-29.3);
\draw (36.3,-28.2) node [left] {$0$};
\draw [black] (38.127,-30.422) arc (181.37791:-106.62209:2.25);
\draw (42.74,-27.6) node [right] {$0$};
\fill [black] (39.73,-32.52) -- (40.54,-33) -- (40.52,-32.01);
\draw [black] (36.8,-36.1) -- (36.8,-40);
\fill [black] (36.8,-40) -- (37.3,-39.2) -- (36.3,-39.2);
\draw (37.3,-38.05) node [right] {$1$};
\draw [black] (6.969,-3.935) arc (101.49796:-186.50204:2.25);
\draw (9.61,-9.12) node [right] {$1$};
\fill [black] (5.18,-5.89) -- (4.85,-6.77) -- (5.83,-6.57);
\node at (38,-63) {Automaton of $\sigma_{V,1}$};
\end{tikzpicture}
& 
\begin{tikzpicture}[scale=0.1]
\tikzstyle{every node}+=[inner sep=0pt]
\draw [black] (14.3,-17.8) circle (3);
\draw (14.3,-17.8) node {$0$};
\draw [black] (7,-8.1) circle (3);
\draw (7,-8.1) node {$1$};
\draw [black] (27.9,-17.8) circle (3);
\draw (27.9,-17.8) node {$0$};
\draw [black] (7,-55.2) circle (3);
\draw (7,-55.2) node {$0$};
\draw [black] (13.7,-41.7) circle (3);
\draw (13.7,-41.7) node {$0$};
\draw [black] (36.9,-26.9) circle (3);
\draw (36.9,-26.9) node {$1$};
\draw [black] (52.3,-26.9) circle (3);
\draw (52.3,-26.9) node {$1$};
\draw [black] (60.9,-17.8) circle (3);
\draw (60.9,-17.8) node {$1$};
\draw [black] (27.9,-41.7) circle (3);
\draw (27.9,-41.7) node {$0$};
\draw [black] (60.9,-41.7) circle (3);
\draw (60.9,-41.7) node {$1$};
\draw [black] (73.8,-55.2) circle (3);
\draw (73.8,-55.2) node {$0$};
\draw [black] (36.9,-55.2) circle (3);
\draw (36.9,-55.2) node {$0$};
\draw [black] (73.8,-8.1) circle (3);
\draw (73.8,-8.1) node {$0$};
\draw [black] (60.9,-8.1) circle (3);
\draw (60.9,-8.1) node {$0$};
\draw [black] (14.3,-24.9) -- (14.3,-20.8);
\draw (14.3,-25.4) node [below] {{$\text{{\small Start}}$}};
\fill [black] (14.3,-20.8) -- (13.8,-21.6) -- (14.8,-21.6);
\draw [black] (17.3,-17.8) -- (24.9,-17.8);
\fill [black] (24.9,-17.8) -- (24.1,-17.3) -- (24.1,-18.3);
\draw (21.1,-17.3) node [above] {$0$};
\draw [black] (12.5,-15.4) -- (8.8,-10.5);
\fill [black] (8.8,-10.5) -- (8.89,-11.44) -- (9.68,-10.84);
\draw (11.23,-11.55) node [right] {$1$};
\draw [black] (5.677,-5.42) arc (234:-54:2.25);
\draw (7,-0.85) node [above] {$0$};
\fill [black] (8.32,-5.42) -- (9.2,-5.07) -- (8.39,-4.48);
\draw [black] (7,-11.1) -- (7,-52.2);
\fill [black] (7,-52.2) -- (7.5,-51.4) -- (6.5,-51.4);
\draw (6.5,-31.65) node [left] {$1$};
\draw [black] (26.37,-20.38) -- (15.23,-39.12);
\fill [black] (15.23,-39.12) -- (16.07,-38.69) -- (15.21,-38.18);
\draw (20.15,-28.49) node [left] {$0$};
\draw [black] (34.37,-28.51) -- (16.23,-40.09);
\fill [black] (16.23,-40.09) -- (17.17,-40.08) -- (16.63,-39.23);
\draw (24.3,-33.8) node [above] {$1$};
\draw [black] (11.658,-39.519) arc (250.8489:-37.1511:2.25);
\draw (10.27,-34.66) node [above] {$0,1$};
\fill [black] (14.19,-38.75) -- (14.92,-38.16) -- (13.98,-37.83);
\draw [black] (35.577,-24.22) arc (234:-54:2.25);
\draw (36.9,-19.65) node [above] {$0$};
\fill [black] (38.22,-24.22) -- (39.1,-23.87) -- (38.29,-23.28);
\draw [black] (49.3,-26.9) -- (39.9,-26.9);
\fill [black] (39.9,-26.9) -- (40.7,-27.4) -- (40.7,-26.4);
\draw (44.6,-26.4) node [above] {$0$};
\draw [black] (58.84,-19.98) -- (54.36,-24.72);
\fill [black] (54.36,-24.72) -- (55.27,-24.48) -- (54.55,-23.79);
\draw (56.07,-20.88) node [left] {$0$};
\draw [black] (57.95,-17.27) -- (9.95,-8.63);
\fill [black] (9.95,-8.63) -- (10.65,-9.27) -- (10.83,-8.28);
\draw (34.44,-12.36) node [above] {$1$};
\draw [black] (27.9,-20.8) -- (27.9,-38.7);
\fill [black] (27.9,-38.7) -- (28.4,-37.9) -- (27.4,-37.9);
\draw (27.4,-29.75) node [left] {$1$};
\draw [black] (49.73,-28.46) -- (30.47,-40.14);
\fill [black] (30.47,-40.14) -- (31.41,-40.16) -- (30.89,-39.3);
\draw (39.1,-33.8) node [above] {$1$};
\draw [black] (30.9,-41.7) -- (57.9,-41.7);
\fill [black] (57.9,-41.7) -- (57.1,-41.2) -- (57.1,-42.2);
\draw (44.4,-41.2) node [above] {$1$};
\draw [black] (30.78,-42.55) -- (70.92,-54.35);
\fill [black] (70.92,-54.35) -- (70.3,-53.65) -- (70.01,-54.61);
\draw (51.67,-47.9) node [above] {$0$};
\draw [black] (60.9,-38.7) -- (60.9,-20.8);
\fill [black] (60.9,-20.8) -- (60.4,-21.6) -- (61.4,-21.6);
\draw (60.4,-29.75) node [left] {$0$};
\draw [black] (57.99,-42.43) -- (9.91,-54.47);
\fill [black] (9.91,-54.47) -- (10.81,-54.76) -- (10.56,-53.79);
\draw (33.23,-47.88) node [above] {$1$};
\draw [black] (8.33,-52.51) -- (12.37,-44.39);
\fill [black] (12.37,-44.39) -- (11.56,-44.88) -- (12.46,-45.33);
\draw (9.65,-47.35) node [left] {$1$};
\draw [black] (9.52,-53.57) -- (25.38,-43.33);
\fill [black] (25.38,-43.33) -- (24.44,-43.34) -- (24.98,-44.18);
\draw (16.45,-47.95) node [above] {$0$};
\draw [black] (70.8,-55.2) -- (39.9,-55.2);
\fill [black] (39.9,-55.2) -- (40.7,-55.7) -- (40.7,-54.7);
\draw (55.35,-55.7) node [below] {$0$};
\draw [black] (37.595,-52.294) arc (194.27935:-93.72065:2.25);
\draw (42.22,-49.52) node [above] {$0$};
\fill [black] (39.63,-53.98) -- (40.53,-54.27) -- (40.28,-53.3);
\draw [black] (33.9,-55.2) -- (10,-55.2);
\fill [black] (10,-55.2) -- (10.8,-55.7) -- (10.8,-54.7);
\draw (21.95,-55.7) node [below] {$1$};
\draw [black] (73.8,-52.2) -- (73.8,-11.1);
\fill [black] (73.8,-11.1) -- (73.3,-11.9) -- (74.3,-11.9);
\draw (73.3,-31.65) node [left] {$1$};
\draw [black] (72.72,-10.9) -- (61.98,-38.9);
\fill [black] (61.98,-38.9) -- (62.73,-38.33) -- (61.8,-37.97);
\draw (68.1,-25.74) node [right] {$1$};
\draw [black] (70.8,-8.1) -- (63.9,-8.1);
\fill [black] (63.9,-8.1) -- (64.7,-8.6) -- (64.7,-7.6);
\draw (67.35,-7.6) node [above] {$0$};
\draw [black] (60.9,-11.1) -- (60.9,-14.8);
\fill [black] (60.9,-14.8) -- (61.4,-14) -- (60.4,-14);
\draw (60.4,-12.95) node [left] {$1$};
\draw [black] (58.02,-8.95) -- (30.78,-16.95);
\fill [black] (30.78,-16.95) -- (31.69,-17.21) -- (31.4,-16.25);
\draw (43.58,-12.4) node [above] {$0$};
\node at (38,-63) {Automaton of $\sigma_{V,2}$};
\end{tikzpicture}
\\
\hline \\
\end{tabular} 
\caption{Automata corresponding to the elements $ \sigma_{V,1}$ and $\sigma_{V,2}$ that generate a copy of the Klein four-group in $\No(\F_2)$.}\label{Klein}
\end{table}
\end{center}

\section{State complexity of automata representing finite order elements in $\No(\F_p)$}  \label{cxproperties} 

\subsection{General bounds on state complexity} How `complex' is an automaton that computes a power series $\us \in\No(\F_p)$ of given order and break sequence? This is usually measured by `state complexity', i.e.\ the minimal number of states in an automaton that computes the series.

This complexity can be bounded theoretically. The currently best results arise from the differential forms method described in Section \ref{section2}: start with an algebraic equation (assumed irreducible) satisfied by $\us=\sigma(t)$ with coefficients from $\F_p[t]$, and consider it instead as a two-variable equation $F(t,X)=0$ describing a (possibly singular) 
algebraic curve over $\F_p$. Consider the \emph{degree} $$d_\sigma:=[\F_p(\us,t):\F_p(t)] = \deg_X F$$ 
and the \emph{height} $$h_\sigma:=[\F_p(\us,t):\F_p(\us)]  = \deg_t F$$ (the latter two equalities hold by the irreducibility of $F$),  and let $g_\sigma$ denote the genus of the normalisation $\mathcal X$ of the projective curve defined by $F(t,X)=0$. Bridy has proven that the series $\sigma$ can be realised by an automaton with less than $$p^{h_\sigma+3d_\sigma+g_\sigma-1}$$ states (see \cite[Cor.\ 3.10]{Bridybounds}, a result that assumes, like this paper, the leading zeros convention, see \cite[Remark 2.1]{Bridybounds}). Concerning the optimality of the upper bound, Bridy has shown in  \cite[Prop.\ 3.14]{Bridybounds} for every $h \geq 1$, there are power series with $d_\sigma=1, h_\sigma=h, g_\sigma=0$ that require at least $\geq p^h$ states. A lower bound for the minimal amount of states required to realise the given power series is given by $\log_p (d_\sigma+1)$ \cite[Prop.\ 2.13]{Bridybounds}; this bound appears optimal when running over all algebraic power series (loc.\ cit.).

\subsection{Degree equals height for series of finite order in $\No(\F_p)$} In our situation we have the following extra information. 

\begin{proposition} \label{d=h}
Let $\sigma(t) \in \F_p\lau{t}$ be an algebraic power series over $\F_p(t)$ of finite compositional order. 
Then  $d_\sigma=h_\sigma$. 
\end{proposition}
\begin{proof}
Write $n$ for the compositional order of $\sigma(t)$. The map $\sigma$, regarded as an automorphism of $\F_p(\!(t)\!)$, restricts to an automorphism of the field $$K:=\F_p(t,\sigma(t),\sigma^{\circ 2}(t),\ldots,\sigma^{\circ (n-1)}(t)).$$ Since $\sigma(t)$ is algebraic over $\F_p(t)$, successive application of the automorphism $\sigma$ shows that $\F_p(\sigma^{\circ k}(t))$ is algebraic over $\F_p(\sigma^{\circ (k-1)}(t))$ for $k\geq 1$, and hence the extension $K/{\F_p(t)}$ is algebraic. Since the automorphism $\sigma$ maps $\F_p(t)$ onto $\F_p(\sigma(t))$, we have $[K:\F_p(t)]=[K:\F_p(\sigma(t))]$, and hence \[d_{\sigma}=[\F_p(t,\sigma(t)):\F_p(t)]=\frac{[K:\F_p(t)]}{[K:\F_p(t,\sigma(t))]}=\frac{[K:\F_p(\sigma(t))]}{[K:\F_p(t,\sigma(t))]}=[\F_p(t,\sigma(t)):\F_p(\sigma(t))]=h_{\sigma}.\qedhere\]
\end{proof}

 In Table \ref{tablebounds} we give the state complexity for the automata we constructed (where the first two rows refer to series that are considered in the next section), plus the theoretical upper and lower bounds (computed using {\sc Singular} \cite{Singular} and {\sc Magma} \cite{Magma}). We observe that the required number of states is much lower than the (generically almost tight, at least in the genus zero case) upper bounds. The reader may be convinced of this non-generic behaviour by perturbing some of the coefficients in the equation for $\sigma_8$ and using \cite{LDG} to compute the number of states required to solve those perturbed equations (which typically also have higher genus).  
 
 \begin{remark} \label{statesklopsch} 
Table \ref{tablebounds} lacks a general formula for the minimal number of states in a $2$-automaton computing Klopsch's series $\sigma_{\mathrm{K},m}$ for general $m$. For $m=1,3,5,\dots,1023$  we computed this in \cite{Rowland} and \cite{LDG} to be $2,6,14,9,28,53,67,12,54,127,\dots,30$. One may show that for $m=2^\mu-1$ such an automaton has $3\mu$ states. We conjecture that for $m=2^\mu+1$ it has $3 \cdot 2^{\mu}+2\mu-2$ states. For $m=2^\mu+3$, we find the sequence $14,9,53,127,90,931,2675,770,\dots$, which we could not fit into any mould. 
 \end{remark} 
 
\begin{center}
\begin{table} 
\hspace*{-7mm}\begin{tabular}{lcccHccc}
\hline 
series & order &  breaks & $d_\sigma=h_\sigma$ & $h$ & $g_\sigma$ & bounds 
& \# of states  \\
\hline 
$ \sigma_{\mathrm{S},1} $ & $2$ & $(1)$ & $2$ & $ $ & $1$ & $[1,2^{8}]$ & $5$ \\
$ \sigma_{\mathrm{S},m=2^\mu-1>1}$ & $2$ & $(m)$ & $\frac{m+1}{2}$ & $ $ & $\frac{m-1}{2}$ & $[\mu-1,2^{\frac{5m+1}{2}}]$ & $\mu+3$ \\
$ \sigma_{\mathrm{S},m=2^\mu+1}$ & $2$ & $(m)$ & $m-1$ & $ $ & $\frac{(m-1)(m-2)}{2}$ & $[\mu,2^{\frac{m^2+5m-8}{2}}]$ & $2^\mu + 3^\mu$? \\
$\us_{\mathrm{K},3}$ & $2$ & $(3)$ & $3$ & $3$ & $1$ & $[2,2^{12}]$ & $6$  \\
$\us_{\mathrm{K},m}$ & $2$ & $(m)$ & $m$ & $m$ & $\frac{(m-1)(m-2)}{2}$ & $[\lfloor \log_2(m+1)\rfloor,2^{\frac{m(m+5)}{2}}]$ & $\times$  \\
$\us_{\mathrm{CS}}^{\circ 2}$ & $2$ & $(3)$ & $2$ & $2$ & $1$ & $[1,2^8]$ & $7$  \\
$ \sigma_{V,1}$ & $2$ & $(1)$ & $4$ & $4$ & $2$ & $[2,2^{17}]$ & $18$ \\
$ \sigma_{V,2}$ & $2$ & $(5)$ & $4$ & $4$ & $2$ & $[2,2^{17}]$ & $14$ \\
$ \sigma_{V,3}$ & $2$ & $(1)$ & $4$ & $4$ & $2$ & $[2,2^{17}]$ & $25$ \\
$\us_{\mathrm{min}}$ & $4$ & $(1,3)$ & $3$ & $3$ & $1$ & $[2,2^{12}]$ & $5$    \\
$\us_{\mathrm{CS}}$ & $4$ &  $(1,3)$    & $2$ & $2$ & $1$ & $[1,2^8]$ & $7$   \\
$\us^{\circ 3}_{\mathrm{CS}}$ & $4$ &  $(1,3)$    & $2$ & $2$ & $1$ & $[1,2^8]$ & $7$    \\
$\us_{\mathrm{J}}$ & $4$ &  $(1,3)$    & $2$ & $2$ & $1$ & $[1,2^8]$ & $9$  \\
$\us^{\circ 3}_{\mathrm{J}}$ & $4$ &  $(1,3)$    & $2$ & $2$ & $1$ & $[1,2^8]$ & $11$   \\
$\sigma_{\mathrm{T},1}$ & $4$ & $(1,3)$ & $4$ & $4$ & $1$ & $[2,2^{16}]$ & $9$     \\
$\sigma_{\mathrm{T},2}, \sigma_{\mathrm{T},3}, \sigma_{\mathrm{T},4}$ & $4$ & $(1,3)$ & $4$ & $4$ & $1$ & $[2,2^{16}]$ & $17$    \\
$\us_{(1,5)}$ & $4$ & $(1,5)$  & $3$  & $3$ & $2$ & $[2,2^{13}]$ & $13$   \\
$\us_{(1,9)}$ & $4$ & $(1,9)$  & $7$  & $7$ & $4$ & $[3,2^{31}]$ & $110$    \\
$\us_8$ & $8$ & $(1,3,11)$ & $6$ & $6$ & $7$ & $[2, 2^{30}]$ & $320$    \\
\hline\\
\end{tabular} 
\caption{For each series, we give: its compositional order, lower break sequence, the degree $d_\sigma$ 
and genus $g_\sigma$ of the algebraic equation it satisfies, the theoretical interval $[\lfloor \log_2(d_\sigma+1) \rfloor,2^{4d_\sigma+g_\sigma-1}]$ for the number of states of a minimal automaton and the actual number of states (`?' means we conjecture this to be the correct answer, `$\times$' means we do not know the answer; see Remark \ref{statesklopsch}).}
\label{tablebounds}
\end{table}
\end{center}

\section{A hierarchy of complexity of power series based on sparseness} \label{aridsection} 

Previously known examples of finite order elements of $\No(\F_2)$ were described as power series having as coefficients binomial coefficients modulo $2$ 
(such as Klopsch's series) or by explicit formulas for the location of the nonzero coefficients (such as the Chinburg--Symonds series $\us_{\mathrm{CS}}$ and $\us_{\mathrm{CS}}^{\circ 3}$). Our automatic description is somewhat different. In this section, we discuss the relation between the existence of `closed/explicit formulas' and properties of the automaton. \subsection{Sparse power series} We propose a definition of a `closed formula' for a power series based on the notion of sparseness (the concept occurs in the literature under various names such as `arid', `poly-slender', `polynomial growth', and `bounded'; compare \cite[\S 3]{BK2020}).

\begin{definition} For a  power series $\us = \sum a_k t^k\in\F_2\pau{t}$ over $\F_2$, let $E(\us)$ denote the \emph{support} of $\sigma$, i.e.\ the set of integers $k$ for which $a_k=1$. 
A power series $\us$ (as well as the corresponding automaton and automatic sequence, if they exist) is called \emph{sparse}  if 
$$\# E(\us) \cap \{0,1,\dots,N\} = O(\log(N)^r)$$ for some $r\geq 0$. The infimum of such $r$ is called the \emph{rank of sparseness} of $\us$. We say that $\us$ is $r$-sparse if the rank of sparseness is at most $r$.  If $\sigma$ is automatic, then this infimum is attained and is an integer (this follows from Proposition \ref{SY} below). 
\end{definition} 

Note that polynomials are sparse, sums of sparse series are sparse, and products of sparse series are sparse. More precisely, if $\us$ is $r$-sparse and $\tau$ is $s$-sparse, then $\us+\tau $ is at most $\max(r,s)$-sparse and $\us \tau$ is at most $(r+s)$-sparse; this follows from the definition, since $E(\us + \tau) \subseteq E(\us)\cup E(\tau)$ and $E(\us \tau) \subseteq E(\us)+E(\tau)$. For automatic sequences, Cobham showed the following dichotomy for the word growth in the associated regular language. 

\begin{proposition}[{Cobham \cite{CobhamTag}}] An automatic sequence $\us \in\F_2\pau{t}$ is either sparse, or  $\# E(\us) \cap \{0,1,\dots,N\} \geq N^\alpha$ for some real $\alpha>0$  and sufficiently large $N$. \qed
\end{proposition} 

Define a \emph{simple sparse set of rank at most $r$} to be a set of integers whose base-$2$  expansion is of the form $v_r w_r^{\ell_r} \cdots v_1 w_1^{\ell_1} v_0$ with $\ell_i \in \Z_{\geq 0}$  for some fixed binary words $v_0,\dots,v_r,w_1,\dots,w_r$. 

\begin{proposition}[Szilard, Yu, Zhang and Shallit \cite{SYZS}] \label{SY} A series $\sigma$ is automatic and sparse of rank at most $r$ precisely if $E(\sigma)$ is a finite union of pairwise disjoint simple sparse sets of rank at most $r$.
\end{proposition} 

\begin{proof} Except for the claim of `pairwise disjointness', this is proven in \cite{SYZS}. The claim that the occurring simple sparse sets can be chosen pairwise disjoint is proven in detail in \cite[Cor.\ 3.10]{BK2020}. 
\end{proof} 

\begin{remark} 
The proof in \cite[Cor.\ 3.10]{BK2020} is a tedious combinatorial verification. Jason Bell pointed out to use that a much simpler  argument is possible if ones uses the structure of the corresponding automaton that results from Proposition \ref{SYZScheck} below.
\end{remark}

\begin{example} \label{exes} The support of $\us_{\mathrm{CS}}^{\circ 3}$ is $E(\us_{\mathrm{CS}}^{\circ 3}) = \{ 3 \cdot 2^k-2 \mid k \geq 0\} \cup \{ 4 \cdot 2^k-2 \mid k \geq 0\},$ and consists of the integers whose base-$2$  expansion is  $1$, $101^\ell0$ or $1^{\ell}1 0$ for some $\ell \in \Z_{\geq 0}$. Similarly, all  power series in Table \ref{s5-s8} are sparse.  On the other hand, the description of the support of $\sigma_{\mathrm{K},3}$ in Example \ref{exklopschauto} in terms of the base-$4$ representation with only half the possible digits allowed shows that $\# E(\sigma_{\mathrm{K},3}) \cap \{1,\dots,N\}$ grows as $\sqrt{N}f(N)$ for  a function $f$ that is bounded away from both $0$ and infinity, and so $\sigma_{\mathrm{K},3}$ is not sparse.
\end{example}

\begin{remark} A sparse automatic series is `easy' in the sense that the full set consisting of the first $N$ terms of the series can be computed in `polylogarithmic time', i.e.\ polynomial time in $\log(N)$, given the words $v_i,w_i$ as in the definition of a simple sparse set, which allow one to output the nonzero exponents in the series. In contrast to this, computation of the $n$-th coefficient of a general automatic sequence can be done in time $O(\log(n))$ (by base-$2$ expansion and running through the automaton), so computing all first $N$ coefficients would require $O(\log(N!))= O(N\log N)$ time. 
\end{remark} 

\subsection{Conjugating to a sparse representative} One may ask whether every series of finite order in $\No(\F_2)$ can be conjugated to a sparse series. We have no general answer to this question, not even for series of order $2$, which form a unique conjugacy class for every value of the break sequence $(m)$, represented by Klopsch's series $\sigma_{\mathrm{K},m}=t{/}\!{\sqrt[m]{1+t^m}}$. Klopsch's series itself is not sparse, since its $m$-th power $\sigma_{\mathrm{K},m}^m=t^m/(1+t^m)=\sum\limits_{k\geq 1} t^{km}$ is not. Nevertheless, for special values of the break sequence we can find a sparse representative. 

\begin{proposition} \label{KSprop} Let  $m$ be an integer of the form $m=2^{\mu}\pm 1$ for $\mu \geq 1$. Then any power series of order $2$ and break sequence $(m)$ is conjugate to a sparse power series. More precisely, we have the following\textup{:} \begin{enumerate} \item \label{KSPropm1} Any power series of order $2$ and break sequence $(1)$ is conjugate to the power series
\begin{equation} \label{order2m1} \sigma_{\mathrm{S},1} = t + \sum_{k \geq 2} \left( t^{2^k-2} + t^{2^k-1} \right), \end{equation}  which is sparse of rank $1$. 
The corresponding automaton is displayed in Table \textup{\ref{sparse2m}}. 
\item \label{KSPropm2}  If $m=2^{\mu}-1>1$, then any power series of order $2$ and break sequence $(m)$ is conjugate to the  power series 
\begin{equation} \label{order2m}  \sigma_{\mathrm{S},m}  = t + \sum_{k \geq 1} t^{\frac{m+1}{m-1} \left( m \cdot \left(\frac{m+1}{2}\right)^{k-1}-1 \right)}, \end{equation} 
which is sparse of rank $1$. The set of exponents occurring in $\sigma$ consists of the integers whose base-$2$ representation is either $1$ or $10^{\mu-1}(10^{\mu-2})^\ell 0 $ for some $\ell \in \Z_{\geq 0}$. The corresponding automata are displayed in Table \textup{\ref{sparse2m}}. 
\item \label{KSPropm3} If $m=2^{\mu}+1$, then any power series of order $2$ and break sequence $(m)$ is conjugate to the  power series 
\begin{equation} \label{order2mplus}  \sigma_{\mathrm{S},m}  = \hspace*{-3mm}  \sum_{\substack{\emptyset \neq J \subseteq \{0,\ldots,\mu-1\} \\ k \colon J \to \Z_{\geq 0}}}
\hspace*{-3mm} t^{\left(\sum\limits_{ j\in J}2^j(m-1)^{k(j)}\right)m-m+1}, \end{equation} 
%
%
which is sparse of rank $\mu$\textup{:} the  support of $ \sigma_{\mathrm{S},m}$ consists precisely of the integers $m(\ell-1)+1$ with $ \ell\geq 1$ an integer whose base-$2$ expansion contains at most $\mu$ occurrences of the digit $1$ and all these occurrences are at distinct positions modulo $\mu$. 
\end{enumerate}
\end{proposition} 

\begin{table}[t]

\begin{tabular}{ccccccc}
\hline \\
\begin{tikzpicture}[scale=0.15]
\tikzstyle{every node}+=[inner sep=0pt]
\draw [black] (37.2,-12.3) circle (3);
\draw (37.2,-12.3) node {$0$};
\draw [black] (24.7,-12.3) circle (3);
\draw (24.7,-12.3) node {$1$};
\draw [black] (37.2,-34.7) circle (3);
\draw (37.2,-34.7) node {$0$};
\draw [black] (24.7,-22.5) circle (3);
\draw (24.7,-22.5) node {$1$};
\draw [black] (24.7,-34.7) circle (3);
\draw (24.7,-34.7) node {$0$};
\draw [black] (34.2,-12.3) -- (27.7,-12.3);
\fill [black] (27.7,-12.3) -- (28.5,-12.8) -- (28.5,-11.8);
\draw (30.95,-11.8) node [above] {$1$};
\draw [black] (37.2,-15.3) -- (37.2,-31.7);
\fill [black] (37.2,-31.7) -- (37.7,-30.9) -- (36.7,-30.9);
\draw (36.7,-23.5) node [left] {$0$};
\draw [black] (37.2,-5.9) -- (37.2,-9.3);
\draw (37.2,-5.4) node [above] {Start};
\fill [black] (37.2,-9.3) -- (37.7,-8.5) -- (36.7,-8.5);
\draw [black] (34.2,-34.7) -- (27.7,-34.7);
\fill [black] (27.7,-34.7) -- (28.5,-35.2) -- (28.5,-34.2);
\draw (30.95,-34.2) node [above] {$0$};
\draw [black] (24.7,-25.5) -- (24.7,-31.7);
\fill [black] (24.7,-31.7) -- (25.2,-30.9) -- (24.2,-30.9);
\draw (24.2,-28.6) node [left] {$1$};
\draw [black] (22.02,-36.023) arc (-36:-324:2.25);
\draw (17.45,-34.7) node [left] {$0,1$};
\fill [black] (22.02,-33.38) -- (21.67,-32.5) -- (21.08,-33.31);
\draw [black] (22.02,-13.623) arc (-36:-324:2.25);
\draw (17.45,-12.3) node [left] {$1$};
\fill [black] (22.02,-10.98) -- (21.67,-10.1) -- (21.08,-10.91);
\draw [black] (35.74,-32.08) -- (26.16,-14.92);
\fill [black] (26.16,-14.92) -- (26.12,-15.86) -- (26.99,-15.37);
\draw (31.61,-22.29) node [right] {$1$};
\draw [black] (24.7,-15.3) -- (24.7,-19.5);
\fill [black] (24.7,-19.5) -- (25.2,-18.7) -- (24.2,-18.7);
\draw (24.2,-17.4) node [left] {$0$};
\draw [black] (22.02,-23.823) arc (324:36:2.25);
\draw (17.45,-22.5) node [left] {$0$};
\fill [black] (22.02,-21.18) -- (21.67,-20.3) -- (21.08,-21.11);
\end{tikzpicture}
& \quad \quad & 
\begin{tikzpicture}[scale=0.15]
\tikzstyle{every node}+=[inner sep=0pt]
\draw [black] (18.5,-24.5) circle (3);
\draw (18.5,-24.5) node {$1$};
\draw [black] (18.5,-37.9) circle (3);
\draw (18.5,-37.9) node {$0$};
\draw [black] (32.7,-37.9) circle (3);
\draw (32.7,-37.9) node {$0$};
\draw [black] (32.7,-24.5) circle (3);
\draw (32.7,-24.5) node {$0$};
\draw [black] (25.9,-13.9) circle (3);
\draw (25.9,-13.9) node {$0$};
\draw [black] (18.5,-44.3) -- (18.5,-40.9);
\draw (18.5,-44.8) node [below] {Start};
\fill [black] (18.5,-40.9) -- (18,-41.7) -- (19,-41.7);
\draw [black] (18.5,-34.9) -- (18.5,-27.5);
\fill [black] (18.5,-27.5) -- (18,-28.3) -- (19,-28.3);
\draw (18,-31.2) node [left] {$1$};
\draw [black] (21.5,-37.9) -- (29.7,-37.9);
\fill [black] (29.7,-37.9) -- (28.9,-37.4) -- (28.9,-38.4);
\draw (25.6,-38.4) node [below] {$0$};
\draw [black] (32.7,-34.9) -- (32.7,-27.5);
\fill [black] (32.7,-27.5) -- (32.2,-28.3) -- (33.2,-28.3);
\draw (32.2,-31.2) node [left] {$0$};
\draw [black] (34.023,-40.58) arc (54:-234:2.25);
\draw (32.7,-45.15) node [below] {$1$};
\fill [black] (31.38,-40.58) -- (30.5,-40.93) -- (31.31,-41.52);
\draw [black] (29.7,-24.5) -- (21.5,-24.5);
\fill [black] (21.5,-24.5) -- (22.3,-25) -- (22.3,-24);
\draw (25.6,-24) node [above] {$1$};
\draw [black] (21.349,-25.402) arc (100.15733:-187.84267:2.25);
\draw (23.9,-30.65) node [right] {$0$};
\fill [black] (19.52,-27.31) -- (19.17,-28.19) -- (20.15,-28.01);
\draw [black] (28.58,-12.577) arc (144:-144:2.25);
\draw (33.15,-13.9) node [right] {$0,1$};
\fill [black] (28.58,-15.22) -- (28.93,-16.1) -- (29.52,-15.29);
\draw [black] (31.08,-21.97) -- (27.52,-16.43);
\fill [black] (27.52,-16.43) -- (27.53,-17.37) -- (28.37,-16.83);
\draw (29.92,-17.89) node [right] {$0$};
\draw [black] (20.22,-22.04) -- (24.18,-16.36);
\fill [black] (24.18,-16.36) -- (23.31,-16.73) -- (24.13,-17.3);
\draw (21.6,-17.84) node [left] {$1$};
\end{tikzpicture}
& \quad \quad & 
\begin{tikzpicture}[scale=0.15]
\tikzstyle{every node}+=[inner sep=0pt]
\draw [black] (19.1,-11.6) circle (3);
\draw (19.1,-11.6) node {$1$};
\draw [black] (19.1,-24.8) circle (3);
\draw (19.1,-24.8) node {$0$};
\draw [black] (41.4,-24.8) circle (3);
\draw (41.4,-24.8) node {$0$};
\draw [black] (41.4,-11.6) circle (3);
\draw (41.4,-11.6) node {$0$};
\draw [black] (30.7,-11.6) circle (3);
\draw (30.7,-11.6) node {$0$};
\draw [black] (17.777,-8.92) arc (234:-54:2.25);
\draw (19.1,-4.35) node [above] {$0$};
\fill [black] (20.42,-8.92) -- (21.3,-8.57) -- (20.49,-7.98);
\draw [black, dashed] (43.526,-13.692) arc (34.62756:-34.62756:7.934);
\fill [black] (43.53,-22.71) -- (44.39,-22.33) -- (43.57,-21.77);
\draw (45.43,-18.2) node [right] {$0, \mu-2$};
\draw [black] (41.4,-14.6) -- (41.4,-21.8);
\fill [black] (41.4,-14.6) -- (41.9,-15.4) -- (40.9,-15.4); 
\draw (40.9,-18.2) node [left] {$1$};
\draw [black] (33.7,-11.6) -- (38.4,-11.6);
\fill [black] (38.4,-11.6) -- (37.6,-11.1) -- (37.6,-12.1);
\draw (36.05,-11.1) node [above] {$0$};
\draw [black] (27.7,-11.6) -- (22.1,-11.6);
\fill [black] (22.1,-11.6) -- (22.9,-12.1) -- (22.9,-11.1);
\draw (24.9,-11.1) node [above] {$1$};
\draw [black] (30.7,-16.9) -- (30.7,-14.6);
\draw (30.7,-17.4) node [below] {Start};
\fill [black] (30.7,-14.6) -- (30.2,-15.4) -- (31.2,-15.4);
\draw [black] (38.4,-24.8) -- (22.1,-24.8);
\fill [black] (22.1,-24.8) -- (22.9,-25.3) -- (22.9,-24.3);
\draw (30.25,-24.3) node [above] {$0$};
\draw [black] (19.1,-21.8) -- (19.1,-14.6);
\fill [black] (19.1,-14.6) -- (18.6,-15.4) -- (19.6,-15.4);
\draw (18.6,-18.2) node [left] {$1$};
\end{tikzpicture}
\\
\hline
\\
\end{tabular}
\caption{Automata corresponding to the power series $ \sigma_{\mathrm{S},1} $ (left), $ \sigma_{\mathrm{S},2}$ (middle) and $ \sigma_{\mathrm{S},2^\mu-1} (\mu \geq 3)$ (right) in Proposition \ref{KSprop}. The dashed arrow replaces a path consisting of $\mu-3$ vertices and $\mu-2$ edges, all with label zero. The remaining missing edges (in the right automaton) all connect to a unique vertex with label $0$, which has been omitted in order to simplify the graphical representation of the automaton. } \label{sparse2m} 
\end{table}

The crucial observation used in the proof is stated in the following lemma.

\begin{lemma} \label{symlem} 
If  a polynomial $F(t,X)=0 \in \F_2[t,X]$ is symmetric in $t$ and $X$, i.e.\ $F(t,X)=F(X,t)$, and, when regarded as an algebraic equation in $X$ over $\F_2(\!(t)\!)$, has,  for some $m \geq 1$, a unique solution  $\sigma\in \No(\F_2)$ of the form $\sigma = t+t^{m+1}+O(t^{m+2})$, then $\sigma$ is of order 2. \end{lemma} 

\begin{proof} 
Composing the equality $F(t,\sigma)=0$ on the right with $\sigma^{\circ -1}$ gives $F(\sigma^{\circ -1},t)=0$, and hence, by symmetry of $F$, $F(t,\sigma^{\circ -1})=0.$ Now note that if $\sigma = t+t^{m+1}+O(t^{m+2})$, then also $\sigma^{\circ -1} = t+t^{m+1}+O(t^{m+2})$. By uniqueness, it follows that $\sigma^{\circ -1} = \sigma$, so $\sigma$ is of order $2$.  
\end{proof} 

\begin{proof} [Proof of Proposition \ref{KSprop}]
We know that there is a unique conjugacy class of order-$2$ power series with a given break sequence $(m)$, so it suffices to construct such a sparse series. When $m=2^{\mu}\pm1$, we will construct a sparse representative by exhibiting a symmetric algebraic equation $F(t,X)=0$ over $\F_2$ as in Lemma \ref{symlem}. 
Choose the polynomial as follows: 
$$ \begin{cases} 
F(t,X)=(tX)^2+(tX)+X+t &\text{for } m=1;  \\
F(t,X)=(tX)^{2^{\mu-1}}+X+t &\text{for }m=2^{\mu}-1>1; \\ 
F(t,X)=(tX)^{2^\mu}+X^{2^\mu-1} + t^{2^\mu-1}& \text{for }  m=2^{\mu}+1. 
 \end{cases}  
 $$
In all cases, Hensel's Lemma implies the existence and uniqueness of a solution $\sigma=t+t^{m+1} + O(t^{m+2})$, so Lemma \ref{symlem} applies.  
We can find an explicit solution iteratively, as follows. 

For $m=1$ we have $$\sigma = \frac{t}{t+1}+ \frac{t^2}{t+1}\sigma^2 = \frac{t}{t+1}+\frac{t^4}{(t+1)^3} + \frac{t^6}{(t+1)^3}\sigma^4=\dots=\frac{t+1}{t^2}\sum_{k\geq 1} \frac{t^{3\cdot 2^{k-1}}}{(t+1)^{2^k}}. $$ The latter sum is  $$\sum_{k\geq 1} \frac{t^{3\cdot 2^{k-1}}}{(t+1)^{2^k}}=\sum_{k\geq 1} \sum_{m\geq 1} t^{(2m+1)\cdot 2^{k-1}} = \frac{t}{t+1}+\sum_{k\geq 1} t^{2^{k-1}},$$ leading to the stated formula for $\sigma= \sigma_{\mathrm{S},1} $. 

For $m=2^{\mu}-1>1$, the same procedure leads to $$ \sigma_{\mathrm{S},m}=t+t^{2^{\mu-1}}\sigma^{2^{\mu-1}}=\dots=t+\sum_{k\geq 0} t^{2^{\mu-1}+2^{2(\mu-1)}+\dots+2^{k(\mu-1)}+2\cdot 2^{(k+1)(\mu-1)}}, $$ which is equivalent to the stated formula.

Finally, for $m=2^{\mu}+1$, we let $\tau = t/\sigma$ and $q=2^{\mu}=m-1$. Then $\tau = 1 +O(t)$ satisfies \begin{equation} \label{ASC} \tau= t^{q+1} + \tau^q\end{equation} and hence \begin{equation*} \tau = 1+ \sum_{k\geq 0} t^{q^k(q+1)}.\end{equation*} We find \begin{align*} t^q \sigma &= 1+ \tau^{q-1} =1+ \tau \cdot \tau^2 \cdot \tau^4  \cdots \tau^{2^{\mu-1}} =1+ \prod_{j=0}^{\mu-1} \left(1+\sum_{k_j\geq 0} t^{(q+1)2^jq^{k_j}}\right), \end{align*}
which  is equivalent to the stated formula.
\end{proof}

\begin{remark}\label{Rem:sparserep1}

For odd $m\geq 1$ consider the degree-$2$ extension $\F_2\lau{z}(x)$ of $\F_2\lau{z}$ with $x^2+x=z^{-m}$. The element $t=xz^{\frac{m+1}{2}}$ is a uniformiser, and the generator $\sigma$ of the Galois group acts by $\sigma(t)=(x+1)z^{\frac{m+1}{2}}$. We can eliminate the variables $x$ and $z$ by hand, obtaining the  equation $ (tX)^{\frac{m+1}{2}}+X+t=0 $. This equation always has a unique solution in $\No(\F_2)$, which has depth $m$, but is not sparse unless $m+1$ is a power of $2$ and $m\neq 1 $ (this follows from Proposition \ref{ABetter} below).  \end{remark}

\begin{remark}\label{Rem:sparserep2} The power series $ \sigma_{\mathrm{S},1} $ from Proposition \ref{KSprop}(\ref{KSPropm1}) is conjugate to Klopsch's series $\sigma_{\mathrm{K},1}:=t/(t+1)$. In this case, the conjugacy can be done  using the simple \emph{algebraic} power series $\chi=t/(t^2+1)$. 
Indeed, with $\psi:= \sum\limits_{k\geq 1} t^{2^k-1}$, we have $$\chi \cdot (\psi \circ \chi) = (t\cdot \psi)\circ \chi = \chi^2+\chi^4+\chi^8+\cdots=t^2/(t^2+1),$$ since the support of $t^2/(t^2+1)$ consists of all even integers, and the support of $\chi^{2^k}$ consists of the odd multiples of $2^k$. Hence $\chi \cdot (\psi \circ \chi) = \chi \cdot t$, so
$\chi^{\circ -1}= \psi$. 
We have $\chi \circ \sigma_{\mathrm{K},1} = t+ t^2$, and hence $$\chi \circ \sigma_{\mathrm{K},1} \circ \chi^{\circ-1} = \chi^{\circ -1} + (\chi^{\circ -1})^2 =  \sum_{k\geq 1} t^{2^k-1} +  \sum_{k\geq 1} t^{2^{k+1}-2} =  \sigma_{\mathrm{S},1} .$$ 
\end{remark}

\begin{remark}\label{Rem:sparserep3} In Table (\ref{tablebounds}), we have used that the genus of the smooth projective curve corresponding to $F(t,X) = (tX)^{k}+X+t$ is $k-1$. This follows easily by the change of variables $t=y/x^k, X=x^{k-1}/y$, leading to the Artin--Schreier equation $y^2+y=x^{2k-1}$, which has genus $k-1$ (see e.g.\ \cite[Thm.\ 6.4.1]{Stichtenoth}). For the case $m=2^\mu+1$, we also used that the genus of the Artin--Schreier curve \eqref{ASC} is $2^{\mu-1}(2^\mu-1)$. 
\end{remark}

\begin{remark}\label{Rem:sparserep4} We did not produce the general form of the automaton for $m=2^\mu+1$. Whereas the series for $m=2^\mu-1>1$ requires $\mu+3 \approx \log(m)$ states and the rank of sparseness is $1$, if $m=2^\mu+1$ an educated guess for the number of states of the minimal automaton is $2^\mu+3^\mu \approx m^{\log(3)/\log(2)}$ and the rank of sparseness is (provably) $\mu$. This looks somewhat similar to what happens with the Klopsch's series $\sigma_{\mathrm{K},m}$ for such values of $m$, cf.\ Remark \ref{statesklopsch}. In all these families, the number of states appears to be logarithmic or polynomial in the genus, and never exponential, as is theoretically possibly by Bridy's bound discussed in Section \ref{cxproperties}. 
\end{remark}

\subsection{Quasi-sparse series} 
Sparse series form an $\F_2[t]$-algebra that we will denote by $S$. Consider the larger $\F_2[t]$-algebra $\hat{S}$ 
consisting  of power series in $\F_2\fl t \fr$ that can be written as products of sparse series and rational functions in $\F_2(t)$.  Elements of this algebra can also be regarded as having nice `closed formulas'. We have the following characterisation: 

\begin{proposition}\label{hatA} Let $\sigma=\sum_{k\ge 0}a_k t^k \in \F_2\fl t \fr$ be a power series. The following conditions are equivalent\textup{:} \begin{enumerate} \item\label{hatA1}  $\sigma \in \hat{S}$\textup{;} \item\label{hatA2}  there exists an integer $m\geq 1$ such that $(t^m+1)\sigma$ is sparse\textup{;} \item\label{hatA3}  there exists an integer $m\geq 1$ such that $\sum_{k \geq 0} (a_k + a_{k+m}) t^k $ is sparse\textup{;}\item \label{hatA4} 
there exists an integer $m\geq 1$ such that for all integers $q\geq 0$ the series $\sum_{k \geq 0} (a_k + a_{k+2^q m}) t^k$ is sparse.
\end{enumerate}
\end{proposition}

\begin{proof} Since sparse power series form a ring and include polynomials, $\sigma \in \hat S$ if and only if there exists a nonzero $p\in \F_2[t]$ such that $p\sigma \in S$. Moreover, we may assume that $p$ is not divisible by $t$ since the class of sparse sequences in closed under shifts.  The equivalence of \eqref{hatA1} and \eqref{hatA2} then follows from the fact that every $p\in \F_2[t]$ that is not divisible by $t$ divides the polynomial $t^m+1$ for some $m\geq 1$: take $m=2^{k}(2^r-1)$ with $r$ and $k$ chosen so that the splitting field of $p$ is  $\F_{2^r}$ and every root of $p$ has multiplicity $\leq 2^k$. The equivalence of \eqref{hatA2} and \eqref{hatA3}, with the same value of $m$, is easy. Finally, the equivalence of \eqref{hatA2} and \eqref{hatA4} follows from the fact that if $(t^m+1)\sigma$ is sparse, then so is $(t^m+1)^{2^q}\sigma = (t^{2^{q}m}+1)\sigma$ for all $q\geq 0$.
\end{proof}

A final operation that we allow without affecting our sense of `admitting a closed formula' is for elements of $\hat{{S}}$ to be twisted by an automorphism of $\F_2(t)$, as follows. There is a unique nontrivial field automorphisms of $\F_2(t)$ that is also an element of $\No(\F_2)$,  given by the map $$\varphi \colon t \mapsto t/(t+1).$$ The order of $\varphi$ is two. It might happen that a power series $\sigma\in \No (\F_2)$ is not in $S$ or $\hat S$, but that $\sigma \circ \varphi$ is. This is equivalent with $\sigma$ being in the algebra of sparse series \emph{in the variable $t/(t+1)$}. Note that while composing with $\varphi$ preserves the property of being an algebraic power series (if $\us$ is a root of $F(t,X)$, then $\us \circ\varphi$ is a root of $F(\varphi(t),X)$), the property of being of finite order need not be preserved.

\begin{definition} A series $\sigma=\sigma(t) \in \F_2\pau{t}$ is called \emph{quasi-sparse} if either $\sigma \in \hat{S}$ or $\sigma \circ \varphi \in \hat{S}$. We denote the collection of quasi-sparse series by $\QS$.  
\end{definition} 

This leads to a \emph{hierarchy of complexity} for power series 
$$ 
S \subset \hat{S} \subset \QS \subset  \F_2\pau{t},
$$
where every inclusion is strict. In the next two sections, we will study whether our series $\sigma$ of finite order are in $S, \hat S$ or $\QS$. The next section will employ field-theoretic methods, whereas the following one will be based purely on characterisations in terms of automata. We believe both methods have their merits.

\section{Detecting sparseness properties using field theory} \label{sparsefield}

\subsection{Field-theoretic characterisation of sparseness} Recently, Albayrak and Bell \cite[Thm.\ 1.1(b)]{Bell-sparse} gave an exact field-theoretic characterisation of sparseness for generalized (Hahn) power series in arbitrary positive characteristic. We will use a special case of one direction of their characterisation, of which we include a short, self-contained proof.
 
 The following result will be used without further reference.  
 
 \begin{lemma} For any algebraic power series $\tau \in \bar{\F}_2\fl t \fr$, the field extension $\F_2(t)(\tau)/\F_2(t)$ is separable.
 \end{lemma} 
 
 \begin{proof} 
 If the extension is not separable, the minimal polynomial $f \in \F_2(t)[X]$ of $\tau$ is of the form $f=\sum c_i(t) X^{2i}$. Since the Cartier operator satisfies $\mathcal C_r(\psi \tau^2 )=\tau \mathcal C_r(\psi)$, applying this to the equation $f(\tau)=0$, we find that $\sum \mathcal C_r(c_i(t)) \tau^{i}=0$. This gives a polynomial of strictly smaller degree satisfied by $\tau$ and nonzero for at least one value of $r \in \{0,1\}$. This contradiction shows the result. 
 \end{proof} 
 
\begin{proposition}[{Albayrak--Bell \cite{Bell-sparse}, special case}] \label{ABetter} Let $\sigma \in \No(\F_2)$ denote a power series  that is algebraic over $\F_2(t)$. 
Consider the field $$\mathcal F = \bigcup\limits_{\substack{\ell \geq 1,\\ \ell \, \text{{\rm odd}}}} \bar \F_2(t^{1/\ell}),$$ where $\bar \F_2$ is an algebraic closure of $\F_2$.  If $\sigma$ is sparse, then the following conditions hold\textup{:} \begin{enumerate} \item \label{ABetter1} $\sigma$ is integral over $\bar{\F}_2[t,t^{-1}]$\textup{;} \item \label{ABetter2} the extension  $\bar{\F}_2(t)(\sigma)/\bar{\F}_2(t)$ is unramified outside of $0,\infty$\textup{;} \item  \label{ABetter3} the splitting field of $\sigma$ over $\mathcal F$ has degree a power of two. \end{enumerate}
\end{proposition}

\begin{proof} The essence of the proof is to show that for sparse power series the combinatorial structure of the support $E(\sigma)$ allows one to construct a tower of  Artin--Schreier extensions of $\mathcal F$ that contains $\sigma$. 

By Proposition 8.2 a series $\us$ is sparse precisely if $E(\us)$ is a finite union of pairwise disjoint simple sparse sets. Properties \eqref{ABetter1}--\eqref{ABetter3} hold for the sum of several power series whenever they hold for the individual summands (for unramifiedness, use \cite[Cor.\ 3.9.3]{Stichtenoth}), and hence it is sufficient to prove that they hold for power series with simple sparse support. This will be done by induction on the rank of sparseness $r$.

Suppose that the support of $\us$ is a simple sparse set, consisting of integers whose base-$2$  expansion is of the form $v_r w_r^{\ell_r} \cdots v_1 w_1^{\ell_1} v_0$ with $\ell_i \in \Z_{\geq 0}$  for some fixed binary words $v_0,\dots,v_r,w_1,\dots,w_r$. If $r=0$, then $\us$ is a monomial, and properties  \eqref{ABetter1}--\eqref{ABetter3} hold. Suppose that $r\geq 1$ so $w_1$ is nontrivial. Let  $k_0 = |v_0|$ and $k_1=|w_1|$ be the lengths of the words $v_0$ and $w_1$, and let  $m_0$ and $m_1$ be the integers whose base-$2$ expansion is $v_0$ and $w_1$. Let $\tau$ be the power series whose support consists of the integers with base-$2$ expansion of the form $v_r w_r^{\ell_r} \cdots w_2^{\ell_2} v_1 0^{k_0}$ with $\ell_i \in \Z_{\geq 0}$. By induction, we know that properties  \eqref{ABetter1}--\eqref{ABetter3} hold for $\tau$. The relation between the supports of $\sigma$ and $\tau$ leads directly to the formula \begin{equation} \label{eqn:ABproof} \sigma^{2^{k_1}}-t^{(2^{k_1}-1)m_0-2^{k_0}m_1}\sigma = t^{2^{k_1}m_0-2^{k_0}m_1}\tau.\end{equation} This allows us to deduce the properties  \eqref{ABetter1}--\eqref{ABetter3} for $\us$ from the corresponding properties of $\tau$.

 First of all, $\sigma$ is integral over $\bar{\F}_2[t,t^{-1}][\tau]$, and hence also over $\bar{\F}_2[t,t^{-1}]$. 
 
 Secondly, the form of Equation \eqref{eqn:ABproof} makes it very easy to compute the ramification of the extension $\bar{\F}_2(t)(\sigma)/\bar{\F}_2(t)(\tau)$. If $f$ is the minimal polynomial of $\sigma$, then \cite[Cor.\ 3.5.11]{Stichtenoth} implies that the extension is unramified at all places $P$ for which $f$ is $P$-integral and $v_P(f'(\sigma))=0$. The same result then holds for any monic (not necessarily minimal) polynomial $g$ satisfied by $\sigma$, since it is divisible by $f$. We apply this with $g$ the polynomial in $\sigma$ given in \eqref{eqn:ABproof}, and we find that the extension is unramified at all places $P$ of $\bar{\F}_2(t)(\tau)$ with $v_P(t)=0$ and $v_P(\tau)\geq 0$, and that for all places $P'$ of $\bar{\F}_2(t)(\sigma)$ lying above such $P$ we have $v_{P'}(\sigma)\geq 0$. This implies that $\bar{\F}_2(t)(\sigma)/\bar{\F}_2(t)$ is unramified outside of $0,\infty$.
 
Finally, multiplying Equation \eqref{eqn:ABproof} by an appropriate (fractional) power of $t$ leads to an equation of the form $(t^c\sigma)^{2^{k_1}}-(t^c \sigma)=t^d \tau$ for some $c$ and $d$, which are rational numbers with odd denominators (more precisely, $c=-m_0+\frac{2^{k_0}m_1}{2^{k_1}-1}$ and $d=\frac{2^{k_0}m_1}{2^{k_1}-1}$). This shows that the extension $\mathcal{F}(\sigma)/\mathcal{F}(\tau)$ is contained in a tower of Artin--Schreier extensions, and hence so is $\mathcal{F}(\sigma)/\mathcal{F}$. Thus, its Galois closure is of degree a power of two.
\end{proof}

\subsection{Field-theoretic test for membership in the hierarchy} From Proposition \ref{ABetter}, we can deduce a method of establishing that a series is not in $S$ or $\hat S$. Since the properties \eqref{ABetter2} and \eqref{ABetter3} depend only on the field $\bar{\F}_2(t)(\us)$, and not on $\us$ itself, any proof that uses them to show that $\us \not \in S$ will establish the stronger property that $\us\not \in \hat{S}$. Actually, the method we will use to show that for a particular $\sigma$ property \eqref{ABetter3} does not hold will even show that $\sigma \not \in \QS$. On the other hand, the integrality property \eqref{ABetter1} will be used to show that certain series are in $\hat S$, but not in $S$. 

The basic ingredient is the following field-theoretic result, restricting possible factorisations of polynomials after extension of the base field. 

\begin{lemma} \label{Gal} Let $L/K$ be a (possibly infinite) Galois extension with Galois group $G$, let $f\in K[X]$ be a monic irreducible polynomial, and let $g\in L[X]$ be a monic irreducible factor of $f$ in $L[X]$. Denote by $H$ the stabiliser of $g$ in $G$. Then 
\begin{equation*} \label{prodconj} 
f =\prod_{\phi \in G/H} g^{\phi}, 
\end{equation*} 
i.e.\ $f$ is the product of all (pairwise distinct) Galois conjugates $g^\phi$ for $\phi$ running through the coset space $G/H$. 

\end{lemma} 

\begin{proof} 
 Let $\alpha$ denote a root of $g$ in an algebraic closure of $L$; then $g$ is the minimal polynomial of $\alpha$ over $L$. Put 
 $\tilde{f}:=\prod g^{\phi}$, the product being taken over all $\phi$ running through the coset space $G/H$. By construction, $\tilde{f}$ lies in $K[X]$ and has $\alpha$ as a root, hence $f$ divides $\tilde{f}$. Conversely, $g$ divides $f$ in $L[X]$, and hence so does $g^{\phi}$ for all $\phi \in G$. Since the elements $g^{\phi}$ are irreducible and pairwise distinct for $\phi \in G/H$, the polynomial $\tilde{f}$ divides $f$. Hence, $f=\tilde{f}$.
\end{proof} 

This implies the following valuation-theoretic result that can be used to check whether a polynomial stays irreducible under base field extension. 

\begin{lemma} \label{NP} Let $L/K$ be a (possibly infinite) Galois extension with Galois group $G$, and let $v\colon L\to \R\cup \{\infty\}$ be an (additive) valuation that is $G$-invariant, in the sense that $v\circ \phi = v$ for all $\phi \in G$. Let $\bar{L}$ be an algebraic closure of $L$, and let  $\tilde v$ be an extension of the valuation $v$ to $\bar L$. For a polynomial $f\in K[X]$, denote by $V_v(f)$ the multiset of valuations $\tilde{v}(\alpha)$ of all the roots $\alpha$ of $f$ in $\bar{L}$. If $f$ is irreducible over $K$, but becomes reducible over $L$, then the multiplicities of the elements of $V_v(f)$ have a nontrivial common divisor. \end{lemma}

Elements of the set $V_v(f)$ are minus the slopes of the Newton polygon $\mathrm{NP}(f)$ of $f = \sum\limits_{i=0}^n a_i X^i$, where $\mathrm{NP}(f)$ is given as the lower convex hull in $\R^2$ of the set of points $(i, v(a_i))$ for $0\leq i\leq n$.
\begin{proof}  Since we assume that $v \circ \phi = v$, we have $\mathrm{NP}(g^\phi)= \mathrm{NP}(g)$.
Since the multiset $V_v(h)$ of a polynomial $h$ is determined by its Newton polygon (and hence by the valuations of its coefficients), it follows from the decomposition $f=\prod\limits g^{\phi}$ as in  Lemma \ref{Gal} that $V_v(f)$ is the union of $[G : H]>1$ copies of $V_v(g)$ (as multisets).\end{proof}

\begin{proposition} \label{method} Let $f \in \F_2(t)[X]$ be a separable irreducible polynomial. If the multiplicities of the elements of the multiset $V_t(f)$ for the $t$-adic valuation have no nontrivial common divisor, then $f$ remains irreducible over $\mathcal F$.  
\end{proposition}

\begin{proof}
The extension ${\mathcal{F}}{ /}{\F_2(t)}$ is Galois. The $t$-adic valuation on $\F_2(t)$ has a unique extension to $\mathcal F$ (which coincides on each $\bar{\F}_2(t^{1/j})$ with the $t^{1/j}$-adic valuation $v$ normalised so that $v(t^{1/j})=1/j$). By uniqueness, this extension is Galois invariant. The claim follows from  Lemma \ref{NP}.
\end{proof} 

\begin{corollary} \label{odddeg}
Let $\sigma \in \No(\F_2)$ denote a power series that is algebraic over $\F_2(t)$ with minimal polynomial $F(t,X)$. Assume that $F$ is  of degree not a pure power of two, and that the multiplicities of the elements of the multiset $V_t(\sigma):=V_t(F)$ for the $t$-adic valuation have no nontrivial common divisor. Then $\sigma \notin \QS$.  
\end{corollary}

\begin{proof} 
We conclude from Proposition \ref{method} that $F$ is the minimal polynomial of $\sigma$ over $\mathcal F$, and so $[\mathcal F(\sigma):\mathcal F]$ is not a pure power of two, contradicting Proposition \ref{ABetter}\eqref{ABetter3}. Hence $\sigma \notin S$. Since the field $\mathcal{F}(\us)$ does not change after multiplying $\us$ by a rational function, we get that $\sigma \notin \hat S$.

For the final claim, observe that replacing $\sigma$ by $\sigma \circ \varphi$ changes neither the degree of the minimal polynomial of $\sigma$ over $\F_2(t)=\F_2(t/(t+1))$ nor the set $V_t(\sigma)$. Hence the same reasoning applied to $\sigma \circ \varphi$ shows that $\sigma \notin \QS$. 
\end{proof}

\begin{corollary} \label{deg4}
Let $\sigma \in \No(\F_2)$ denote a power series that is algebraic over $\F_2(t)$ with minimal polynomial $F(t,X)$ of degree $4$
$$F(t,X)=a_4 X^4 +a_3 X^3 + a_2 X^2 + a_1 X + a_0$$
and with cubic resolvent 
$$R_3[F]:=a_4^3 X^3+a_2 a_4^2 X^2 + a_1 a_3 a_4 X + a_0 a_3^2+a_1^2 a_4.$$
Assume that $R_3[F]$ is irreducible over $\F_2(t)$ and that the multiplicities of the elements of the multisets $V_t(F)$ and $V_t(R_3[F])$ for the $t$-adic valuation have no nontrivial common divisor. Then $\sigma \notin \QS$.  
\end{corollary}

\begin{proof} 
The possible Galois groups of an irreducible separable quartic are $S_4$, $A_4$ $D_4$, ${\Z}/{4}{\Z}$ and ${\Z}/{2}{\Z} \times {\Z}/{2}{\Z}$. Only the last three of these are $2$-groups, and those occur precisely when the cubic resolvent is reducible (see \cite[Thm.\ 3.4]{Kconrad}). 

Since $F$ is separable, so is its cubic resolvent $R_3[F]$. From the hypotheses and Proposition \ref{method}, we conclude that $F$ and $R_3[F]$ are irreducible over $\mathcal F$. Therefore, the Galois group of $\sigma$ over $\mathcal F$ is not a $2$-group, and $\sigma \notin S$ by  Proposition \ref{ABetter}\eqref{ABetter3}. Since this argument uses only the information about the field $\mathcal{F}(\us)$, we conclude that $\sigma \notin \hat{S}$. 

Finally, since changing $\sigma$ to $\sigma\circ \varphi$ affects neither the irreducibility of $F$ and $R_3[F]$ nor the sets $V_t(F)$ and $V_t(R_3[F])$, we find similarly that $\sigma \circ \varphi \notin \hat S$, and so $\sigma \notin \QS$. 
\end{proof}
\begin{theorem}  \label{enumarid} We have the following membership properties (see also Table \textup{\ref{tableeqs}})\textup{:}  
\begin{enumerate}
\item \label{enumarid1} $ \sigma_{\mathrm{S},2^\mu\pm1} (\mu \geq 1), \us_{\mathrm{CS}}^{\circ 3}, \sigma_{\mathrm{T},1},\dots,\sigma_{\mathrm{T},4} \in S$\textup{;}
\item  \label{enumarid2} $ \us_{\mathrm{CS}}^{\circ 2}, \us_{\mathrm{CS}} \in \hat{{S}} \setminus S$\textup{;}
\item  \label{enumarid3} $ \us_{\mathrm{J}},  \us_{\mathrm{J}}^{\circ 3} \in \QS\setminus \hat{{S}}$\textup{;}
\item  \label{enumarid4} $\us_{\mathrm{K},m} (m\geq 3),  \sigma_{V,1}, \sigma_{V,2}, \sigma_{V,3}, \us_{\mathrm{min}}\textup{,} \us_{(1,5)}, \sigma_{(1,9)}, \us_8 \notin \QS$.
\end{enumerate}
\end{theorem}

\begin{proof}
The series $\sigma_{\mathrm{S},2^\mu\pm1}$ are sparse by Proposition \ref{KSprop}. The sparseness of the series $\sigma_{\mathrm{CS}}^{\circ 3},\sigma_{\mathrm{T},1},\dots \sigma_{\mathrm{T},4}$  follows by representing $E(\sigma)$ in the same way as was done for $E(\us_{\mathrm{CS}}^{\circ 3})$ in Example \ref{exes}, using the closed formulas for the series in Table \ref{s5-s8}. 

The series $\sigma_{\mathrm{CS}}^{\circ 2}$ and $\sigma_{\mathrm{CS}}$ are not sparse by Proposition \ref{ABetter} since their minimal polynomials are not $\bar \F_2[t,t^{-1}]$-integral. To show the series are in $\hat S$, we have the following explicit relations, obtained from Remark \ref{rem:cs2} and  Equation \eqref{cseq}, with sparse right hand side: 
$$ (t+1)^2 \us_{\mathrm{CS}}^{\circ 2} = t+t^3 + \sum_{k \geq 1} \left( t^{2\cdot 2^k}+ t^{3 \cdot 2^{k}} \right) \qquad \mbox{and}\qquad
 (t+1)^2\us_{\mathrm{CS}} = \sum\limits_{k \geq 0}\left(t^{2^k}+t^{3\cdot 2^k}\right). $$

If $\sigma$ is any of the series $\sigma_{\mathrm{J}}$ and $\sigma_{\mathrm{J}}^{\circ 3}$, then it is not in $\hat S$.  Indeed, from their minimal polynomial we can read out that the extension ${\bar{\F}_2(t)(\sigma)}{/}{\bar{\F}_2(t)}$ is ramified above $t+1$, and the conclusion follows from Proposition \ref{ABetter}\eqref{ABetter2}. 
To prove the series are in $\QS$, we use the following explicit relations with sparse right hand side: 
 $$ (t+1)\us_{\mathrm{J}}(\varphi(t)) = (t+1)^2\us_{\mathrm{CS}}(t) \qquad \mbox{and}\qquad (t^2+t)  \us_{\mathrm{J}}^{\circ 3}\left(\varphi(t)\right) = \sum_{k \geq 0} \left( t^{3 \cdot 2^k} + t^{2\cdot 2^k} \right). $$
Indeed, for the former equation, one verifies that $\sigma_{\mathrm{CS}}$ and $\sigma_{\mathrm{J}}(\varphi(t))/(t+1)$ are equal, since they satisfy the same irreducible algebraic equation (\ref{cseq}) having a unique solution $t+O(t^2)$. For the latter equation, the left hand side is the unique solution to $\tau^2+\tau=t^3+t^2$ of the form $t^2+O(t^3)$. But this solution is clearly equal to the right hand side. 

To prove that $\sigma_{\mathrm{K},m} \notin \hat S$ for any odd $m\geq 3$, we use Proposition \ref{ABetter}\eqref{ABetter3}. To this end, it suffices to check that $(t^m+1)X^m+t^m$ is irreducible over $\mathcal F$, which by \cite[Prop.\ 3.7.3]{Stichtenoth} is equivalent to showing that $t^m/(t^m+1)$ is not a $d$-th power in $\mathcal{F}$ for any $d>1$, $d{\mid}m$, or, equivalently, that $t^{mj}/(t^{mj}+1)$ is not a $d$-th power in $\F_2(t)$ for any odd $j$. This holds since $t^{mj}+1$ has only simple roots in $\bar\F_2$. Similarly, $\sigma_{\mathrm{K},m}  \circ \varphi$ satisfies $(t^m + (t+1)^m)X^m+t^m=0$, and the polynomial $t^{mj}+(t+1)^{mj}$ has only simple roots in $\bar \F_2$ (as can be seen from computing its derivative); hence  for the same reason $\sigma_{\mathrm{K},m}\circ \varphi \notin \hat S$. We conclude that $\sigma_{\mathrm{K},m} \notin \QS$.

The multisets of slopes for the minimal polynomials of $\sigma_{V,1}$, $\sigma_{V,2}$ and $\sigma_{V,3}$ can be found in Table \ref{tableeqs}. The cubic resolvent for the minimal polynomial of $\sigma_{V,1}$ is 
$t^{12}X^3+t^8 X^2 + t^7(t+1) X + t^4(t^4+t^3+t^2+1)$, which is irreducible over $\F_2(t)$ with $v_t$-slopes $\{ -4,(-2)^2\}$. (A convenient way to check irreducibility of the cubic resolvent over $\F_2(t)$ is to consider the $v_{t^{-1}}$-slopes for  the $t^{-1}$-adic valuation.)   
 Similarly, the minimal polynomial for $\sigma_{V,2}$ has resolvent 
$(t+1)^{12}X^3+t(t+1)^8X^2+(t+1)^4 t^4$, which is irreducible over $\F_2(t)$ and has $v_t$-slopes $\{ 1,(3/2)^2\}$, and  the minimal polynomial for $\sigma_{V,3}$ has resolvent 
$t^{12}X^3+t^9(t^2+t+1)X^2+t^4(t+1)^6 X+t(t+1)^6(t^3+t^2+1)$, which is irreducible over $\F_2(t)$ and has $v_t$-slopes $\{ (-4)^2,-3\}$.  By Corollary \ref{deg4} we conclude that $\sigma_{V,1},\sigma_{V,2}, \sigma_{V,3} \notin \QS$. 

For all further series, $\deg F$ is not a pure power of $2$ and $V_t(F)$ has no nontrivial common divisor of multiplicities (listed in Table \ref{tableeqs}), so we immediately conclude that $\sigma \notin \QS$ by Corollary \ref{odddeg}.  
 This finishes the proof. 
\end{proof}

\def\arraystretch{1.4}
\begin{table}
{\small 
\hspace*{-1cm}  \begin{tabular}{llllll}
\hline 
series &   {\footnotesize $\in S$} & {\footnotesize $\in \hat{S}$} & {\footnotesize $\in \QS$} & minimal polynomial  $F$ & method \\
\hline 
$ \sigma_{\mathrm{S},1} $  & \checkmark (1) & \checkmark & \checkmark  & $ t^2 X^2 + (t+1)X+t$ & direct  \\
$ \sigma_{\mathrm{S},m=2^\mu-1>1}$ & \checkmark (1) & \checkmark & \checkmark  & $ (tX)^{(m+1)/2} + X+t$ & direct  \\
$ \sigma_{\mathrm{S},m=2^\mu+1}$ & \checkmark ($\mu$) & \checkmark & \checkmark  & $ (tX)^{m-1} + X^{m-2}+t^{m-2}$ & direct  \\
$\us^{\circ 3}_{\mathrm{CS}}$  & \checkmark (1) & \checkmark & \checkmark  & $ t^2 X^2 + X + t^2 +t $ & direct  \\
$\sigma_{\mathrm{T},1}$ &  \checkmark (2) & \checkmark & \checkmark & $t^2X^4+(t^4+t^2+t+1)X^2+(t^3+t^2+t)X+t^3$ &  direct \\
$\sigma_{\mathrm{T},2}$ &  \checkmark (3) & \checkmark & \checkmark &  $t^2X^4+(t+1)X^3+(t^4+t^2+t)X^2+(t^2+t)X+t^2$ & direct\\
$\sigma_{\mathrm{T},3},\sigma_{\mathrm{T},4}$ &  \checkmark (3) & \checkmark & \checkmark & $t^4X^4+(t^2+1)X^3+(t^3+t)X^2+t^2X+t^3$ &  direct \\
\hline
$\us_{\mathrm{CS}}^{\circ 2}$  &   $\times$ & \checkmark & \checkmark  & $(t+1)^2 X^2 + X + t^2 + t$ & not integral \\
$\us_{\mathrm{CS}}$ &  $\times$ & \checkmark & \checkmark  & $(t+1)^2 X^2 + X + t$ & not integral   \\
\hline
$\us_{\mathrm{J}}$ &  $\times$ & $\times$ & \checkmark & $(t+1)X^2 + (t^2+1)X + t$ & not unramified \\
$\us^{\circ 3}_{\mathrm{J}}$ &    $\times$ & $\times$ & \checkmark & $t X^2 + (t^2+1) X+ t^2 + t$& not unramified  \\
\hline
$\us_{\mathrm{K},m}$ &  $\times$ & $\times$ & $\times$ & { $(t^{m}+1)X^{m}+t^m$} & odd deg \& direct\\ 
$ \sigma_{V,1}$ & $\times$ & $\times$ & $\times$ &$ t^4X^4+t^3X^3+X^2+(t+1)X+t^2+t$ & $R_3$ \& $V_t=\{(-2)^2,0,1\}$ \\
$\sigma_{V,2}$  & $\times$ & $\times$ & $\times$ & $(t+1)^4 X^4+tX^2+t^2X+t^4$ & $R_3$ \& $V_t=\{\left(\frac{1}{2}\right)^2,1,2\}$ \\
$\sigma_{V,3}$  & $\times$ & $\times$ & $\times$ & $t^4 X^4+(t+1)^3 X^3+t(t^2+t+1)X^2+(t+1)^3 X+t(t+1)^2$ & $R_3$ \& $V_t=\{-4,0^2,1\}$ \\
$\us_{\mathrm{min}}$ &  $\times$ & $\times$ & $\times$  & { $(t+1)^3X^3+(t^3+t)X^2+(t^3+t+1)X+t^3+t$ } & odd deg \& $V_t=\{0^2,1\}$ \\
$\us_{(1,5)}$ & $\times$ & $\times$ & $\times$ & { $t^2X^3+(t+1)^3X+t^3+t$} & odd deg \& $V_t=\{(-1)^2,1\}$ \\
$\us_{(1,9)}$ & $\times$ & $\times$ & $\times$ &  { $t^2X^7  + t^3X^6 + (t^5 + t^4 + t^2)X^5 + (t^5 + t^3)X^4 +$} & odd deg \&  \\
&  &   &  &  { $ + (t^7+t^5+t^4+t^3+t)X^3 + t^5X^2 + (t^3 + t + 1)X + t$}   &  $V_t=\{\left(-\frac{1}{3}\right)^6,1\}$  \\  
$\us_8$ &  $\times$ & $\times$ & $\times$  & { $t^6 X^6 + (t^6+t^2) X^4 + (t^6+t^5+t^4+t^3+t^2+1) X^2 +$} &  deg not a power of $2$ \& \\
&  &   &  &     { $ +(t+1)^3 X + t^6 + t^5 + t^2 + t$} &  $V_t=\{(-2)^2,(-1)^2,0,1\}$ \\ 
\hline\\
\end{tabular} 
}
\caption{For each series, in column `$\in S$'  the symbol `$\times$' indicates the series is not sparse and `$\checkmark (r)$' indicates the series is $r$-sparse; the column `$\in \hat{S}$' describes the property of being sparse up to multiplication with a rational function; the column `$\in \QS$' indicates whether or not the series itself or its composition with $t\mapsto t/(t+1)$ is in $\hat{S}$; `minimal polynomial $F$' is the minimal polynomial of the series over $\F_2(t)$; `method' indicates the method of proof, where $V_t:=V_t(F)$ is the multiset of $t$-adic valuations of the roots of $F$.}
\label{tableeqs}
\end{table}
\def\arraystretch{1}


\section{Sparseness and automaton properties} \label{sparseauto}

\subsection{Combinatorial characterisation of sparseness} We describe automaton-theoretic methods to verify whether a series $\sigma$ is in $S, \hat S$ or $\QS$. In \cite{SYZS}, it is shown that sparseness may be checked directly using a corresponding automaton (recall our convention that all states in the automaton are accessible, which is also part of the conditions below). 

\begin{definition} \label{def:tiedvert}
Call a vertex $v$ of an automaton \emph{tied} if the following two properties hold: 
\begin{enumerate} 
\item[(a)] \label{def:tiedvert1} there exists a (possibly empty)  path from $v$ to a vertex with output $1$ [`$v$ is co-accessible']; 
\item[(b)] \label{def:tiedvert2} there exist two different walks of the same length from $v$ to itself.\end{enumerate} 
\end{definition}

\begin{proposition}[{\cite{SYZS}, \cite[Prop.\ 3.4]{BK2020}}] \label{SYZScheck} An automatic series $\us$ is not sparse if and only if there  exists a tied vertex $v$ in a corresponding automaton. \qed \end{proposition} 

This criterion can be used immediately to verify that the series $ \sigma_{\mathrm{S},2^\mu-1} (\mu \geq 1), \sigma_{\mathrm{CS}}^{\circ 3},\sigma_{\mathrm{T},1},\dots, \sigma_{\mathrm{T},4}$ are sparse. 

\begin{example} \label{Klopschsparse}

The $2$-automaton $A$ corresponding to the expansion of the series $(1+t)^{-1/m}$ can be succinctly described as follows.  
Let $\varpi$ denote the multiplicative order of $2$ modulo $m$ and consider the base-$2$ expansion $(2^\varpi-1)/m=\sum_{i=0}^{\varpi -1} x_i 2^i$. 
The set of vertices of $A$ is $\{ v_0,\dots,v_{\varpi-1}, w\}$. All $v_j$ have vertex label $1$, $w$ has vertex label $0$, and $v_0$ is the start vertex. For any $j$, $v_j$ is connected to $v_{j+1\, \mathrm{mod}\, \varpi}$, always by an edge with label $0$, and by an edge with label $1$ exactly if $x_j=1$. If $x_j=0$, an edge with label $1$ connects $v_j$ to $w$. Finally, $w$ has two self-loops labelled $0$ and $1$.  
The automaton $A$ is not sparse since any vertex $v_j$ with $x_j=1$ is not tied: $0^\varpi$ and $0^{\varpi-1}1$ are two paths that satisfy condition (b). 
(This incidentally provides another proof of the non-sparseness of Klopsch's series $\sigma_{\mathrm{K},m} (t) = {t}{/}\!{\sqrt[m]{1+t^m}}$; however, we do not have a synthetic description for an automaton corresponding to $\sigma_{K,m}$ for general $m$ and, in particular, do not have a formula for the minimal number of states as a function of $m$, cf.\ Table \ref{tablebounds}.)

A similar description of a minimal $p$-automaton for $(1+at)^{-1/m} \in \F_p\pau{t}$ for any prime $p$, $m$ coprime to $p$ and $a \in \F_p^*$ is given in \cite{Djurre-thesis}. 
\end{example}

\subsection{Combinatorial tests for membership in the hierarchy} 
We have not been able to find a necessary and sufficient condition for a series to be in $\hat S$ in terms of the automaton alone. We will however give a  simple necessary criterion, from which one may deduce all statements in Theorem \ref{enumarid}, \emph{except} the facts that $\sigma_{\mathrm{min}} \notin \hat S$ and  $\sigma_{(1,9)} \circ \varphi \notin \hat S$.  

In applying the criterion, it is necessary to move the `start' label to other vertices. This might produce non-accessible vertices, which should then be removed from the automaton; this does not affect the resulting automatic sequence.

\begin{proposition} \label{hatA.2} 
Let $\sigma(t)=\sum_{k\geq 0}a_kt^k\in\F_2\pau{t}$ be a power series generated by an automaton $A$. Then $\sigma(t)\notin\hat S$ if there exists a vertex $v$ in $A$ 
satisfying the following two properties\textup{:}
\begin{enumerate}
\item there exist arbitarily long walks from the start vertex to $v$\textup{;}
\item let  $v_0$ and $v_1$ denote the vertices reached by following the edge starting at $v$ and labelled $0$ and $1$, respectively, and let $A_i$ be the automaton obtained from $A$ by changing the start vertex to $v_i$. Then exactly one of the automata $A_0$ and $A_1$ is sparse (and the other one is not sparse).  
\end{enumerate}
\end{proposition}

\begin{remark}\label{rem:hatA.2}
Since the automaton is finite, the existence of arbitarily long walks from the start vertex to $v$ is equivalent to  the existence of paths $w_0$, $w_1$ and $w_2$ such that $w_1$ is nontrivial and for every integer $\ell\geq 0$ the walk $w_2w_1^{\ell}w_0$ goes from the start vertex to $v$.
\end{remark}

\begin{proof}[Proof of Proposition \ref{hatA.2}]

For the purpose of the proof, we let $(n)_2$ denote the base-$2$ expansion of an integer $n\geq 0$. 

Consider a walk from the start vertex to $v$, say of length $\ell$, and let $w$ be the binary word given by the concatenation of its labels. Let $c$ be the integer such that $(c)_2=w$. It follows directly from the definition that the automatic sequences produced by $A_0$ and $A_1$ are $(a_{2^{\ell+1}n+c})_{n\geq 0}$ and $(a_{2^{\ell+1}n+2^\ell+c})_{n\geq 0}$, respectively. Let $i\in\{0,1\}$ be such that the automaton $A_i$ is not sparse; the automaton $A_{1-i}$ is then sparse.

Let $m\geq 1$ be a fixed arbitrary odd integer. Consider integers $k$ of the form $k=k(n)=2^{\ell+1}n+2^{\ell}i+c$ (where $\ell$ and $i=0,1$ are fixed while $n$ runs through $\Z_{\geq 0}$). The base-$2$ expansion of $k+2^{\ell}m$ is of the form $(k+2^{\ell}m)_2 = u(1-i)w$ for some binary word $u$, and hence the walk given by it leads from the start vertex to a vertex in $A_{1-i}$. Since $A_{1-i}$ is sparse, the number of $n\leq N$ such that $a_{k+2^{\ell}m}=1$ grows as $O(\log(N)^r)$ for some $r\geq 0$. On the other hand, the base-$2$ expansion of $k$ is $(k)_2=(n)_2 i w$, the automaton $A_i$ is not sparse, and hence the number of $n\leq N$ such that $a_{k}=1$ grows faster than $\log(N)^r$ for any $r\geq 0$, and so does the number of $n$ such that $a_{k} + a_{k+2^{\ell}m}=1$. It follows that the power series \[\sum_{n\geq 0} \left(a_{k(n)} + a_{k(n)+2^{\ell}m}\right)t^n\] is not sparse, and hence neither is  the series \[\label{eqn:notShatcrit} \sum_{n\geq 0} \left(a_{n} + a_{n+2^{\ell}m}\right)t^n.\] Since the integer $m\geq 1$ was arbitrary odd, and since the walk from the start vertex to $v$ can be chosen with $\ell$ arbitrarily large, we conclude from Proposition \ref{hatA}\eqref{hatA4} that $\sigma$ is not in $\hat S$.\end{proof}

A heuristics to apply  Proposition \ref{hatA.2} is now as follows. To verify that one of the automata $A_0, A_1$ is non-sparse, we can use Proposition \ref{SYZScheck}; for this one can use cycle-finding algorithms. The tricky part is to verify that the other automaton is sparse---to this end, we need to exclude the existence of appropriate walks in the graph. To simplify this problem one may insist that the sparse of the automata $A_0$, $A_1$ be very simple; in fact, in all the examples discussed below it is possible to find such an automaton consisting of only one state, with label $0$, making the verification obvious. Inverting this logic, we can hope to apply the criterion by first finding a vertex $w$ with label $0$ and two self-loops (a so-called `absorbing state', cf.\ Section \ref{nonran} below), and then going through all the vertices $v$ admitting an edge from $v$ to $w$, and checking if any of them satisfies the conditions of Proposition \ref{hatA.2}. 

\begin{proof}[Sketch of a second proof of (part of) Theorem \ref{enumarid}] The verification that certain series belong to $S, \hat S$ or $\QS$ is direct and the same as in the first proof. The verification that certain series do not belong to $S, \hat S$ or $\QS$ can be done by studying the corresponding automata and using Propositions \ref{SYZScheck} and \ref{hatA.2}. We have summarised some of the combinatorial data for this in Tables \ref{proofnotinS} \&\ref{proofnotin}. For small automata, these data can be easily found just by looking at the graphical representation. This is the case for all the series in  Tables \ref{proofnotinS} \&\ref{proofnotin} except for $\sigma_{(1,9)}$. To illustrate how one can use a computer algebra system  to find these data for larger automata, we have written a {\sc Mathematica} notebook doing this for the series $\sigma_{(1,9)}$, generated by an automaton with $110$ states, see \cite{Database}.

To verify that a series is not in $S$, one indicates a path from the start vertex to a tied vertex $v$ and two different walks of the same length from $v$ to itself. One also checks that $v$ is co-accessible by indicating a path from $v$ to a vertex with output $1$. These data are gathered in Table \ref{proofnotinS}. 

To verify that a series is not in $\hat S$ one indicates paths $w_0, w_1, w_2$ such that every walk $w_2w_1^\ell w_0$ leads from the start vertex to the same vertex $v$; a digit $i\in\{0,1\}$ such that the automaton $A_i$ (resp.\ $A_{1-i}$) obtained by moving the start vertex to the endpoint $v_i$ of the edge starting at $v$ and labelled $i$ is non-sparse (resp.\ sparse); a path from $v_i$ to a tied vertex; a path from that tied vertex to a vertex with output $1$; and different walks of the same length from the tied vertex to itself, verifying that the automaton $A_i$ is non-sparse. In all the cases listed in Table \ref{proofnotin} the vertex $v_{1-i}$ has label zero and two self-loops, implying that the automaton $A_{1-i}$ is sparse, and providing the final step of the verification that the considered series is not in $\hat S$.

We have listed the combinatorial data only for some of the series, but a similar procedure can be performed for all the series considered in Table  \ref{tableeqs} except $\sigma_{\mathrm{min}}$ and $\sigma_{(1,9)}\circ \varphi$, which are not in $\hat S$, but for which the criterion from Proposition \ref{hatA.2} is not satisfied. \end{proof}

\begin{table} 
\begin{tabular}{llll}
\hline 
\text{series}  & \text{path} & \text{path to vertex} & $(p_1,p_2)$   \\
 & & \text{with output $1$} &  \\ \hline
$\sigma_{\mathrm{CS}}$ & $0$   & $1$& $(101,100)$  \\
$\sigma_{\mathrm{CS}}^{\circ2}$ & $0$ & $10$ & $(1101,1110)$ \\
$\sigma_{\mathrm{min}}$ &  $1$  & $\epsilon$ & $(011,100)$  \\

\hline \\ 
\end{tabular}
\caption{`path' indicates a path from the start vertex to a tied vertex; `path to vertex with output $1$' indicates a path from the tied vertex to a vertex with output $1$; $p_1$ and  $p_2$ are walks of the same length that connect the tied vertex to itself, indicating that the series in non-sparse; $\epsilon$ indicates the empty path.}
\label{proofnotinS}
\end{table}

\begin{table} 
\begin{tabular}{llllll}
\hline 
\text{series}  & $(w_2,w_1,w_0)$ & \text{edge to} & \text{path}  & \text{path to vertex } & $(p_1,p_2)$ \\
 && \text{non-sparse} && \text{with output $1$} &    \\ \hline
$\sigma_{\mathrm{J}}$ & $(1, 0, 00)$ & $1$ & $\epsilon$  & $\epsilon$  & $(0,1)$  \\
$\sigma_{\mathrm{J}}^{\circ 3}$ & $(1, 0, 001)$ & $1$ & $\epsilon$  & $\epsilon$ & $(0, 1)$ \\
$\sigma_{V,1}$ & (1, 0,000) & $0$ & $\epsilon$ & $1$ & $(1001,0100)$   \\
$\sigma_{V,2}$ & $(1, 0, 1)$ & $0$ & $\epsilon$ & $1$ & $(1001,0100)$   \\
$\sigma_{\mathrm{K},3}$ & $(\epsilon, 00, 0)$ & $0$ & $01$  & $\epsilon$ & $(00, 11)$  \\
$\sigma_{(1,5)}$ & $(1, 0, 001)$ & $0$ & $1$  & $\epsilon$ & $(11001,01011)$  \\
$\sigma_{(1,9)}$ & $(0^5 1010,1,1^30^3)$ & $1$ &  $001$ & $1$ & $(0^2 1^2 0^5 1^2 0^2 10^2 101^2 0^2,$  \\
&&&&\multicolumn{2}{r}{$0^3 1^2 0^2 10^2 101^20^4 1^2 0^2 )$} \\
\hline \\ 
\end{tabular}
\caption{ The words $w_i$ are the words needed to apply Remark \ref{rem:hatA.2}; `edge to non-sparse' has value $i\in\{0,1\}$ if the automaton $A_i$ considered in Proposition \ref{hatA.2} is non-sparse; `path' indicates a path from the vertex $v_i$ from Proposition \ref{hatA.2} to a tied vertex; `path to vertex with output $1$' indicates a path from the tied vertex to a vertex with output $1$; $p_1$ and $p_2$ are walks of the same length that connect the tied vertex to itself, indicating that the series is not in $\hat S$; $\epsilon$ indicates the empty path.}
\label{proofnotin}
\end{table}
\section{`Non-randomness' of the series and synchronisability of the automata} \label{nonran}

\subsection{Synchronising automata} Recall that an automaton is called \emph{synchronising} if there is an input string (a `synchronising word' $p_{\mathrm{sync}}$), which, when followed from an arbitrary vertex, always leads to the same end vertex; this means that the word resets the automaton---if the base-$2$ expansion of $n$ contains the word $p_{\mathrm{sync}}$, the corresponding coefficient $a_n$ depends only on the part of the expansion that is to the left of the occurrence of $p_{\mathrm{sync}}$. 

\begin{example} 
The word $1011$ is synchronising for $\sigma_{\mathrm{K,3}}$. Following this word (right to left) starting at any state of the automaton leads to the state in the middle of the bottom row of Figure \ref{ka1}.
\end{example}

Synchronisation is particularly easy to check when there is an \emph{absorbing state} $v$, meaning that both outgoing edges from $v$ are loops.
\begin{lemma} \label{absorbing} If an automaton $A$ has an absorbing state $v$, then $A$ is synchronising if and only  for any vertex $w$ in $A$ there is a path from $w$ to $v$ (in particular, $A$ is not synchronising if there is more than one absorbing state). 
\end{lemma}

\begin{proof}
Since $v$ is mapped to itself by any word, the synchronising word should map any vertex to $v$. In particular, for $A$ to be synchronising, any vertex needs to be connected by a path to $v$. If this holds, choose an input string $p$ for which the number of end vertices of all paths with label $p$ and arbitrary beginning vertex is minimal. If the only such end vertex is the absorbing state $v$, $p$ is a synchronising word. If not, let $v_1$ denote another such end vertex and choose a path $p_1$ from $v_1$ to $v$ (which exists by assumption). Now, the number of end vertices of paths with label $p_1 p$ is strictly smaller than for $p$ (since both $v_1$ and $v$ are connected to $v$ by a path with label $p_1$), contradicting the minimality. 
\end{proof}

As the number $N$ tends to infinity, the fraction of synchronising automata with $N$ states tends to $1$ \cite{Berlinkov}, but the fraction of automata with $N$ states having an absorbing state tends to $0$.  The next lemma shows something very different happens for the class of minimal \emph{sparse} automata. 

\begin{lemma} \label{absorbingsparse}
If an automaton $A$ is minimal and sparse, then $A$ has a unique absorbing state $v$, and for any vertex $w$ in $A$ there is a path from $w$ to $v$.
\end{lemma}
\begin{proof}
Call any maximal subgraph of $A$ that is connected as a directed graph a \emph{strongly connected component}. For example, any absorbing state is a strongly connected component. 

Let $U$ denote the union of all strongly connected components. For any vertex $v$ of $A$ let $n(v)$ be the number of vertices that can be reached from $v$ by following some directed path. It is easy to see that if for some vertex $w$ there is a path from $v$ to $w$, then $n(w) \leq n(v)$, and that equality holds for all such $w$ exactly if  $v$ lies in $U$. Choosing $w$ to be a vertex admitting a path from $v$ to $w$ for which the value of $n(w)$ is minimal, we see that for any vertex there is a path from that vertex to a vertex in $U$. An argument analogous as in the proof of Lemma \ref{absorbing} (but with $U$ playing the role of the vertex $v$ in that proof) shows that there is an input string $p$ such that for every path with label $p$ originating from any vertex, the end vertex lies in some strongly connected component.

We now assume that $A$ is \emph{sparse}, and we claim that then all vertices in $U$  have vertex label $0$. Indeed, by the combinatorial criterion in Proposition \ref{SYZScheck}, $A$ has no tied vertices, but any vertex $v$ with label $1$ lying in some strongly connected component is tied: by strong connectedness, two directed edges starting at $v$ with different labels can each be continued to paths $p$ and $q$ leading back to $v$, and then $pq$ and $qp$ are two different paths of the same length connecting $v$ to itself.

Thus, the automaton $A'$ obtained by replacing every vertex in $U$ with a single absorbing state with vertex label $0$ produces the same output as $A$. We conclude that if $A$ is \emph{sparse and minimal}, it has only one strongly connected component, and this component is an absorbing state with label $0$. \end{proof}

\subsection{`Non-randomness'} A power series corresponding to a synchronising automaton with an absorbing state is not `random' at all: if the binary expansion of $n$ contains a synchronising word $p_{\mathrm{sync}}$ leading to an absorbing state, the corresponding coefficient $a_n$ will always be the same, namely the output value of the absorbing state. Since most integers have binary expansions containing $p_{\mathrm{sync}}$, it follows that $a_n$ is constant for `almost all' $n$, i.e.\  there is some $c>0$ such that $a_n$ takes the same value for all except $O(N^{1-c})$ values $n<N$.

So far, we used the convention that our automata were leading-zero invariant, which we now drop. In order to produce automatic sequences from automata, we  used the backwards-reading convention (starting from the least significant digit), and sequences obtained in this manner from synchronising automata may be more properly called \emph{backwards synchronising} to distinguish them from the forwards-reading convention (starting from the most significant digit), which leads to the notion of a \emph{forwards synchronising} automatic sequence. For a given sequence, these two notions are not equivalent (the sequence $(n \bmod 2)$ is forwards synchronising, but not backwards synchronising, and we will see below that the sequence of coefficients of the series $\sigma_{\mathrm{min}}$ is backwards synchronising, but not forwards synchronising). With both of these notions at hand, we may now refer to the following precise result about structured versus random sequences. 
 In \cite[Thm.\ C]{BKM} it was shown that any $\C$-valued automatic sequence (such as our sequences with the output alphabet $\F_2$ lifted to $\{0,1\} \subset \C$) can be decomposed as a sum of a `structured sequence', in which the $n$-th coefficient is a function of the $n$-th coefficients of a periodic sequence and forwards and backwards synchronising sequences, and a `random  sequence', meaning a highly Gowers uniform sequence. (Since in this sense sequences that are $0$ almost everywhere are `random', the terminology is somewhat loose.) The classical Thue--Morse sequence is an example of a highly Gowers uniform sequence \cite{Kon}. By contrast, it turns out that our sequences are very structured and non-random in the sense of this decomposition. As an example, consider the series $\sigma_{\mathrm{CS}}$: it follows from Equation \eqref{cseq} that the value of its $n$-th coefficient for $n\geq 3$ depends  only on the two leading digits and the final digit of the base-$2$ expansion of $n$.

\begin{proposition}
For all series $\sigma=\sum a_n t^n$ in Table \textup{\ref{tableeqs}} the sequence $(a_n)$ is \emph{structured}\textup{:} there exists a backwards synchronising sequence $(b_n)$, a forwards synchronising sequence $(f_n)$ and a function $F \colon  \F_2^2 \rightarrow \F_2$ such that $a_n=F(b_n,f_n)$ for all $n$.
\end{proposition}

\begin{proof} All series in Table \ref{tableeqs} except $\sigma_{\mathrm{min}}$, $\us_{\mathrm{CS}}^{\circ 2}$, $\us_{\mathrm{CS}}$, $\us_{\mathrm{J}}$ and $\us^{\circ 3}_{\mathrm{J}}$ are produced by automata that admit an absorbing state that is accessible from any other state of the automaton, and hence by Lemma \ref{absorbing}, they are (backwards and forwards) synchronising. Indeed, for small automata, one may inspect the pictures; for the larger automata, the verification can be found in \cite{Database}; for the series $\sigma_{\mathrm{S},2^{\mu}+1}$, for which we have not given a representation of the corresponding automata, one may rely on their sparseness and invoke Lemma \ref{absorbingsparse}.

To treat the remaining cases, we observe the following. The minimal automaton corresponding (in backwards-reading convention) to $\sigma_{\mathrm{min}}$ is  synchronising with synchronising word $1^3$, and so the corresponding sequence is backwards synchronising (using  \cite[Lemma 3.2]{BKM} it can be proven that it is not forwards synchronising). The automata corresponding to $\us_{\mathrm{J}}$ and $\us^{\circ 3}_{\mathrm{J}}$ have two absorbing states, and every state has a path to one of these two states; this is enough to conclude that these sequences are forwards synchronising (cf.\ \cite[Lemma 3.2]{BKM}). Finally, the automata for $\us_{\mathrm{CS}}^{\circ 2}$ and $\us_{\mathrm{CS}}$ have two subgraphs that are synchronising and the start vertex is connected by an outgoing edge to these two subgraphs; it follows that the value $a_n$ of the corresponding sequence depends on the value of a backwards synchronising sequence (the sequence produced by the product automaton for the subgraphs) and on the value of the sequence $(n \bmod 2)$, which is forwards synchronising.\end{proof}

Synchronisability is not invariant under conjugation of the corresponding power series, so one may wonder whether every conjugacy class of elements of finite order in $\No(\F_2)$ has a synchronising representative. 

\section*{How computations and visualisations were done}
{\footnotesize
\begin{itemize}[leftmargin=*]
\item Equations and uniformisers were computed by hand. {\sc Singular} or {\sc Mathematica} were used for elimination of variables and checking irreducibility of equations.
\item Automata were generated in {\sc Mathematica} by Rowland's package \cite{Rowland}. Shapes of automata were verified using the {\sc Magma} code in \cite{LDG}. This  code was also used to compute the number of states of certain automata that were not computed in further detail.
\item Automata were redrawn using tikz and Evan Walace's Finite State Machine Design app ({\tt github.com/evanw/fsm}), with the exception of the visualisation of the automaton for $\sigma_{(1,9)}$, which was drawn in {\sc Mathematica}, exported as eps and the `Start'-label was added in {\sc Inkskape}.
\item The genus of the curves in Table \ref{tablebounds} were  computed using {\sc Singular}, with the exception of $\sigma_{(1,9)}$, which was computed in {\sc Magma}.
\item All claimed automata and explicit series representations were verified in {\sc Mathematica} to $O(t^{200})$ at least.
\end{itemize}
}

\section*{Description of supplementary material}
{\footnotesize
\begin{itemize}[leftmargin=*]
\item The file {\tt automata-of-finite-order} contains, for each of the series occurring in this paper, an irreducible algebraic equation that it satisfies, initial coefficients that uniquely determine it as a solution to that algebraic equation, and the corresponding automaton, stored in the format of \cite{Rowland} and visualised as a graph. The series occur by the name used in the current paper, and are ordered by compositional order, then by lexicographical order of the lower break sequence.
\item The file {\tt verification-of-non-sparseness} contains the material needed to verify combinatorially that $\sigma_{(1,9)} \notin \hat{S}$.
\item The file {\tt verification-of-synchronisation} contains the material needed to verify that $\sigma_{V,3}$, $\sigma_{(1,9)}$ and $\sigma_8$ are synchronising.
\item The file {\tt LabelledDirectedGraph.txt} in \cite{LDG} contains the \textsc{Magma}-routine to compute the labelled directed graph structure (without vertex output labels) from Algorithm \ref{algoLDG} using the method of differential forms, in a form that can be parsed by Rowland's \textsc{Mathematica} package \cite{Rowland}.
We give two examples of the running time using the online calculator for \textsc{Magma} V2.25-5: for $\sigma_{\mathrm{min}}$ the labelled directed graph is computed in  $0.090$ seconds, and
the computation of the number of states in Remark \ref{subsecCarlitz} being $668$ required $2.74$ seconds.

\end{itemize}} 

\addtocontents{toc}{\SkipTocEntry}
\bibliographystyle{amsplain}

\end{document}